\documentclass[11pt]{article}
\usepackage{pdfsync}
\usepackage{subfigure}
\usepackage{amssymb}
\usepackage{color}
\usepackage{amsmath,amsthm}
\usepackage{graphicx}
\usepackage{empheq}
\usepackage{indentfirst}
 \newtheorem{thm}{Theorem}[section]
 
 \newtheorem{lem}{Lemma}[section]
 
 \theoremstyle{definition}
 \newtheorem{defn}{Definition}[section]
 
 \numberwithin{equation}{section}
\def\f{\frac}
\def\pa{\partial}

\def\ov #1{\overline{#1}}
\def\i1n{i=1,\cdots,n}
\def\j1n{j=1,\cdots,n}
\def\ij1n{i,j=1,\cdots,n}

\newcommand{\be}{\begin{equation}}
\newcommand{\ee}{\end{equation}}
\newcommand{\beq}{\begin{equation*}}
\newcommand{\eeq}{\end{equation*}}

\title{Cauchy problem for multiscale conservation laws: \\Application to structured cell populations
\thanks{This work was supported by the large scale INRIA project REGATE (REgulation of
the GonAdoTropE axis).}}

\author{Peipei Shang\thanks{INRIA Paris-Rocquencourt Centre.
Universit\'{e} Pierre et Marie Curie-Paris 6.
UMR 7598 Laboratoire Jacques-Louis
Lions, 75005 Paris, France. E-mail: {\tt Peipei.Shang@inria.fr}.}}
\date{October 11, 2010}
\begin{document}

\maketitle

\begin{abstract}
In this paper, we study a vector conservation law that models the growth 
and selection of ovarian follicles. During each ovarian cycle, only
a definite number of follicles ovulate, while the others undergo a degeneration
process called atresia. This work is motivated by a multiscale
mathematical model starting on the cellular scale, 
where ovulation or atresia result from a hormonally controlled
selection process. A two-dimensional conservation law describes the age and maturity 
structuration of the follicular cell populations. The densities intersect through a coupled hyperbolic system between
different follicles and cell phases, which results in a vector conservation law and coupling boundary conditions.
The maturity velocity functions
possess both a local and nonlocal character. We prove the existence and uniqueness 
of the weak solution to the Cauchy problem with bounded initial and boundary data.
\end{abstract}


{\bf 2000 MR Subject Classification:}\quad
         35L65, 
         92D25, 
         92B05. 

\section{Introduction}
In this paper, we study the following vector conservation law
\be \label{eq}
{\vec{\phi}_{ft}(t,x,y)}+(A_f\vec{\phi}_f(t,x,y))_x +(B_f\vec{\phi}_f(t,x,y))_y= C\vec{\phi}_f(t,x,y),
\ee
\begin{equation*}
t\in[0,T],\quad(x,y)\in {[0,1]^2},
\end{equation*}
where $\vec\phi_f=(\ov\phi^f_1,\cdots,\ov\phi^f_N,\ \widehat\phi^f_1,\cdots,\widehat\phi^f_N,\ \widetilde\phi^f_1,\cdots,\widetilde\phi^f_N)^T$, 
$f=1,\cdots,n$, and
\begin{align*}
&A_f:=\mbox{diag}\ \{\overbrace{\ov g_f,\cdots,\ov g_f}^{N},\ \overbrace{\widehat g_f,\cdots,\widehat g_f}^{N},\
 \overbrace{\widetilde g_f,\cdots,\widetilde g_f}^{N}\},\nonumber\\
&B_f:=\mbox{diag}\ \{\overbrace{\ov h_f,\cdots,\ov h_f}^{N},\ \overbrace{0,\cdots,0}^{N},\
 \overbrace{\widetilde h_f,\cdots,\widetilde h_f}^{N}\},\nonumber\\
&C:=-\mbox{diag}\ \{\overbrace{\ov\lambda,\cdots,\ov\lambda}^{N},\ \overbrace{0,\cdots,0}^{N},\
 \overbrace{\widetilde\lambda,\cdots,\widetilde\lambda}^{N}\}.
\end{align*}

Here $\widehat g_f$ and $\widetilde g_f$ are positive constants, 
$\ov g_f=\ov g_f(u_f)\in C^1([0,\infty))$, $\ov h_f=\ov h_f(y,u_f)$, $\widetilde h_f=\widetilde h_f(y,u_f)$, $\ov\lambda=\ov\lambda(y,U)$ 
and $\widetilde\lambda=\widetilde\lambda(y,U)$ all belong to $C^1([0,1]\times[0,\infty))$, 
with $u_f=u_f(M_f(t),M(t),t)$ the local control and $U=U(M(t),t)$ the global control, where
$u_f>0$ and $U>0$ are continuous differentiable, i.e.,
$u_f\in C^1([0,\infty)^2\times[0,T])$ and $U\in C^1([0,\infty)\times[0,T])$.
For instance, the specific case that motivated our work is presented in Appendix 6.1.
The function $M_f$ is the maturity on the follicular scale given by 
\begin{align*} M_f(t):=& \sum_{k=1}^N\int_0^1\int_0^1 a_1\gamma_s^2y
\ov\phi^f_k(t,x,y)\,dx\,dy+\sum_{k=1}^N\int_0^1\int_0^1(a_2-a_1)\gamma^2_s y\widehat\phi^f_k(t,x,y)\,dx\,dy\\
&\!\!+\!\!\sum_{k=1}^N\int_0^1\int_0^1a_2\gamma_0(\gamma_0 y +\gamma_s)\widetilde\phi^f_k(t,x,y)\,dx\,dy 
\end{align*}
and
\begin{eqnarray*}
M(t):=\sum_{f=1}^{n} M_f(t)
\end{eqnarray*}
is the global maturity on the ovarian scale. The parameters $a_1, a_2, \gamma_s$ and $\gamma_0$ are positive constants with $a_2>a_1$.

This work is motivated by problems of cell dynamics arising in the control of the development of 
ovarian follicles. Ovarian follicles are spheroidal structures sheltering the maturing oocyte.
During each ovarian cycle, numerous follicles are in competition
for their survival. However, only very few follicles reach an ovulatory size, most of them
undergo a degeneration process, known as atresia (see for instance \cite{MH}).
A mathematical model, proposed by F. Cl\'{e}ment \cite{FC05} using both 
multi-scale modeling and control theory concepts, describes the follicle selection process.
For each follicle, the cell population dynamics is ruled by a conservation law, which describes the changes in cell age and maturity.
Two acting controls, $u_f$ and $U$ (see \cite{F08} and also Appendix 6.1), are distinguished.
The global control $U$ results from the ovarian feedback onto the 
pituitary gland and impacts the secretion of the follicle stimulating hormone (FSH). The feedback is responsible for reducing 
FSH release, leading to the degeneration of all but those follicles 
selected for ovulation.
The local control $u_f$, specific to each follicle, accounts for the modulation in FSH 
bioavailability related to follicular vascularization. In addition, the status of cells are characterized by three
phases. Phase 1 and 2 correspond to the proliferation phases, and Phase 3 corresponds to the
differentiation phase (see Fig 1). In Appendix 6.2, we have reformulated the original model to a new system (\ref{eq}), where the unknowns are
all defined on the same domain $[0,T]\times [0,1]^2$, so that it can be treated as a general
model for multiscale structured cell populations 
(see Appendix 6.2 for the relation between the original notations and the new notations).

The initial conditions are given by
\begin{align}\label{ic}
\vec{\phi}_{f0}:&=\!\!\vec{\phi}_f(0,x,y)\\ 
&=\!\!(\ov\phi^f_{10}(x,y),\!\cdots\!,\ov\phi^f_{N0}(x,y),\ \widehat\phi^f_{10}(x,y),\!\cdots\!,\widehat\phi^f_{N0}(x,y),\
\widetilde\phi^f_{10}(x,y),\!\cdots\!,\widetilde\phi^f_{N0}(x,y))^T\nonumber.
\end{align}

We use the simplified notations $u_f(t):=u_f(M_f(t),M(t),t)$, $U(t):=U(M(t),t)$, $\ov u_f(t):=u_f(\ov M_f(t),\ov M(t),t)$
and $\ov U(t):=U(\ov M(t),t)$ in the whole paper.

The boundary conditions at $x=0$ are given by
\be \label{bcx} 
\left\{
\begin{array}{lc}
\overline\phi^f_1(t,0,y)=0,\ (t,y)\in [0,T]\times[0,1],\\
\overline\phi^f_k(t,0,y)=\displaystyle\f{2\tau_{gf}}{a_1\overline g_f(u_f(t))}\ \widehat\phi^f_{k-1}(t,1,y),(t,y)\in [0,T]\times [0,1],
k=2,\cdots,N. \\
\widehat\phi^f_k(t,0,y)=\displaystyle\f{a_1\ov g_f(u_f(t))}{\tau_{gf}}\ \overline \phi^f_k(t,1,y),(t,y)\in[0,T]\times[0,1],k=1,\cdots,N.\\
\widetilde\phi^f_1(t,0,y)=0, \ (t,y)\in [0,T]\times[0,1],\\
\widetilde\phi^f_k(t,0,y)=\widetilde\phi^f_{k-1}(t,1,y),
  \ (t,y)\in [0,T]\times[0,1],\ k=2,\cdots, N.
\end{array}
  \right.
 \ee

The boundary conditions at $y=0$ are given by
\be \label{bcy}
\left\{
\begin{array}{l}
\ov\phi^f_k(t,x,0)=\widehat\phi^f_k(t,x,0)=0,\quad (t,x)\in [0,T]\times [0,1],\quad k=1,\cdots,N.\\
\widetilde\phi^f_k(t,x,0)=
  \left\{
  \begin{array}{l}
   \ov\phi^f_k(t,\displaystyle\f{a_2}{a_1}x,1),
   \quad   (t,x)\in [0,T]\times [0,\displaystyle\f{a_1}{a_2}],\\
   0,\quad  (t,x)\in [0,T]\times [\displaystyle\f{a_1}{a_2},1],
  \end{array} 
  \right.
  k=1,\cdots,N.
  \end{array}
  \right.
 \ee

The boundary conditions at $y=1$ are given by
\be \label{bcy1}
\widetilde\phi^f_k(t,x,1)=0,\quad (t,x)\in [0,T]\times [0,1],\quad k=1,\cdots, N.
\ee

\vspace{10mm}
\setlength{\unitlength}{0.085in}
\begin{figure}[htbp!]
\begin{picture}(-10,10)
\put(7,1){\vector(1,0){50}} \put(7,1){\vector(0,1){12}}
\put(7,10.9){\line(1,0){48}} \put(7,7){\line(1,0){48}}
\put(6,1){\line(0,1){5.9}} \put(5,7){\line(0,1){3.9}}
\put(12,1){\line(0,1){5.9}} \put(17,1){\line(0,1){9.8}}
\put(22,1){\line(0,1){5.9}} \put(27,1){\line(0,1){9.8}}
\put(31,1){\line(0,1){9.8}}
\put(28,3){\makebox(2,1)[l]{$\cdots$}}
\put(28,8){\makebox(2,1)[l]{$\cdots$}}
\put(36,1){\line(0,1){5.9}}
\put(41,1){\line(0,1){9.8}} \put(45,1){\line(0,1){9.8}}
\put(50,1){\line(0,1){5.9}} \put(55,1){\line(0,1){9.8}}
\put(6,-1){\makebox(2,1)[l]{0}}
\put(5,6){\makebox(2,1)[l]{1}}
\put(4,6){\makebox(2,1)[l]{0}}
\put(4,10){\makebox(2,1)[l]{1}}
\put(6.5,14){\makebox(2,1)[l]{$y$}}
\put(58,0.5){\makebox(2,1)[l]{$x$}}

\put(7,-0.01){\line(1,0){5}}
\put(12,-0.7){\makebox(2,1)[l]{1}}
\put(11,-2){\makebox(2,1)[l]{0}}
\put(12,-1){\line(1,0){5}}
\put(17,-1.7){\makebox(2,1)[l]{$1$}}
\put(16,-2.5){\makebox(2,1)[l]{0}}
\put(17,-2){\line(1,0){5}}
\put(22,-2.5){\makebox(2,1)[l]{$1$}}
\put(21,-3.5){\makebox(2,1)[l]{0}}
\put(22,-3){\line(1,0){5}}
\put(27,-3.5){\makebox(2,1)[l]{$1$}}

\put(30,-0.5){\makebox(2,1)[l]{0}}
\put(31,-0.01){\line(1,0){5}}
\put(36,-0.5){\makebox(2,1)[l]{$1$}}
\put(35,-1.5){\makebox(2,1)[l]{0}}
\put(36,-1){\line(1,0){5}}
\put(41,-1.5){\makebox(2,1)[l]{$1$}}

\put(44,-0.5){\makebox(2,1)[l]{0}}
\put(45,-0.01){\line(1,0){5}}
\put(50,-0.5){\makebox(2,1)[l]{$1$}}
\put(49,-1.7){\makebox(2,1)[l]{0}}
\put(50,-1){\line(1,0){5}}
\put(55,-1.7){\makebox(2,1)[l]{$1$}}

\put(28,-1){\makebox(2,1)[l]{$\cdots$}}
\put(42,-1){\makebox(2,1)[l]{$\cdots$}}
\put(8,3){\makebox(2,1)[l]{ $\ov\phi^f_1$}}
\put(14,3){\makebox(2,1)[l]{$\widehat\phi^f_1$}}
\put(11,8){\makebox(2,1)[l]{$\widetilde\phi^f_1$}}
\put(18,3){\makebox(2,1)[l]{ $\ov\phi^f_2$}}
\put(24,3){\makebox(2,1)[l]{$\widehat\phi^f_2$}}
\put(21,8){\makebox(2,1)[l]{$\widetilde\phi^f_2$}}
\put(32,3){\makebox(2,1)[l]{ $\ov\phi^f_k$}}
\put(38,3){\makebox(2,1)[l]{$\widehat\phi^f_k$}}
\put(35,8){\makebox(2,1)[l]{$\widetilde\phi^f_k$}}
\put(46,3){\makebox(2,1)[l]{ $\ov\phi^f_N$}}
\put(52,3){\makebox(2,1)[l]{$\widehat\phi^f_N$}}
\put(49,8){\makebox(2,1)[l]{$\widetilde\phi^f_N$}}
\put(42,3){\makebox(2,1)[l]{$\cdots$}}
\put(42,8){\makebox(2,1)[l]{$\cdots$}}
\put(16,7.5){\vector(1,1){1.5}} \put(16,11){\vector(1,-1){1.5}} 
\put(26,7.5){\vector(1,1){1.5}} \put(26,11){\vector(1,-1){1.5}}  
\put(30,7.5){\vector(1,1){1.5}} \put(30,11){\vector(1,-1){1.5}} 
\put(40,7.5){\vector(1,1){1.5}} \put(40,11){\vector(1,-1){1.5}}
\put(44,7.5){\vector(1,1){1.5}} \put(44,11){\vector(1,-1){1.5}}
\put(9,6){\vector(1,1){1.5}} 
\put(11,4){\vector(1,1){1.5}}
\put(16,4){\vector(1,0){1.5}}
\put(19,6){\vector(1,1){1.5}} 
\put(21,4){\vector(1,1){1.5}}
\put(26,4){\vector(1,0){1.5}}
\put(33,6){\vector(1,1){1.5}} 
\put(35,4){\vector(1,1){1.5}}
\put(40,4){\vector(1,0){1.5}}
\put(47,6){\vector(1,1){1.5}} 
\put(49,4){\vector(1,1){1.5}}
\end{picture}
\vspace{5mm}
\caption{The illustration of $N$ cell cycles for follicle $f$; $x$ denotes the age velocity and $y$ denotes the 
maturity velocity. The top of the domain corresponds to the differentiation phase and the bottom to the the proliferation phase}.
\label{Fig 1}
\end{figure}
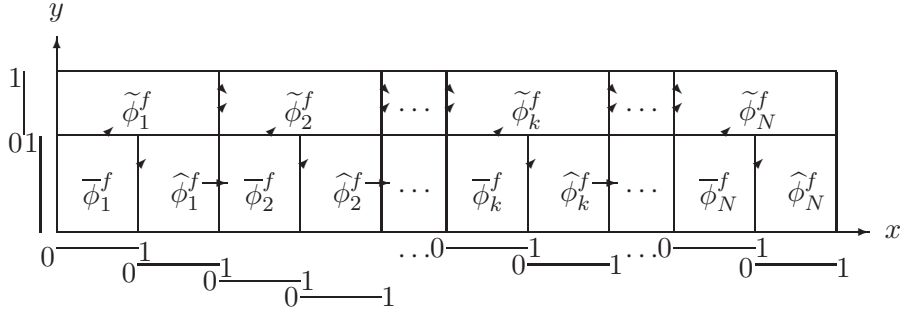
The well-posedness problems of the hyperbolic conservation laws have been widely studied
for a long time. We refer to the works
\cite{BBressan, BressanBook, Kruzkov, LiuYang, Poretta} (and the references therein) in the content of weak solutions to systems 
in conservation laws, and \cite{LiBook94, LiYuBook} in the content of classical solutions to general
quasi-linear hyperbolic systems.
In this paper, we perform the mathematical analysis of this system:
we prove the existence, uniqueness and regularity of the weak solution to the
Cauchy problem defined by (\ref{eq})-(\ref{bcy1}) with initial
and boundary data in $L^{\infty}$.
The main difficulty tackled with in this paper comes from the nonlocal velocity, 
the coupling between boundary conditions and the coupling 
between different follicles in the model. 
Additionally, we have to deal with loss terms, and the problem is a 2 - space dimension one.
 
In the related works that have considered nonlocal velocity problems \cite{Wang, SW}, 
the problems are only
1 - space dimension and do not have source terms. The velocities there are always positive,
while in our case, the velocities $\widetilde h_f(f=1,\cdots,n)$ we considered in this paper change sign in Phase 3. 
In another related work \cite{Michel}, which was also motivated by \cite{FC05},
the author has performed a mathematical analysis on this kind of model.
He has reduced the model to a 1-space dimension mass-maturity dynamical system of coupled
Ordinary Differential Equations,
basing on the asymptotic properties of the original law.
Once reduced, the model is amenable to analysis by bifurcation theory,
that allows one to predict the issue of the selection for one specific follicle 
amongst the whole population.
And he gave some assumptions on velocities according to biological observations.
In another work \cite{Michel1}, the author also studied the well-poseness of the model, he 
proved the existence of a measure valued solution without proving the uniqueness of the solution,
and he gave the behavior
of the solution to the main equation along its characteristics.
An associated reachability problem has been tackled in \cite{FC07}, which described the set of microscopic initial
conditions leading to the macroscopic phenomenon of either ovulation or atresia
in the framework of backwards reachable sets theory. The authors also performed some mathematical analysis on
well-poseness of this model in a simplified open loop like case, there the authors assumed that the local 
control $u_f$ and the global control $U$ are given functions of time $t$. While in this paper, we consider the close/open loop case.

This paper is organized as follows. In Section 2 
we first give the main results. 
In Section 3 after giving some basic notations,
we prove that the vector maturity $\vec M:=(M_1,\cdots,M_{n})$ exists as a fixed point of a map from
continuous function space, and then we construct a local weak solution to the Cauchy problem defined by (\ref{eq})-(\ref{bcy1}).
In Section 4 we prove the uniqueness of the weak solution.
Finally in Section 5 we prove the existence of a global solution.
In complement, we introduce in Appendix 6.1 the original mathematical model proposed 
by F. Cl\'{e}ment \cite{FC05}. In Appendix 6.2 we reformulate this model to the
new system (\ref{eq}) and we give the equivalence between the original notations and the new notations.
In Appendix 6.3 we give a basic lemma that is used to prove the existence and uniqueness of the weak solution.

\section{Main results} 
 
We recall from \cite[Section 2.1]{CoronBook}, the usual definition of a weak solution to the Cauchy problem defined by
(\ref{eq})-(\ref{bcy1}).
\begin{defn}
\label{weaksol} Let $T>0$, $\vec{\phi}_{f0}\in L^{\infty}((0,1)^2)$ 
be given. A weak solution of Cauchy problem
(\ref{eq})-(\ref{bcy1}) is a vector function $\vec{\phi}_f\in
C^0([0,T];L^1((0,1)^2))$ such that
for every $\tau\in[0,T]$ and every vector function $\vec{\varphi}:=\!\!(\ov\varphi_1,\!\cdots\!,\ov\varphi_N,\widehat\varphi_1,\!\cdots\!,\widehat\varphi_N,
\widetilde\varphi_1,\!\cdots\!,\widetilde\varphi_N)^T\!\!\!\in
C^1([0,\tau]\times[0,1]^2)$ with
 \begin{align}\label{varphi}
 &\vec{\varphi}(\tau,x,y)= 0,\quad \forall (x,y)\in[0,1]^2,\\
 &\vec{\varphi}(t,1,y)= 0,\quad \forall (t,y)\in[0,\tau]\times[0,1],\\
 &\ov\varphi_1(t,0,y)=\widetilde\varphi_1(t,0,y)=0,\quad \forall (t,y)\in[0,\tau]\times[0,1],\\
 &\label{varphi0}\ov\varphi_k(t,x,0)=\widehat\varphi_k(t,x,0)=\ov\varphi_k(t,x,1)=\widehat\varphi_k(t,x,1)=0,\quad \forall (t,x)\in[0,\tau]\times[0,1],
 \end{align}
one has
\begin{align}\label{intequality}
 &\int_0^{\tau}\!\!\int_0^1\!\!\int_0^1 \vec\phi_f(t,x,y)\cdot(\vec{\varphi}_t(t,x,y)
 \!+\!A_f\vec{\varphi}_x(t,x,y)\!+\!B_f\vec{\varphi}_y(t,x,y)\!+\!C\vec{\varphi}(t,x,y)) dx dy dt
  \\
  &+\!\!\int_0^1\!\!\int_0^1 \vec\phi_{f0}(x,y)\cdot\vec{\varphi}(0,x,y)dxdy\!+\!\!\sum_{k=1}^N\int_0^{\tau}\!\!\int_0^{\f{a_1}{a_2}}\widetilde h_f(0,u_f(t))\ov\phi^f_k(t,\f{a_2}{a_1}x,1)\widetilde\varphi_k(t,x,0)dxdt\nonumber \\
  &+\!\!\sum_{k=2}^N\int_0^{\tau}\!\!\int_0^1 \f{2\tau_{gf}}{a_1}\ \widehat\phi^f_{k-1}(t,1,y)\ov\varphi_k(t,0,y)dydt\!+\!\!
  \sum_{k=2}^N\int_0^{\tau}\!\!\int_0^1 \widetilde g_f\widetilde\phi^f_{k-1}(t,1,y)\widetilde\varphi_k(t,0,y)dydt\nonumber \\
   &+\!\!\sum_{k=1}^N \int_0^{\tau}\!\!\int_0^1 \f{a_1\widehat g_f\ov g_f(u_f(t))}{\tau_{gf}}\ \ov\phi^f_k(t,1,y)\widehat\varphi_k(t,0,y)dydt=0.
   \nonumber 
 \end{align}
\end{defn}

With the definition, we have the main result
\begin{thm}\label{thm-sol}
Let $T>0$, $\vec\phi_{f0}\in L^{\infty}((0,1)^2)$ be given. Let us further assume that
\begin{align*}
&\ov g_f(u_f)>0,\quad \forall\  u_f\in [0,\infty), \\
&\ov h_f(y,u_f)>0,\quad \forall\  (y,u_f)\in [0,1]\times[0,\infty), \\
&\widetilde h_f(0,u_f)>0,\ \widetilde h_f(1,u_f)<0,\quad \forall\  u_f\in [0,\infty).
\end{align*} 

Then the Cauchy problem defined by
(\ref{eq})-(\ref{bcy1}) admits a unique weak solution $\vec\phi_f=(\ov\phi^f_1,\cdots,\ov\phi^f_N,\ \widehat\phi^f_1,\cdots,\widehat\phi^f_N,\ 
\widetilde\phi^f_1,\cdots,\widetilde\phi^f_N)^T$.
Moreover, the weak solution $\vec\phi_f$ even belongs to
$C^0([0,T];L^p((0,1)^2))$ for all $p\in [1,\infty)$.
\end{thm}

The sketch of the proof of Theorem \ref{thm-sol} consists in
first proving that the maturity $\vec M(t)=(M_1(t),\cdots, M_{n}(t))$ exists as a fixed point of the
map $\vec M\mapsto \vec G(\vec M)$ (see Section 3.1), and then in constructing a (unique) local solution (see Section 3.2 and Section 4),
before finally proving the existence of a global solution to the Cauchy problem defined by (\ref{eq})-(\ref{bcy1}) (see Section 5).
\section{Fixed point argument and construction of a local solution to the Cauchy problem}
In this section, we first derive the contraction mapping function $\vec G$. 
Given the existence of fixed point to this contraction mapping function, we can
then construct a local solution to the Cauchy problem defined by (\ref{eq})-(\ref{bcy1}).
\subsection{Fixed point argument}
First we introduce some notations.
Let us define:
\begin{align*}
&\|\ov\phi^f_{k0}\|:=\|\ov\phi^f_{k0}\|_{L^{\infty}((0,1)^2)}:=ess\,sup_{(x,y)\in[0,1]^2}|\ov\phi^f_{k0}(x,y)|,\\
&\|\widehat\phi^f_{k0}\|:=\|\widehat\phi^f_{k0}\|_{L^{\infty}((0,1)^2)}:=ess\,sup_{(x,y)\in[0,1]^2}|\widehat\phi^f_{k0}(x,y)|,\\
&\|\widetilde\phi^f_{k0}\|:=\|\widetilde\phi^f_{k0}\|_{L^{\infty}((0,1)^2)}:=ess\,sup_{(x,y)\in[0,1]^2}|\widetilde\phi^f_{k0}(x,y)|.
\end{align*}
\begin{align}\label{K}
&K:=\!\!2^N(\gamma_0+\gamma_s)^2 \sum_{f=1}^{n}\!\sum_{k=1}^N\Big(a_1\|\ov\phi^f_{k0}\|_{L^1((0,1)^2)}
\!\!+\!\!(a_2-a_1)\|\widehat\phi^f_{k0}\|_{L^1((0,1)^2)}
\!\!+\!\!a_2\|\widetilde\phi^f_{k0}\|_{L^1((0,1)^2)}\Big),\\
\label{K1}&K_1:=\max\Big\{\widehat g_f,\|\ov g_f(u_f(M_f,M,t))\|_{C^1(Q_1)},\|\ov h_f(y,u_f(M_f,M,t))\|_{C^1(Q_2)},\\
&\qquad\|\widetilde h_f(y,u_f(M_f,M,t))\|_{C^1(Q_2)},
\|\ov\lambda(y,U(M,t))\|_{C^1(Q_3)},\|\widetilde\lambda(y,U(M,t))\|_{C^1(Q_3)}\Big\}\nonumber,\\
\label{K2}
&K_2:=\min\Big\{\inf_{(M_f,M,t)\in Q_1}\ov g_f(u_f(M_f,M,t)),\inf_{(y,M_f,M,t) \in Q_2}\ov h_f(y,u_f(M_f,M,t))\Big\}>0,\\
&Q_1:=[0,K]^2\times[0,T],\ Q_2:=[0,1]\times[0,K]^2\times[0,T],\ Q_3:=[0,1]\times[0,K]\times[0,T].\nonumber
\end{align}
For any $\delta>0$, let
 \begin{equation*}
 \Omega_{\delta,K}\!\!
  :=\!\!\Big\{\!\vec M(t)\!\!=\!\!(M_1(t),\!\cdots\!, M_{n}(t))\!\in\! C^0([0,\delta]) \colon\!
  \|\vec M\|_{C^0([0,\delta])}\!\!:=\!\!\max_f \|M_f\|_{C^0([0,\delta])}\!\!\leq\!\! K\!
  \Big\},
 \end{equation*}
where the constant $K$ is given by (\ref{K}). 

In order to derive the expression of the contraction mapping function $\vec G$, 
we solve the corresponding \emph{linear Cauchy problem}
(\ref{eq})-(\ref{bcy1}) with given $\vec M\in\Omega_{\delta,K}$.
For any fixed $t\in[0,\delta]$ and $f\in\{1,\cdots,n\}$, we  
trace back the density function $\vec\phi_f$ at time $t$ along the characteristics to the initial and boundary data,
hence we divide the plane time $t$ into several parts. For Phase 1, the velocities $\ov h_f$ are always
positive, we introduce three subsets $\ov\omega^{f,t}_1$, $\ov\omega^{f,t}_2$ and $\ov\omega^{f,t}_3$ of $[0,1]^2$
(see Fig 2 (a))
\begin{align*}
&\ov\omega^{f,t}_1:=\Big\{(x,y)|\ \int_0^t\ov g_f(u_f(\sigma))\, d\sigma\leq x\leq 1,\ \ov\eta(t,\int_0^t \ov g_f(u_f(\sigma))\,d\sigma)\leq y\leq 1\Big\},\\
&\ov\omega^{f,t}_2:=\Big\{(x,y)|\ 0\leq x\leq\int_0^t\ov g_f(u_f(\sigma))\,d\sigma,\ \ov\eta(t,x)\leq y\leq 1\Big\},\\
&\ov\omega^{f,t}_3:=[0,1]^2\backslash(\ov\omega^{f,t}_1\cup\ov\omega^{f,t}_2).
\end{align*}
Here $y=\ov\eta(t,x)$ (see Fig 2 (a))satisfies 
\be \label{oveta}
\displaystyle\f{d\ov\eta}{ds}=\ov h_f(\ov\eta, u_f)(s),\quad \ov\eta(\theta)=0,\quad \theta\leq s\leq t,
\ee
with $\theta$ defined by $x=\displaystyle\int_{\theta}^t\ov g_f(u_f(\sigma))d\sigma$.

If $(x,y)\in \ov\omega^{f,t}_1$, we trace back the density function $\ov\phi^f_k$ at time $t$ 
along the characteristics to the initial data.
Otherwise, we trace back the density function at time $t$ along the characteristics to the boundary data.
For any fixed $t\in[0,\delta]$ and $(x,y)\in[0,1]^2$, let us define characteristics $\xi_i(s):=(x_i(s),y_i(s)),
i=1,\cdots,4$ (see Fig 2), which will be used to construct the contraction mapping function.

In the whole paper, we denote by $(0,x_0,y_0)$ the points on the bottom face $(t=0)$, $(t_0,\alpha_0,0)$ the
points on the left face $(y=0)$, $(t_0,\alpha_0,1)$ the
points on the right face $(y=1)$, $(\tau_0,0,\beta_0)$ the points on the back face
$(x=0)$ and $(\tau_0,1,\beta_0)$ the points on the front face $(x=1)$.
\begin{figure}[htbp!]
\subfigure[]{
\label{a} 
\begin{minipage}[b]{0.45\textwidth}
\centering
\includegraphics[width=2in]{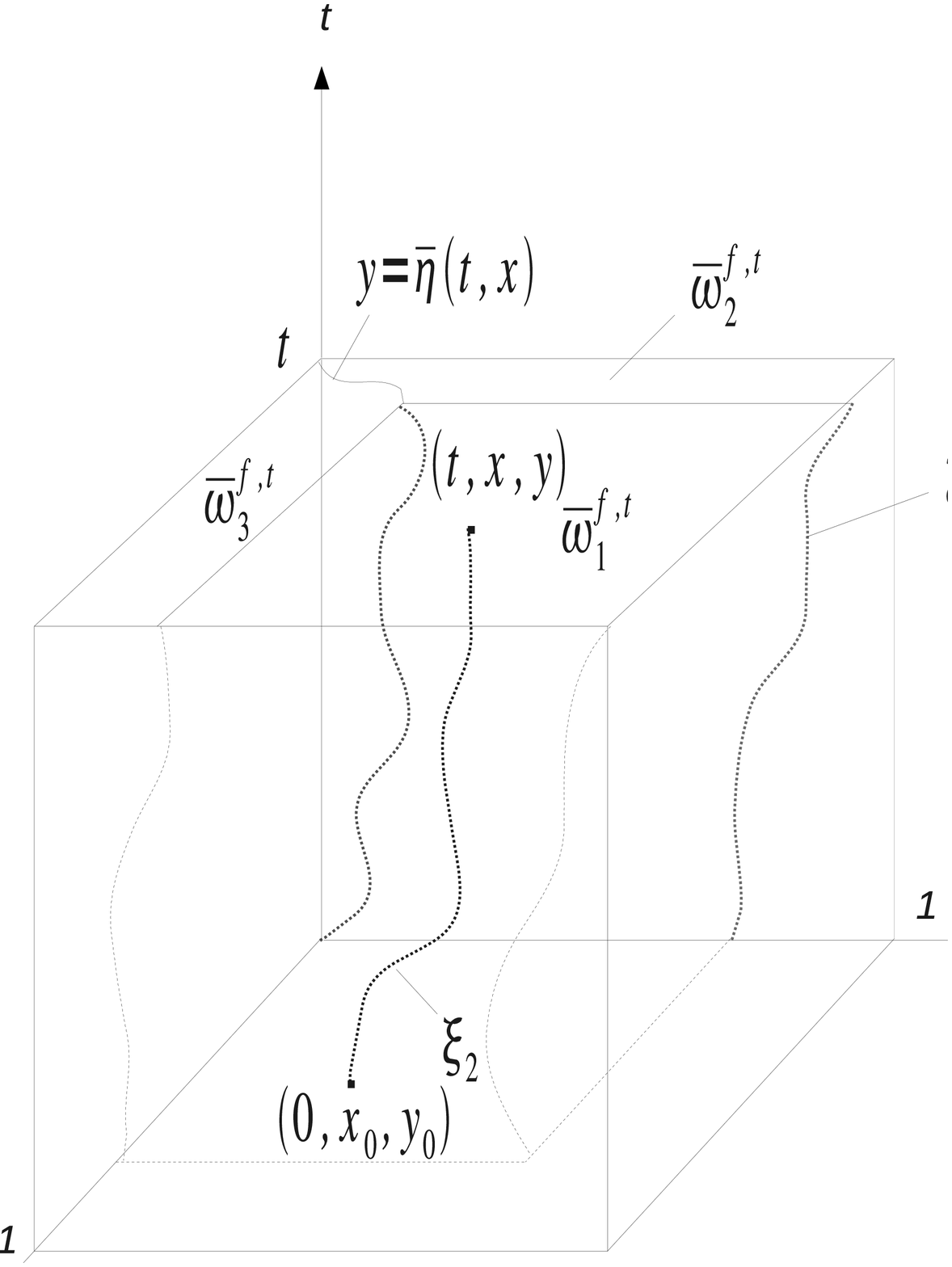}
\end{minipage}}%
\subfigure[]{
\label{b} 
\begin{minipage}[b]{0.45\textwidth}
\centering
\includegraphics[width=2in]{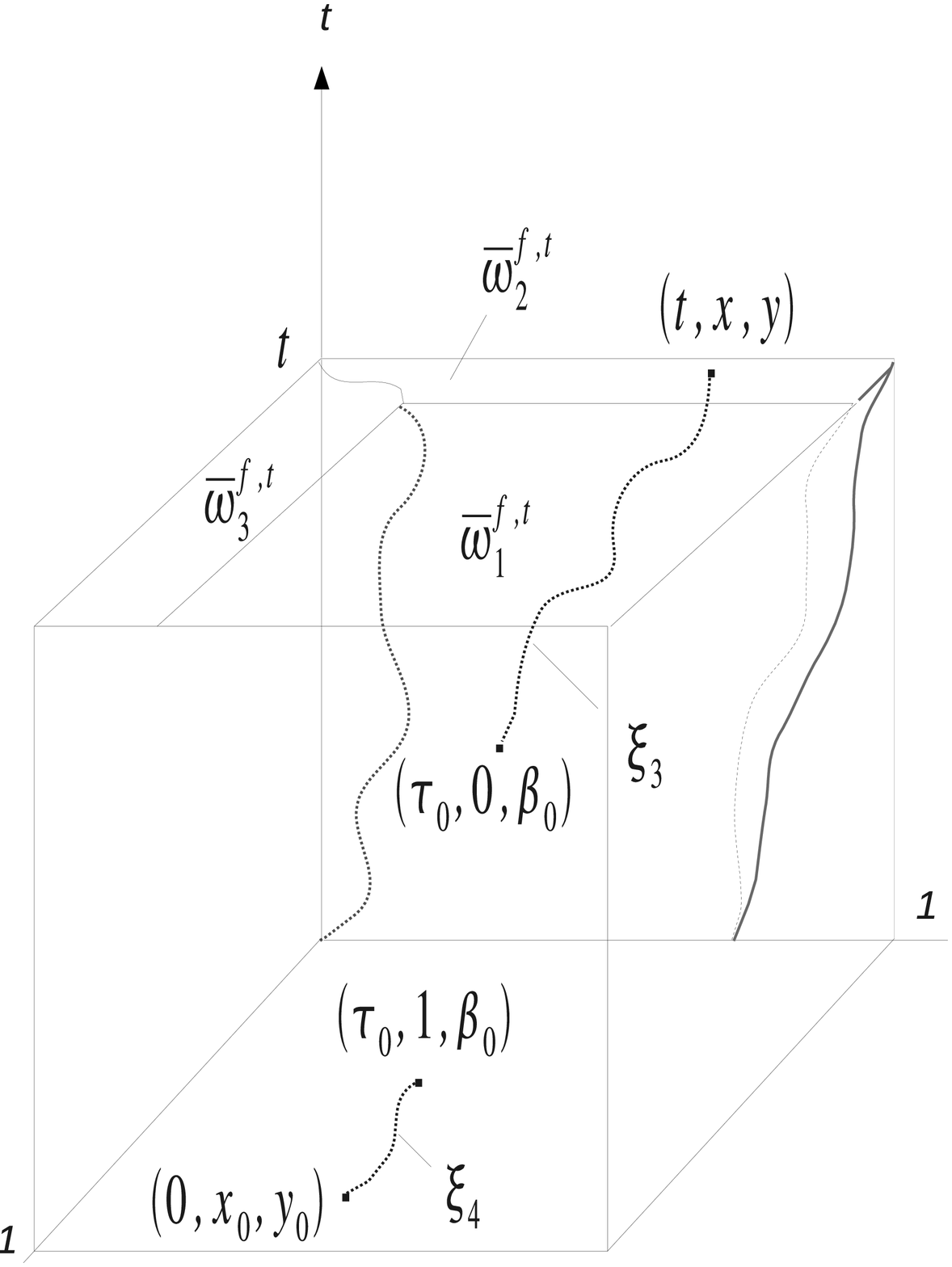}
\end{minipage}}%
\caption{For Phase 1, time $t$ plane is divided into three parts $\ov\omega^{f,t}_1$, $\ov\omega^{f,t}_2$ and $\ov\omega^{f,t}_3$. (a) Case $(t,x,y)\in\ov\omega^{f,t}_1$, characteristic $\xi_2$ connects $(t,x,y)$ with $(0,x_0,y_0)$; (b) Case $(t,x,y)\in\ov\omega^{f,t}_2$,
characteristic $\xi_3$ connects $(t,x,y)$ with $(\tau_0,0,\beta_0)$; According to the coupled boundary between Phase 1 at $x=1$ and Phase 2
at $x=0$, we define characteristic $\xi_4$ which connects $(\tau_0,1,\beta_0)$ with $(0,x_0,y_0)$.}
\label{Fig 3.1} 
\end{figure}

For any fixed $x\in[0,1]$, we define characteristic $\xi_1=(x_1,y_1)$ by (see Fig 2 (a))
\begin{equation*}
\displaystyle\f{dx_1}{ds}=\ov g_f(u_f(s)),\quad 
\displaystyle\f{dy_1}{ds}=\ov h_f(y_1,u_f)(s),\quad \xi_1(t)=(x,1).
\end{equation*}
If $(x,y)\in\ov\omega^{f,t}_1$ (see Fig 2 (a)),
we define characteristic $\xi_2=(x_2,y_2)$ by
\begin{equation*}
\displaystyle\f{dx_2}{ds}=\ov g_f(u_f(s)),\quad 
\displaystyle\f{dy_2}{ds}=\ov h_f(y_2,u_f)(s),\quad \xi_2(t)=(x,y).
\end{equation*}
One has $\xi_2(s)\in[0,1]^2,\ \forall s\in[0,t]$. Let us define
\be\label{x2}
(x_0,y_0):=(x_2(0),y_2(0)).
\ee
If $(x,y)\in\ov\omega^{f,t}_2$ (see Fig 2 (b)),
we define characteristic $\xi_3=(x_3,y_3)$ by
\begin{equation*}
\displaystyle\f{dx_3}{ds}=\ov g_f(u_f(s)),\quad 
\displaystyle\f{dy_3}{ds}=\ov h_f(y_3,u_f)(s),\quad \xi_3(t)=(x,y).
\end{equation*}
There exists a unique $\tau_0\in[0,t]$ such that
$x_3(\tau_0)=0$, so that we can define 
\be\label{x3}
\beta_0:=y_3(\tau_0).
\ee
According to the coupled boundary between Phase 1 at $x=1$ and Phase 2 at $x=0$, 
for any fixed $(x_0,y_0)\in[0,1]^2$, we define characteristic $\xi_4\!\!=\!\!(x_4,y_4)$ by (see Fig 2 (b))
\begin{equation*}
\displaystyle\f{dx_4}{ds}=\ov g_f(u_f(s)),\quad 
\displaystyle\f{dy_4}{ds}=\ov h_f(y_4,u_f)(s),\quad \xi_4(0)=(x_0,y_0).
\end{equation*}

For Phase 3, due to
the fact that velocities $\widetilde h_f$ change sign in Phase 3, we divide the plane at time $t$ into four subsets $\widetilde\omega^{f,t}_1$, $\widetilde\omega^{f,t}_2$, $\widetilde\omega^{f,t}_3$ and $\widetilde\omega^{f,t}_4$ of $[0,1]^2$ (see Fig 3 (a))
\begin{align*}
&\widetilde\omega^{f,t}_1:=\!\!\Big\{\!(x,y)|\widetilde g_ft\leq x\leq 1,\ \widetilde\eta_1(t,\widetilde g_ft)\leq y\leq \widetilde\eta_2(t,\widetilde g_ft)\Big\},\\
&\widetilde\omega^{f,t}_2:=\!\!\Big\{\!(x,y)|0\leq x\leq \widetilde g_ft,\ 0\leq y\leq \widetilde\eta_1(t,x)\Big\}
\!\cup\!\Big\{(x,y)|\ \widetilde g_ft\leq x\leq 1,0\leq y\leq \widetilde\eta_1(t,\widetilde g_ft)\!\Big\},\\
&\widetilde\omega^{f,t}_3:=\!\!\Big\{(x,y)|0\leq x\leq \widetilde g_ft,\ \widetilde\eta_1(t,x)\leq y\leq \widetilde\eta_2(t,x)\Big\},\\
&\widetilde\omega^{f,t}_4:=\![0,1]^2\backslash(\widetilde\omega^{f,t}_1\cup\widetilde\omega^{f,t}_2\cup\widetilde\omega^{f,t}_3).
\end{align*} 
Here 
$y=\widetilde\eta_1(t,x)$ and $y=\widetilde\eta_2(t,x)$ (see Fig 3 (a)) satisfy  
\begin{align*}
&\displaystyle\f{d\widetilde\eta_1}{ds}=\widetilde h_f(\widetilde\eta_1, u_f)(s),\quad \widetilde\eta_1(t-\f{x}{\widetilde g_f})=0,
\quad t-\f{x}{\widetilde g_f}\leq s\leq t,\\
&\f{d\widetilde\eta_2}{ds}=\widetilde h_f(\widetilde\eta_2, u_f)(s),\quad \widetilde\eta_2(t-\f{x}{\widetilde g_f})=1,
\quad t-\f{x}{\widetilde g_f}\leq s\leq t.
\end{align*}
If $(x,y)\in \widetilde\omega^{f,t}_1$, we trace back the density function $\widetilde\phi^f_k$ at time $t$ 
along the characteristics to the initial data.
Otherwise, we trace back the density function at time $t$ along the characteristics to the boundary data.
For any fixed $t\in[0,\delta]$ and $(x,y)\in[0,1]^2$, let us define characteristics $\xi_i(s):=(x_i(s),y_i(s)), i=5,\cdots,10$ (see Fig 3),
which will be used to construct the contraction mapping function.

If $(x,y)\in\widetilde\omega^{f,t}_1$ (see Fig 3 (a)),
we define characteristic
$\xi_5=(x_5,y_5)$ by
\begin{equation*}
\displaystyle\f{dx_5}{ds}=\widetilde g_f,\quad 
\displaystyle\f{dy_5}{ds}=\widetilde h_f(y_5,u_f)(s),\quad \xi_5(t)=(x,y).
\end{equation*}
One has $\xi_5(s)\in[0,1]^2,\ \forall s\in[0,t]$. Let us define 
\be \label{x5}
(x_0,y_0):=(x_5(0),y_5(0)).
\ee
If $(x,y)\in\widetilde\omega^{f,t}_2$ (see Fig 3 (a)), 
we define characteristic $\xi_6=(x_6,y_6)$ by
\begin{equation*}
\displaystyle\f{dx_6}{ds}=\widetilde g_f,\quad
\displaystyle\f{dy_6}{ds}=\widetilde h_f(y_6,u_f)(s),\quad \xi_6(t)=(x,y).
\end{equation*}
There exists a unique $t_0\in[0,t]$ such that $y_6(t_0)=0$,
so that we can define 
\be\label{x6}
\alpha_0:=x_6(t_0).
\ee

For any fixed cell cycle $k=1,\cdots,N$, according to the coupled boundary between Phase 1 and Phase 3 
(corresponding to $\ov\phi^f_k(t,x,y)$ at $y=1$ and $\widetilde\phi^f_k(t,x,y)$ at $y=0$), 
we separate the right face of Phase 1 into two parts.
For $k=1,\cdots,N$, we define characteristic $\xi_7$ passing through 
the right face intersects with the bottom face at $(0,x_0,y_0)$ (see Fig 3 (b)). 
For $k=2,\cdots,N$, we define characteristic $\xi_8$ passing through the right face intersects with the back face, and then
back to the bottom face at $(0,x_0,y_0)$ of the $k-1$ cell cycle in Phase 2 (see Fig 3 (b)).
Hence, for any fixed $(x_0,y_0)\in[0,1]^2$, we define $\xi_7=(x_7,y_7)$ by 
\begin{equation*}
\displaystyle\f{dx_7}{ds}=\ov g_f(u_f(s)),\quad
\displaystyle\f{dy_7}{ds}=\ov h_f(y_7,u_f)(s),\quad \xi_7(0)=(x_0,y_0),
\end{equation*}
 and $\xi_8=(x_8,y_8)$ by
\begin{equation*}
\displaystyle\f{dx_8}{ds}=\ov g_f(u_f(s)),\quad \displaystyle\f{dy_8}{ds}=\ov h_f(y_8,u_f)(s),\quad \xi_8(\f{1-x_0}{\widehat g_f})=(0,y_0).
\end{equation*}
\begin{figure}[htbp!]
\subfigure[]{
\label{c} 
\begin{minipage}[b]{0.32\textwidth}
\centering
\includegraphics[width=1.6in]{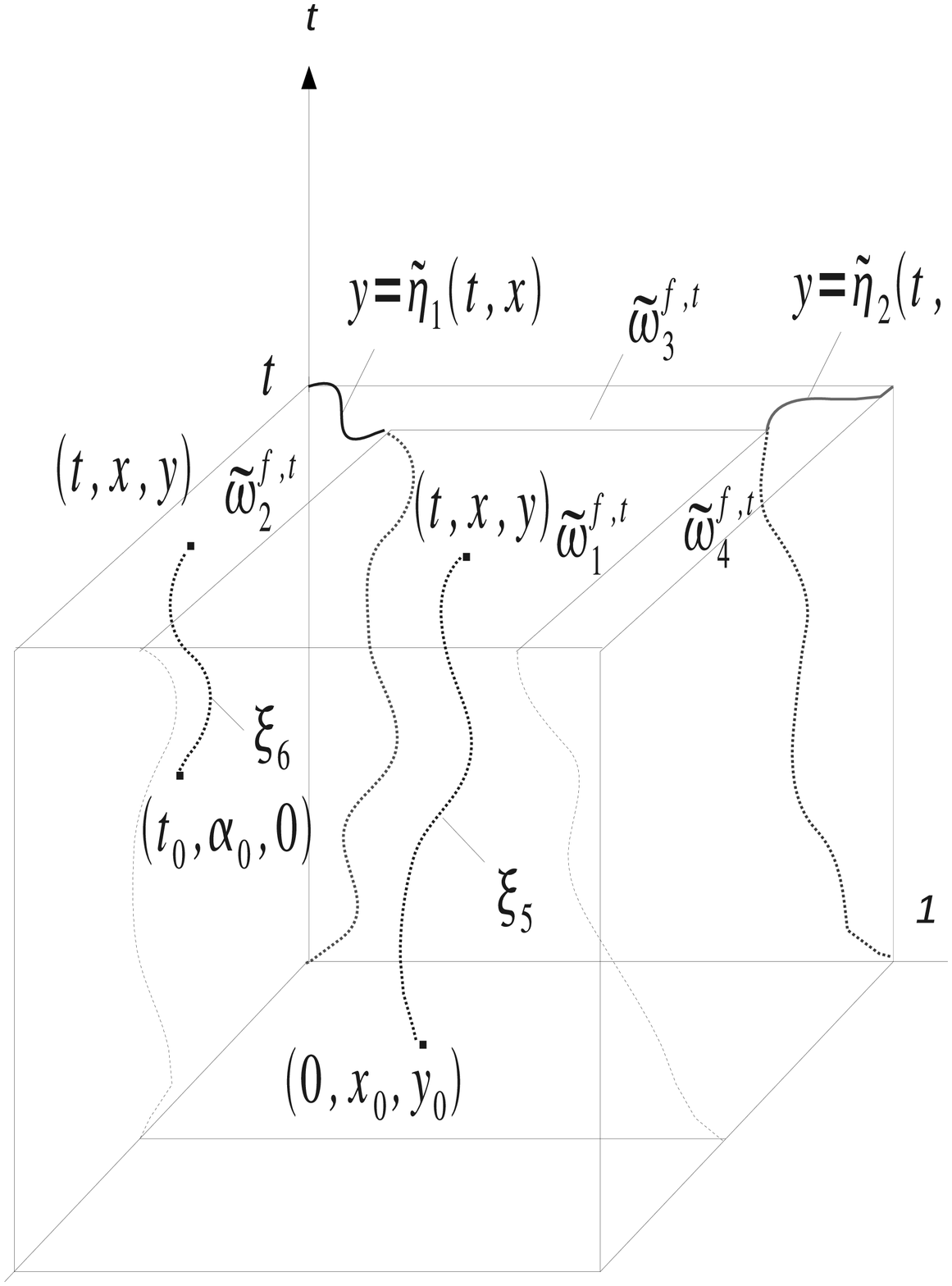}
\end{minipage}}%
\subfigure[]{
\label{d} 
\begin{minipage}[b]{0.32\textwidth}
\centering
\includegraphics[width=1.6in]{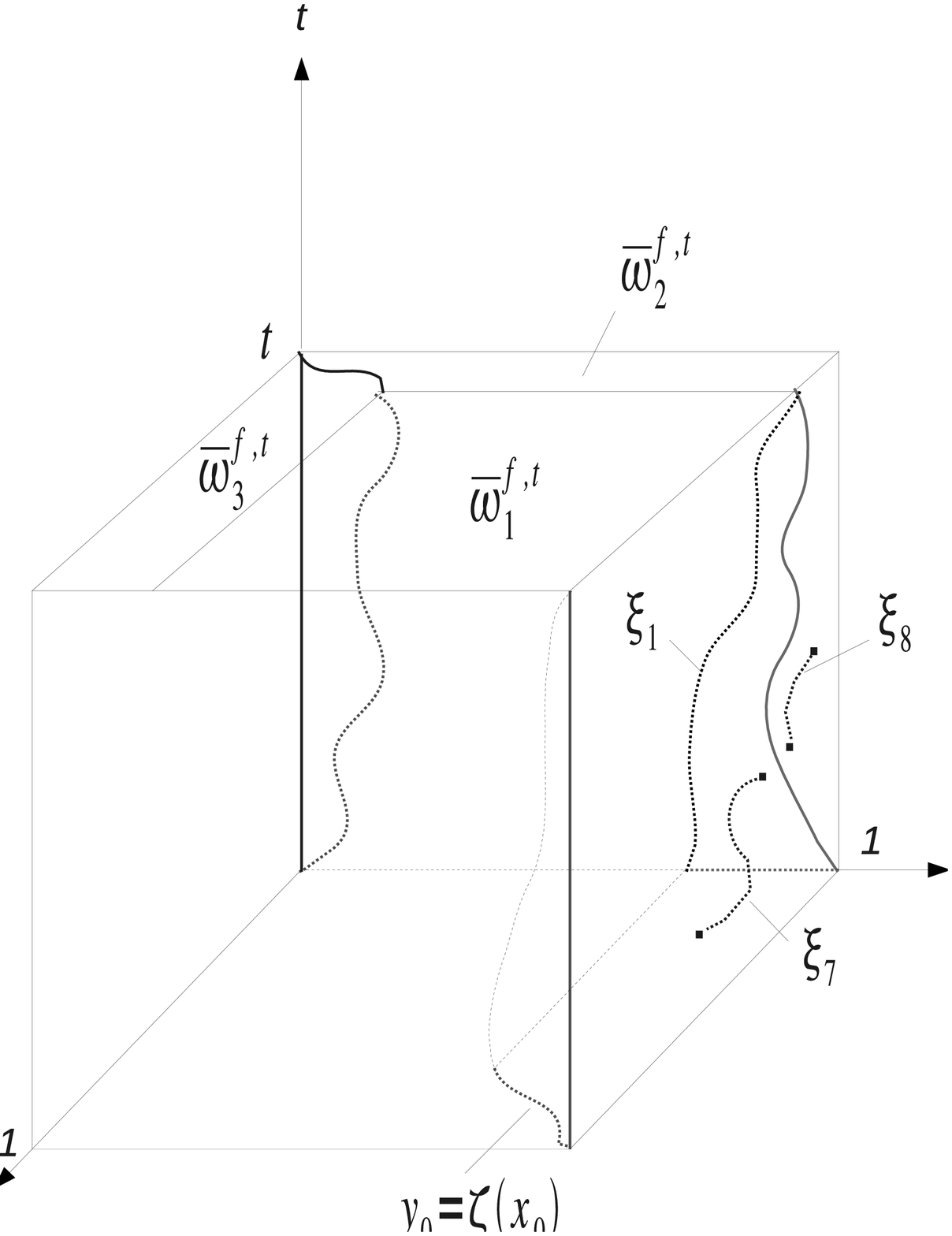}
\end{minipage}}%
\subfigure[]{
\label{e} 
\begin{minipage}[b]{0.32\textwidth}
\centering
\includegraphics[width=1.6in]{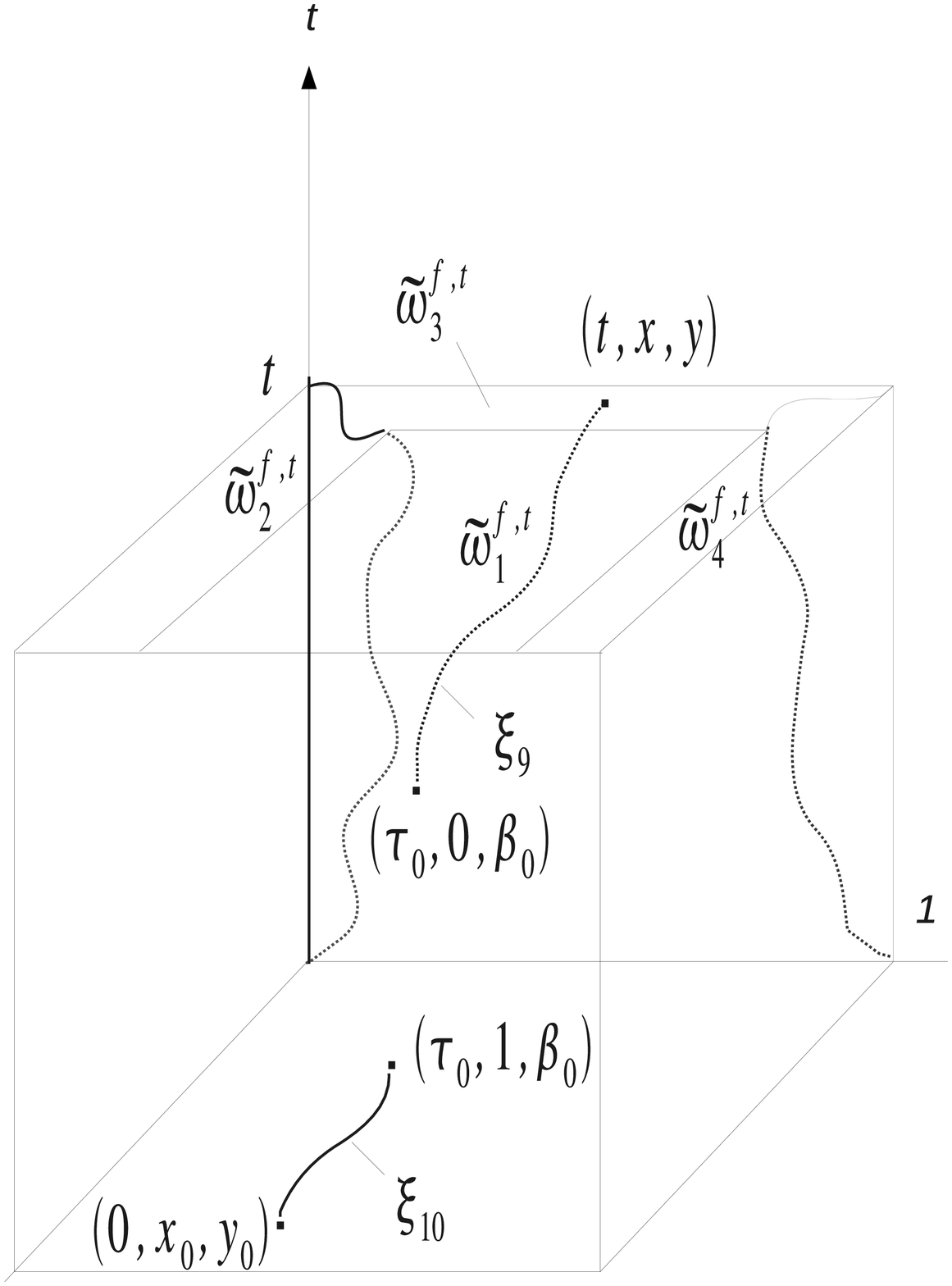}
\end{minipage}}%
\caption{For Phase 3, time $t$ plane is divided into four parts $\widetilde\omega^{f,t}_1$, $\widetilde\omega^{f,t}_2$, $\widetilde\omega^{f,t}_3$ and $\widetilde\omega^{f,t}_4$.
(a) Case $(t,x,y)\in\widetilde\omega^{f,t}_1$, characteristic $\xi_5$ connects $(t,x,y)$ with $(0,x_0,y_0)$; Case $(t,x,y)\in\widetilde\omega^{f,t}_2$, 
characteristic $\xi_6$ connects $(t,x,y)$ with $(t_0,\alpha_0,0)$;
(b) According to the coupled boundary between Phase 1 at $y=1$ and Phase 3 at $y=0$, we separate the right face into two parts, 
and we define characteristic $\xi_7$ and characteristic $\xi_8$;
(c) Case $(t,x,y)\in\widetilde\omega^{f,t}_3$, characteristic $\xi_9$ connects $(t,x,y)$ with $(\tau_0,0,\beta_0)$; According to the coupled boundary
between two consecutive cell cycles, we define characteristic $\xi_{10}$ which connects $(\tau_0,1,\beta_0)$ with $(0,x_0,y_0)$.}
\label{Fig 3} 
\end{figure}

If $(x,y)\in\widetilde\omega^{f,t}_3$ (see Fig 3 (c)), 
we define characteristic $\xi_9=(x_9,y_9)$ by
\begin{equation*}
\displaystyle\f{dx_9}{ds}=\widetilde g_f,\quad
\displaystyle\f{dy_9}{ds}=\widetilde h_f(y_9,u_f)(s),\quad \xi_9(t)=(x,y).
\end{equation*}
There exists a unique $\tau_0\in[0,t]$ such 
that $x_9(\tau_0)=0$, so that we can define 
\be\label{x9}
\beta_0:=y_9(\tau_0).
\ee
According to the coupled boundary between two consecutive cell cycles (corresponding to $\widetilde\phi^f_{k-1}(t,x,y)$ at $x=1$ and
$\widetilde\phi^f_k(t,x,y)$ at $x=0$) $(k=2,\cdots,N)$,
for any fixed $(x_0,y_0)\in[0,1]^2$, 
we define $\xi_{10}=(x_{10},y_{10})$ by (see Fig 3 (c)) 
\begin{equation*}
\displaystyle\f{dx_{10}}{ds}=\widetilde g_f,\quad \displaystyle\f{dy_{10}}{ds}=\widetilde h_f(y_{10},u_f)(s),\quad \xi_{10}(0)=(x_0,y_0).
\end{equation*}

With all of the above defined characteristics and noting Lemma \ref{jacobi} in Appendix 6.3, we are now able to define a map $\vec G(\vec M)(t):=\Big(G_1(M_1,M)(t),\cdots,G_{n}(M_{n},M)(t)\Big)$ for all $t\in[0,\delta]$ with
\begin{align}\label{GGG}
& G_f(M_f,M)(t)\!\!:=\!\!\int_0^{y_1(0)}\!\!\int_0^{1-\int_0^t \ov g_f(u_f(\sigma)) \,d\sigma}\!\!\!\!\!\!\!\!a_1\gamma_s^2\sum_{k=1}^N\ov\phi^f_{k0}(x_0,y_0)y_2(t)e^{-\int_0^t\ov\lambda(y_2,U)(s)\, ds}dx_0dy_0 \\
&+\!\!\int_{1-\widehat g_ft}^1\!\!\int_0^{y_1(\f{1-x_0}{\widehat g_f})}\!\!\!\!\!\!2(a_2-a_1)\gamma^2_s\sum_{k=2}^{N}\widehat\phi^f_{k-1,0}(x_0,y_0) y_3(t) e^{-\int_{\f{1-x_0}{\widehat g_f}}^t\ov\lambda(y_3,U)(s)\, ds}\, dy_0\, dx_0 \nonumber\\
&+\!\!\int_0^1\!\!\int_0^{1-\widehat g_ft}\!\!\!\!\!\!(a_2-a_1)\gamma^2_s\sum_{k=1}^N y_0\widehat\phi^f_{k0}(x_0,y_0)\, dx_0\, dy_0\nonumber\\
&+\!\!\int_{1-\int_0^t \ov g_f(u_f(\sigma))d\sigma}^1\!\!\int_0^{\zeta(x_0)}\!\!\!\!\!\!\!\!a_1\gamma^2_s\sum_{k=1}^N(y_0\!\!+\!\!\!\!\int_0^{\tau_0}\!\!\ov h_f(y_4,u_f)(\sigma)d\sigma)\ov\phi^f_{k0}(x_0,y_0)
e^{-\!\!\int_0^{\tau_0}\ov\lambda(y_4,U)(s)ds}\!dy_0dx_0\nonumber\\
&+\!\!\int_0^1\!\!\int_0^{1-\widetilde g_ft} a_2\gamma_0\sum_{k=1}^N(\gamma_0 y_5(t)\!\!+\!\!\gamma_s)\widetilde\phi^f_{k0}(x_0,y_0)e^{-\int_0^t\widetilde\lambda(y_5,U)(s)\, ds}dx_0dy_0 \nonumber\\
&+\!\!\int_{y_1(0)}^1\!\!\int_0^{\zeta^{-1}(y_0)}\!\!\!\!\!\!\!\!
a_1\gamma_0\sum_{k=1}^N(\gamma_0 y_6(t)\!\!+\!\!\gamma_s)\ov\phi^f_{k0}(x_0,y_0)e^{-(\int_0^{t_0}\ov\lambda(y_7,U)(s)ds+\int_{t_0}^t\widetilde\lambda(y_6,U)(s)ds)}\, dx_0\, dy_0\nonumber\\
&+\!\!\int_{1-\widehat g_ft}^1\!\!\int_{y_1(\f{1-x_0}{\widehat g_f})}^1 \!\!\!\!\!\!\!\!2(a_2\!\!-\!\!a_1)\gamma_0
\sum_{k=2}^{N}(\gamma_0 y_6(t)\!\!+\!\!\gamma_s)\widehat\phi^f_{k-1,0}(x_0,y_0)e^{-\!(\int_{\f{1-x_0}{\widehat g_f}}^{\ov t_0}\ov\lambda(y_8,U)(s)\, ds +\int_{{\ov t_0}}^t\widetilde\lambda(y_6,U)(s)\, ds)}\!\!\!\!\!\!\!\!\!\!\!\!\!\!\!\!\!\!\!\!\!\!\!\!dy_0 dx_0\nonumber\\
&+\!\!\int_0^1\!\!\int_{1-\widetilde g_ft}^1\!\!\!\!a_2\gamma_0\sum_{k=1}^{N-1}(\gamma_0 y_9(t)\!\!+\!\!\gamma_s)\widetilde\phi^f_{k0}(x_0,y_0)
e^{-(\int_{0}^{\f{1-x_0}{\widetilde g_f}}\widetilde\lambda(y_{10},U)(s)ds+\int_{\f{1-x_0}{\widetilde g_f}}^t\widetilde\lambda(y_9,U)(s)ds)}\!\!dx_0dy_0\nonumber .
\end{align}
Here $y_i(t)\ (i=2,\cdots,10)$ can also be determined by $(x_0,y_0)$,
$\tau_0$ is defined by $1-x_0=\displaystyle\int_0^{\tau_0}\ov g_f(u_f(\sigma))d\sigma$,
$t_0$ is defined by $y_7(t_0)=1$, 
and $\ov t_0$ is defined by $y_8(\ov t_0)=1$.
In (\ref{GGG}), $y_0=\zeta(x_0):=\eta(0,x_0)$ (see Fig 3 (b)),
where $\eta(s,x_0)$ satisfies
\begin{equation*}
\displaystyle\f{d\eta}{ds}=\ov h_f(\eta, u_f)(s),\quad \eta(\theta)=1,\quad 0\leq s\leq\theta,
\end{equation*}
with $\theta$ defined by $1-x_0=\displaystyle\int_0^{\theta}\ov g_f(u_f(\sigma))d\sigma$.

Next we prove the following fixed point theorem.
\begin{lem}\label{contraction}
If $\delta$ is small enough, $\vec G$ is a contraction mapping on $\Omega_{\delta,K}$
with respect to the $C^0$ norm.
\end{lem}

\begin{proof}
It is easy to check that $\vec G$ maps into $\Omega_{\delta,K}$ itself if
 \be \label{delta}  
 0<\delta \leq \min\{\f{1}{2K_1},T\},
 \ee
where $K_1$ is defined by (\ref{K1}).
Let $\vec M(t)\!\!=\!\!(M_1(t),\!\cdots\!,M_{n}(t))$, $\ov {\vec M}(t)\!\!=\!\!(\ov M_1(t),\!\cdots\!,\ov M_{n}(t))$ $\in\!\Omega_{\delta,K}$, and 
$M\!\!=\!\!\sum_{f=1}^{n}M_f$, $\ov M\!\!=\!\!\sum_{f=1}^{n}\ov M_f$. 
In order to estimate $\|\vec G(\ov{\vec M})-\vec G(\vec M)\|_{C^0([0,\delta])}$, 
we first estimate the norms $\|G_f(\ov M_f,\ov M)-G_f(M_f,M)\|_{C^0([0,\delta])}$ separately.
Observing the definition (\ref{GGG}) of $G_f$, it is sufficient to estimate $\|\ov y_i-y_i\|_{C^0([0,\delta])}$, where
\begin{eqnarray}\label{yi} 
y_i&=&C+\displaystyle{\int_{\alpha_i}^{\beta_i}{h_f(y_i,u_f)}}(\sigma)d\sigma,\nonumber\\
\ov y_i&=&C+\displaystyle{\int_{\alpha_i}^{\beta_i}{h_f(\ov y_i,\ov u_f)}}(\sigma)d\sigma.
\end{eqnarray}
Here $C$ denotes various constants, $h_f$ represents $\ov h_f$ or $\widetilde h_f$, and $0\!\!\leq\!\! \alpha_i\!\!\leq\!\! \beta_i\!\!\leq \!\!t\!\!\leq \!\!\delta\!\!\leq\!\!\min\{\displaystyle\f{1}{2K_1},T\}$.
We have
\begin{align*}
|{\overline y}_i(s)-y_i(s)|
&\leq \int_{\alpha_i}^{\beta_i}|h_f(\ov y_i,\ov u_f)(\sigma)-h_f(y_i,u_f)(\sigma)|\, d\sigma\nonumber\\
& \leq tK_1(\|\ov M_f-M_f\|_{C^0([0,\delta])}+\|\ov M-M\|_{C^0([0,\delta])})+tK_1\|\ov y_i-y_i\|_{C^0([\alpha_i,\beta_i])},
\end{align*}
hence
\be \label{yyi}
\|\ov y_i-y_i\|_{C^0([0,\delta])}\leq\f{tK_1(\|\ov M_f-M_f\|_{C^0([0,\delta])}+\|\ov M-M\|_{C^0([0,\delta])})}{1-tK_1}.
\ee
By \eqref{yyi}, we have
\begin{align}\label{1hf}
&\int_{\alpha_i}^{\beta_i}|h_f(\ov y_i,\ov u_f)(\sigma)-h_f(y_i,u_f)(\sigma)|d\sigma\\
&\leq tK_1(\|\ov M_f-M_f\|_{C^0([0,\delta])}+\|\ov M-M\|_{C^0([0,\delta])})+tK_1\|\ov y_i-y_i\|_{C^0([0,\delta])}\nonumber\\
&\leq tK_1(\| \ov M_f\!\!-\!\! M_f\|_{C^0([0,\delta])}\!\!+\!\!\| \ov M\!\!-\!\! M\|_{C^0([0,\delta])})\!\!+\!\!\f{t^2K_1^2(\| \ov M_f\!\!-\!\! M_f\|_{C^0([0,\delta])}\!\!+\!\!\| \ov M\!\!-\!\! M\|_{C^0([0,\delta])})}{1-tK_1}\nonumber\\
&=\f{tK_1(\| \ov M_f-M_f\|_{C^0([0,\delta])}+\| \ov M-M\|_{C^0([0,\delta])})}{1-tK_1}\nonumber.
\end{align}
Similarly, we have
\begin{align}\label{1lambda}
\int_{\alpha_i}^{\beta_i}|\lambda(\ov y_i,\ov U)(\sigma)-\lambda(y_i,U)(\sigma)|d\sigma
\leq &tK_1\| \ov M-M\|_{C^0([0,\delta])}+tK_1\| \ov y_i-y_i\|_{C^0([0,\delta])}\nonumber\\
\leq &\f{t^2K_1^2\| \ov M_f-M_f\|_{C^0([0,\delta])}+tK_1\| \ov M-M\|_{C^0([0,\delta])}}{1-tK_1},
\end{align}
where $\lambda$ represents $\ov\lambda$ or $\widetilde\lambda$, and
\be\label{1gf}
\int_{\alpha_i}^{\beta_i}|\ov g_f(\ov u_f(\sigma))-\ov g_f(u_f(\sigma))|d\sigma
\leq tK_1(\| \ov M_f-M_f\|_{C^0([0,\delta])}+\| \ov M-M\|_{C^0([0,\delta])}).
\ee
By \eqref{yyi}-\eqref{1gf} and the definition \eqref{GGG} of $G_f$, we have
\begin{align}
&\|G_f(\ov M_f,\ov M)-G_f(M_f,M)\|_{C^0([0,\delta])}\nonumber\\
&\leq tC_1^f\| \ov M_f-M_f\|_{C^0([0,\delta])}+tC_2^f\| \ov M-M\|_{C^0([0,\delta])}.
\end{align}
The expressions of $C_1^f$ and $C_2^f$ are given by (\ref{Cf}) in Appendix 6.3. They depend on
$\|\ov\phi^f_{k0}\|, \|\widehat\phi^f_{k0}\|, \|\widetilde\phi^f_{k0}\|, K_1$ and $K_2$.
Thus we have
\begin{align}\label{FM}
&\|\vec G(\ov{\vec M})-\vec G(\vec M)\|_{C^0([0,\delta])}\\
=&\Big\|\Big(G_1(\ov M_1,\ov M)-G_1(M_1,M),\cdots, G_{n}(\ov M_{n},\ov M)-G_{n}(M_{n},M)\Big)\Big\|_{C^0([0,\delta])}\nonumber\\
=&\max_{f}\| G_f(\ov M_f,\ov M)-G_f(M_f,M)\|_{C^0([0,\delta])}\nonumber\\
\leq & \max_{f}\Big(tC_1^f\|\ov M_f -M_f\|_{C^0([0,\delta])}+tC_2^f\| \ov M-M\|_{C^0([0,\delta])}\Big)\nonumber\\
\leq & \max_{f}\Big(tC_1^f\|\ov M_f\!\! -\!\!M_f\|_{C^0([0,\delta])}\!\!+\!tC_2^f\|\ov M_1\!\! -\!\!M_1\|_{C^0([0,\delta])}\!\!+\!\cdots +\!tC_2^f\|\ov M_{n}\!\! -\!\!M_{n}\|_{C^0([0,\delta])}\!\Big)\nonumber\\
\leq & \max_{f}t(C_1^f+C_2^f)\Big(\|\ov M_1 -M_1\|_{C^0([0,\delta])}+\cdots +\|\ov M_{n} -M_{n}\|_{C^0([0,\delta])}\Big)\nonumber\\
\leq & nt\max_{f}(C_1^f+C_2^f)\max_{f}\|\ov M_f -M_f\|_{C^0([0,\delta])}\nonumber .
\end{align}
We finally get 
\be 
\|\vec G(\ov{\vec M})-\vec G(\vec M)\|_{C^0([0,\delta])}\leq nt\max_{f}(C_1^f+C_2^f)\|\ov{\vec M} -\vec M\|_{C^0([0,\delta])}.
\ee
Hence we can choose $\delta$ small enough (depending on $\|\ov\phi^f_{k0}\|, \|\widehat\phi^f_{k0}\|, \|\widetilde\phi^f_{k0}\|, K_1, K_2, T $) so that
\be
\|\vec G(\ov{\vec M})-\vec G(\vec M)\|_{C^0([0,\delta])}\leq \f{1}{2}\|\ov{\vec M}-\vec M\|_{C^0([0,\delta])}.
\ee
\end{proof}
By Lemma \ref{contraction} and the contraction mapping principle, there exists a unique fixed point $\vec M=\vec G(\vec M)$ in $\Omega_{\delta,K}$.
\subsection{Construction of a local solution to the Cauchy problem}
Now we show how the fixed point $\vec M$ allows to find a solution to Cauchy problem (\ref{eq})-(\ref{bcy1}) for $t\in[0,\delta]$.
Let us recall the definition of three subsets $\ov\omega^{f,t}_1$, $\ov\omega^{f,t}_2$ and $\ov\omega^{f,t}_3$ of $[0,1]^2$, the definition
of the characteristics $\xi_2, \xi_3$
and also the definition (\ref{x2}) of $(x_0,y_0)$ and (\ref{x3}) of $(\tau_0,\beta_0)$.
For $k=1$, we define $\ov\phi^f_1(t,x,y)$ by
\be \label{for1}
\ov\phi^f_1(t,x,y):=\left\{
\begin{array}{l}
\ov\phi^f_{10}(x_0, y_0)e^{-\int_0^t[\ov\lambda(y_2,U)+\f{\pa\ov h_f}{\pa y}(y_2,u_f)](\sigma)\, d\sigma},\quad \mbox{if}\ (x,y)\in \ov\omega^{f,t}_1,\\
 0,\quad \mbox{else}.
 \end{array}
\right.
\ee
For $k=2,\cdots,N$, we define $\ov\phi^f_k(t,x,y)$ by
\be \label{so1}
\ov\phi^f_k(t,x,y):=\left\{
\begin{array}{l}
\ov\phi^f_{k0}(x_0, y_0)e^{-\int_0^t[\ov\lambda(y_2,U)+\f{\pa\ov h_f}{\pa y}(y_2,u_f)](\sigma)\, d\sigma},\quad \mbox{if}\ (x,y)\in \ov\omega^{f,t}_1,\\
\displaystyle\f{2\tau_{gf}\widehat\phi^f_{k-1}(\tau_0,1,\beta_0)}{a_1\ov g_f(u_f(\tau_0))}\ e^{-\int_{\tau_0}^t[\ov\lambda(y_3,U)+\f{\pa\ov h_f}{\pa y}(y_3,u_f)](\sigma)\,d\sigma},\ \mbox{if}\ (x,y)\in \ov\omega^{f,t}_2,\\
 0,\quad \mbox{else}.
\end{array}
\right.
\ee
Since the dynamic in Phase 2 amounts to pure transport equations, we define $\widehat\phi^f_k(t,x,y)$, $k=1,\cdots,N$ by
\be \label{phase2}
\widehat\phi^f_k(t,x,y):=\left\{
\begin{array}{l}
\widehat\phi^f_{k0}(x-\widehat g_ft,y),\quad \mbox{if}\ (x,y)\in [\widehat g_ft,1]\times[0,1],\\
\displaystyle\f{a_1\ov g_f(t-\f{x}{\widehat g_f})}{\tau_{gf}}\ \ov\phi^f_k(t-\f{x}{\widehat g_f},1,y),\quad\mbox{if}\ (x,y)\in [0,\widehat g_ft]\times[0,1].
\end{array}
\right.
\ee
Let us recall the definition of the four subsets $\widetilde\omega^{f,t}_1$, $\widetilde\omega^{f,t}_2$, $\widetilde\omega^{f,t}_3$ and $\widetilde\omega^{f,t}_4$ of $[0,1]^2$,
the definition of the characteristics $\xi_5, \xi_6, \xi_9$ and also the definition (\ref{x5}) of $(x_0,y_0)$, (\ref{x6}) of $(t_0,\alpha_0)$ and (\ref{x9}) of $(\tau_0,\beta_0)$.
For $k=1$, we define $\widetilde\phi^f_1(t,x,y)$ by
\be \label{for11}
\widetilde\phi^f_1(t,x,y):=\left\{
\begin{array}{l}
\widetilde\phi^f_{10}(x_0,y_0)e^{-\int_0^t[\widetilde\lambda(y_5,U)
+\f{\pa\widetilde h_f}{\pa y}(y_5,u_f)](\sigma)\, d\sigma},\quad\mbox{if}\ (x,y)\in \widetilde\omega^{f,t}_1,\\
\ov\phi^f_1(t_0,\displaystyle\f{a_2}{a_1}\alpha_0, 1)e^{-\int_{t_0}^t[\widetilde\lambda(y_6,U)
+\f{\pa\widetilde h_f}{\pa y}(y_6,u_f)](\sigma)\, d\sigma},\quad\mbox{if}\ (x,y)\in \widetilde\omega^{f,t}_2,\\
0,\quad \mbox{else}.
\end{array}
\right.
\ee
For $k=2,\cdots,N$, we define $\widetilde\phi^f_k(t,x,y)$ by
\be\label{so3}
\widetilde\phi^f_k(t,x,y):=\left\{
\begin{array}{l}
\widetilde\phi^f_{k0}(x_0,y_0)e^{-\int_0^t[\widetilde\lambda(y_5,U)
+\f{\pa\widetilde h_f}{\pa y}(y_5,u_f)](\sigma)\, d\sigma},\quad\mbox{if}\ (x,y)\in \widetilde\omega^{f,t}_1,\\
\ov\phi^f_k(t_0,\displaystyle\f{a_2}{a_1}\alpha_0, 1)e^{-\int_{t_0}^t[\widetilde\lambda(y_6,U)
+\f{\pa\widetilde h_f}{\pa y}(y_6,u_f)](\sigma)\, d\sigma},\quad\mbox{if}\ (x,y)\in \widetilde\omega^{f,t}_2,\\
\widetilde\phi^f_{k-1}(\tau_0, 1,\beta_0)e^{-\int_{\tau_0}^t[\widetilde \lambda(y_9,U)+\f{\pa\widetilde h_f}{\pa y}(y_9,u_f)](\sigma)\,d\sigma},\quad\mbox{if}\ (x,y)\in \widetilde\omega^{f,t}_3,\\
0,\quad \mbox{else}.
\end{array}
\right.
\ee
Next we prove that the vector function $\vec\phi_f$ defined by (\ref{for1})-(\ref{so3})
is a weak solution to Cauchy problem (\ref{eq})-(\ref{bcy1}) for $t\in[0,\delta]$. 
To that end, we first prove that $\vec\phi_f$ defined by (\ref{for1})-(\ref{so3}) satisfies
equality \eqref{intequality} of Definition \ref{weaksol}, then we prove that $\vec\phi_f\in C^0([0,\delta];L^1((0,1)^2))$.
Let $\tau\in[0,\delta]$. For any vector function $\vec{\varphi}\in
C^1([0,\tau]\times[0,1]^2)$, $\vec{\varphi}:=(\ov\varphi_1,\cdots,\ov\varphi_N,\widehat\varphi_1,\cdots,\widehat\varphi_N,\\
\widetilde\varphi_1,\cdots,\widetilde\varphi_N)^T$ such that \eqref{varphi}-\eqref{varphi0} hold, by definition (\ref{for1})-(\ref{so3}) of $\vec\phi_f$, we have
\be \label{ai}
\int_0^{\tau}\int_0^1\int_0^1\vec{\phi}_f(t,x,y)\cdot(\vec{\varphi}_t+A_f\vec{\varphi}_x+
B_f\vec{\varphi}_y+C\vec{\varphi})dxdydt:=\sum_{i=1}^7 A_i.
\ee
In (\ref{ai}),
\begin{align*}
A_1\!\!:=&\!\!\sum_{k=1}^N\!\!\int_0^{\tau}\!\!\iint\limits_{\ov\omega^{f,t}_1}\ov\phi^f_{k0}(x_0,y_0)e^{-\int_0^t[\ov\lambda(y_2,U)+\f{\pa\ov h_f}{\pa y}(y_2,u_f)](\sigma)d\sigma}(\ov\varphi_{kt}\!\!+\!\!\ov g_f\ov\varphi_{kx}\!\!+\!\!\ov h_f\ov\varphi_{ky}\!\!-\!\!\ov\lambda\ov\varphi_k)
dxdydt,\\
A_2\!\!:=&\!\!\sum_{k=2}^N\!\!\int_0^{\tau}\!\!\!\!\iint\limits_{\ov\omega^{f,t}_2}\!\f{2\tau_{gf}\widehat\phi^f_{k-1}(\tau_0,\!1,\!\beta_0)}{a_1\ov g_f(u_f(\tau_0))}
e^{-\!\int_{\tau_0}^t[\ov\lambda(y_3,U)\!+\!\f{\pa\ov h_f}{\pa y}(y_3,u_f)](\sigma)d\sigma}\!(\ov\varphi_{kt}\!\!+\!\!\ov g_f\ov\varphi_{kx}\!\!+\!\!\ov h_f\ov\varphi_{ky}\!\!-\!\!\ov\lambda\ov\varphi_k)
dxdydt,\\
A_3\!\!:=&\!\!\sum_{k=1}^N\!\!\int_0^{\tau}\int_0^1\int_{\widehat g_ft}^1\widehat\phi^f_{k0}(x-\widehat g_ft,y)(\widehat\varphi_{kt}+\widehat g_f\widehat\varphi_{kx})dxdydt,\\
A_4\!\!:=&\!\!\sum_{k=1}^N\!\!\int_0^{\tau}\int_0^1\int_0^{\widehat g_ft}\f{a_1\ov g_f(u_f(t-\f{x}{\widehat g_f}))}{\tau_{gf}}\ \ov \phi^f_k(t-\f{x}{\widehat g_f},1,y)
(\widehat\varphi_{kt}+\widehat g_f\widehat\varphi_{kx})dxdydt,\\
A_5\!\!:=&\!\!\sum_{k=1}^N\!\!\int_0^{\tau}\!\!\iint\limits_{\widetilde\omega^{f,t}_1}\widetilde\phi^f_{k0}(x_0,y_0)e^{-\int_0^t[\widetilde\lambda(y_5,U)+\f{\pa\widetilde h_f}{\pa y}(y_5,u_f)](\sigma)d\sigma}(\widetilde\varphi_{kt}\!\!+\!\!\widetilde g_f\widetilde\varphi_{kx}\!\!+\!\!\widetilde h_f\widetilde\varphi_{ky}\!\!-\!\!\widetilde\lambda\widetilde\varphi_k)
dxdydt,\\
A_6\!\!:=&\!\!\sum_{k=1}^N\!\!\int_0^{\tau}\iint\limits_{\widetilde\omega^{f,t}_2}
\ov\phi^f_{k}(t_0,\f{a_2}{a_1}\alpha_0,1)e^{-\int_{t_0}^t[\widetilde\lambda(y_6,U)
+\f{\pa\widetilde h_f}{\pa y}(y_6,u_f)](\sigma)d\sigma}(\widetilde\varphi_{kt}\!\!+\!\!\widetilde g_f\widetilde\varphi_{kx}\!\!+\!\!\widetilde h_f\widetilde\varphi_{ky}\!\!-\!\!\widetilde\lambda\widetilde\varphi_k)dxdydt,\\
A_7\!\!:=&\!\!\sum_{k=2}^N\!\!\int_0^{\tau}\!\!\iint\limits_{\widetilde\omega^{f,t}_3}\widetilde\phi^f_{k-1}(\tau_0,1,\beta_0)e^{-\int_{\tau_0}^t[\widetilde \lambda(y_9,U)
+\f{\pa\widetilde h_f}{\pa y}(y_9,u_f)](\sigma)d\sigma}(\widetilde\varphi_{kt}\!\!+\!\!\widetilde g_f\widetilde\varphi_{kx}\!\!+\!\!\widetilde h_f\widetilde\varphi_{ky}\!\!-\!\!\widetilde\lambda\widetilde\varphi_k)dxdydt.
\end{align*}
Let us consider the first term $A_1$ as an instance. By the change of variable
$(x,y)\rightarrow(x_0,y_0)$ and noting \eqref{jy2} of Lemma \ref{jacobi} in Appendix 6.3, we have
\begin{align*}
&A_1\!\!=\!\!\sum_{k=1}^N\int_0^{\tau}\int_0^{\beta}\int_0^{\alpha}
\ov\phi^f_{k0}(x_0,y_0)e^{-\int_0^t\ov\lambda(y_2,U)(\sigma)d\sigma}
\Big(\ov\varphi_{kt}(t,\!x_2(t),\!y_2(t))\!\!+\!\!\ov g_f\ov\varphi_{kx}(t,\!x_2(t),\!y_2(t))\nonumber\\
&\qquad\qquad\qquad\qquad\qquad +\ov h_f\ov\varphi_{ky}(t,x_2(t),y_2(t))-\ov\lambda\ov\varphi_k(t,x_2(t),y_2(t))\Big)dx_0dy_0dt,
\end{align*}
where (see Fig 4 (a) and (b))
\be
\label{suppose}
\alpha:= 1-\int_0^t\ov g_f(u_f(\sigma))\,d\sigma :=f_1(t),\quad
\beta:= 1-\int_0^t\ov h_f(y_1,u_f)(\sigma)d\sigma:=g_1(t).
\ee
Clearly, $\alpha$ is a function of $\beta$, suppose that $\alpha=h(\beta)$.
After changing the order of integration, $A_1$ can be rewritten as 
\begin{align*}
A_1\!\!=\!\!\sum_{k=1}^N&\!\Big\{\!\!\int_0^{g_1(\tau)}\!\!\int_0^{f_1(\tau)}\!\!\int_0^{\tau}\!\!+\!\!\int_0^{g_1(\tau)}\!\!\int_{f_1(\tau)}^1\!\!\int_0^{f_1^{-1}(\alpha)}\!\!+\!\!
\int_{g_1(\tau)}^1\!\!\int_0^{f_1(\tau)}\!\!\int_0^{g_1^{-1}(\beta)}\!\!+\!\!\int_{g_1(\tau)}^1\!\!\int_{h(\beta)}^1\!\!\int_0^{f_1^{-1}(\alpha)}\\
&+\!\!\int_{f_1(\tau)}^1\!\!\int_{h^{-1}(\alpha)}^1\!\!\int_0^{g_1^{-1}(\beta)}\Big\}\ov\phi^f_{k0}(x_0,y_0)\f{d\left(e^{-\int_0^t\ov\lambda(y_2,U)(\sigma)d\sigma}
\ov\varphi_k(t,x_2(t),y_2(t))\right)}{dt}dtdx_0y_0\\
=\!\!\sum_{k=1}^N&\!\Big\{\!\!\!-\!\!\!\int_0^{g_1(\tau)}\!\!\!\!\int_0^{f_1(\tau)}\!\ov\phi^f_{k0}(x_0,y_0)\ov\varphi_k(0,x_0,y_0)dx_0dy_0
\!\!-\!\!\int_0^{g_1(\tau)}\!\!\!\!\int_{f_1(\tau)}^1\!\ov\phi^f_{k0}(x_0,y_0)\ov\varphi_k(0,x_0,y_0)dx_0dy_0\\
&-\!\!\int_{g_1(\tau)}^1\!\!\int_0^{f_1(\tau)}\ov\phi^f_{k0}(x_0,y_0)\ov\varphi_k(0,x_0,y_0)dx_0dy_0
\!\!-\!\!\int_{g_1(\tau)}^1\!\!\int_{h(\beta)}^1\ov\phi^f_{k0}(x_0,y_0)\ov\varphi_k(0,x_0,y_0)dx_0dy_0\\
&-\!\!\int_{f_1(\tau)}^1\!\!\int_{h^{-1}(\alpha)}^1\ov\phi^f_{k0}(x_0,y_0)\ov\varphi_k(0,x_0,y_0)dy_0dx_0\\
&+\!\!\int_0^{g_1(\tau)}\!\!\int_{f_1(\tau)}^1\!\!\ov\phi^f_{k0}(x_0,y_0)
e^{-\int_0^{f_1^{-1}(\alpha)}\ov\lambda(y_2,U)(\sigma)d\sigma}\ov\varphi_k(f_1^{-1}(\alpha),x_2(f_1^{-1}(\alpha)),y_2(f_1^{-1}(\alpha)))dx_0dy_0\\
&+\!\!\int_{g_1(\tau)}^1\!\!\int_0^{f_1(\tau)}\ov\phi^f_{k0}(x_0,y_0)
e^{-\int_0^{g_1^{-1}(\beta)}\ov\lambda(y_2,U)(\sigma)d\sigma}\ov\varphi_k(g_1^{-1}(\beta),x_2(g_1^{-1}(\beta)),y_2(g_1^{-1}(\beta)))dx_0dy_0\\
&+\!\!\int_{g_1(\tau)}^1\!\!\int_{h(\beta)}^1\ov\phi^f_{k0}(x_0,y_0)
e^{-\int_0^{f_1^{-1}(\alpha)}\ov\lambda(y_2,U)(\sigma)d\sigma}\ov\varphi_k(f_1^{-1}(\alpha),x_2(f_1^{-1}(\alpha)),y_2(f_1^{-1}(\alpha)))dx_0dy_0\\
&+\!\!\int_{f_1(\tau)}^1\!\!\int_{h^{-1}(\alpha)}^1\ov\phi^f_{k0}(x_0,y_0)
e^{-\int_0^{g_1^{-1}(\beta)}\ov\lambda(y_2,U)(\sigma)d\sigma}\ov\varphi_k(g_1^{-1}(\beta),x_2(g_1^{-1}(\beta)),y_2(g_1^{-1}(\beta)))dx_0dy_0
\Big\}.
\end{align*}
Changing the order of integration again (see Fig 4 (c)), we obtain
\be \label{good}
-\!\!\int_{f_1(\tau)}^1\!\!\int_{h^{-1}(\alpha)}^1\!\!\ov\phi^f_{k0}(x_0,y_0)\ov\varphi_k(0,x_0,y_0)dy_0dx_0
\!\!=\!\!-\!\!\int_{g_1(\tau)}^1\!\!\int_{f_1(\tau)}^{h(\beta)}\ov\phi^f_{k0}(x_0,y_0)\ov\varphi_k(0,x_0,y_0)dx_0dy_0.
\ee
By \eqref{good} and noting that $x_2(f_1^{-1}(\alpha))=y_2(g_1^{-1}(\beta))=1$, we get
\be \label{g}
A_1=-\sum_{k=1}^N\int_0^1\int_0^1\ov\phi^f_{k0}(x_0,y_0)\ov\varphi_k(0,x_0,y_0)dx_0dy_0.
\ee
Similar to $A_1$, we can prove that 
\begin{align}
&A_2=-\sum_{k=2}^N\int_0^{\tau}\int_0^1\f{2\tau_{gf}}{a_1}\ \widehat\phi^f_{k-1}(\tau_0,1,\beta_0)\ov\varphi_k(\tau_0,0,\beta_0)d\beta_0 d\tau_0,\\
&A_3=-\sum_{k=1}^N\int_0^1\int_0^1\widehat\phi^f_{k0}(x_0,y_0)\widehat\varphi_k(0,x_0,y_0)dx_0dy_0,\\
&A_4=-\sum_{k=1}^N\int_0^{\tau}\int_0^1\f{a_1\widehat g_f\ov g_f(u_f(\tau_0))}{\tau_{gf}}\ \ov\phi^f_k(\tau_0,1,\beta_0)\widehat\phi_k(\tau_0,0,\beta_0)d\beta_0 d\tau_0,\\
&A_5=-\sum_{k=1}^N\int_0^1\int_0^1\widetilde\phi^f_{k0}(x_0,y_0)\widetilde\varphi_k(0,x_0,y_0)dx_0dy_0,\\
&A_6=-\sum_{k=1}^N\int_0^{\tau}\int_0^{\f{a_1}{a_2}}\widetilde h_f(0,u_f(t_0))\ov\phi^f_k(t_0,\f{a_1}{a_2}\alpha_0,1)\widetilde\varphi_k(t_0,\alpha_0,0)d\alpha_0 dt_0,\\
&\label{A}
A_7=-\sum_{k=2}^N\int_0^{\tau}\int_0^1 \widetilde g_f\widetilde\phi^f_{k-1}(\tau_0,1,\beta_0)\widetilde\varphi_k(\tau_0,0,\beta_0)d\beta_0 d\tau_0.
\end{align}

\begin{figure}[htbp!]
\subfigure[]{
\label{f} 
\begin{minipage}[b]{0.32\textwidth}
\centering
\includegraphics[width=1.6in]{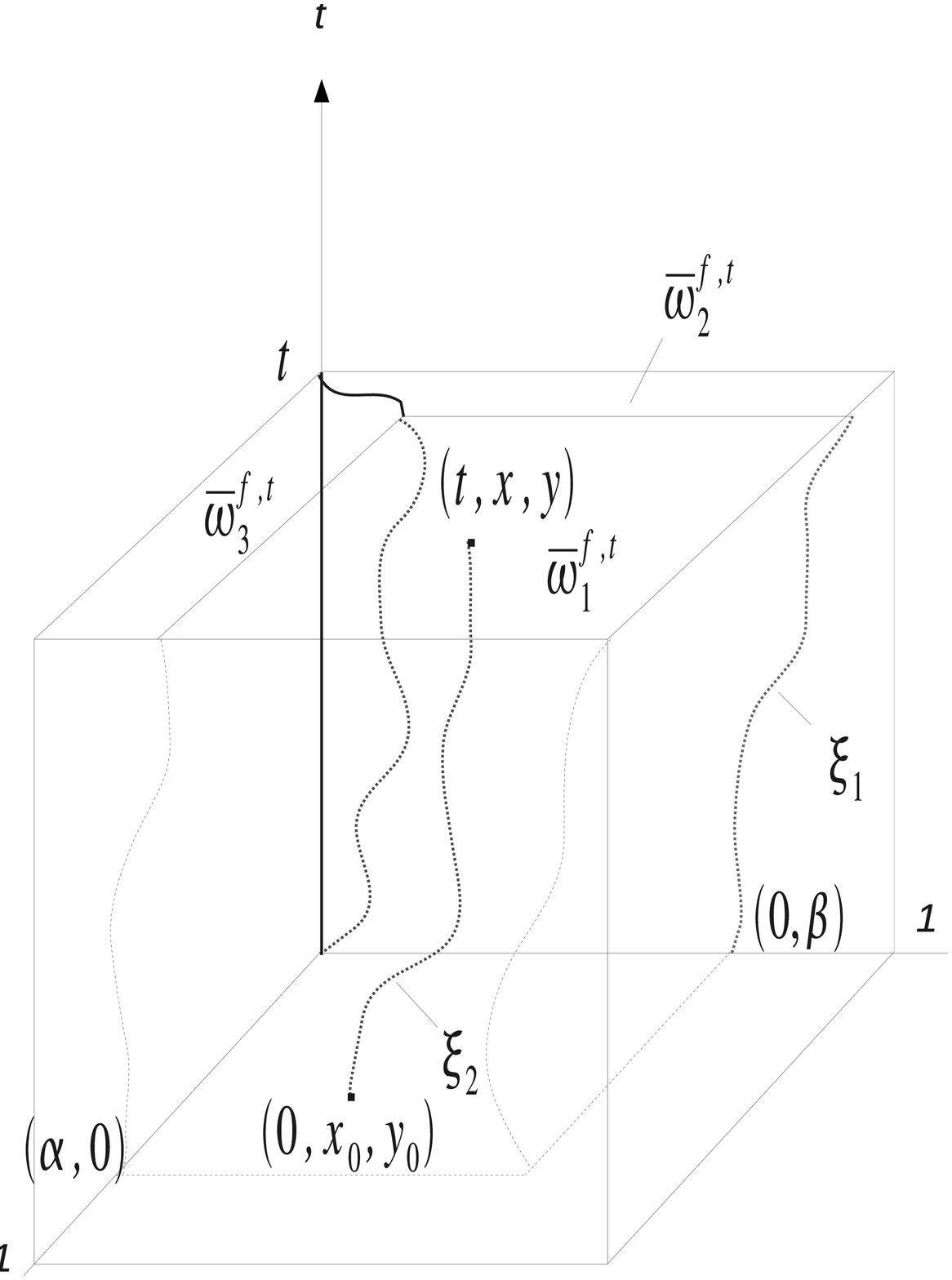}
\end{minipage}}%
\subfigure[]{
\begin{minipage}[b]{0.32\textwidth}
\centering
\includegraphics[width=1.6in]{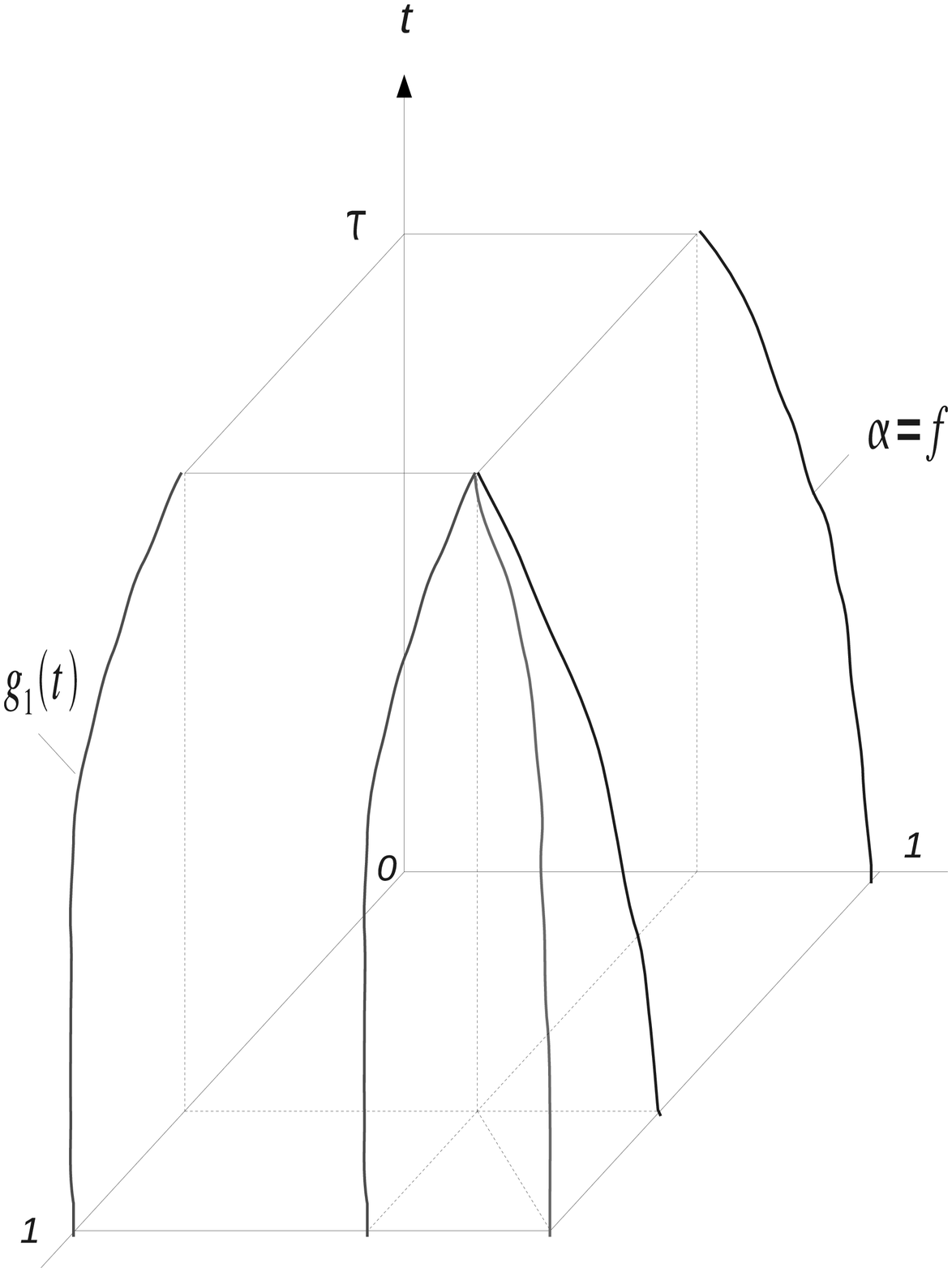}
\end{minipage}}%
\subfigure[]{
\label{h} 
\begin{minipage}[b]{0.35\textwidth}
\centering
\includegraphics[width=1.6in]{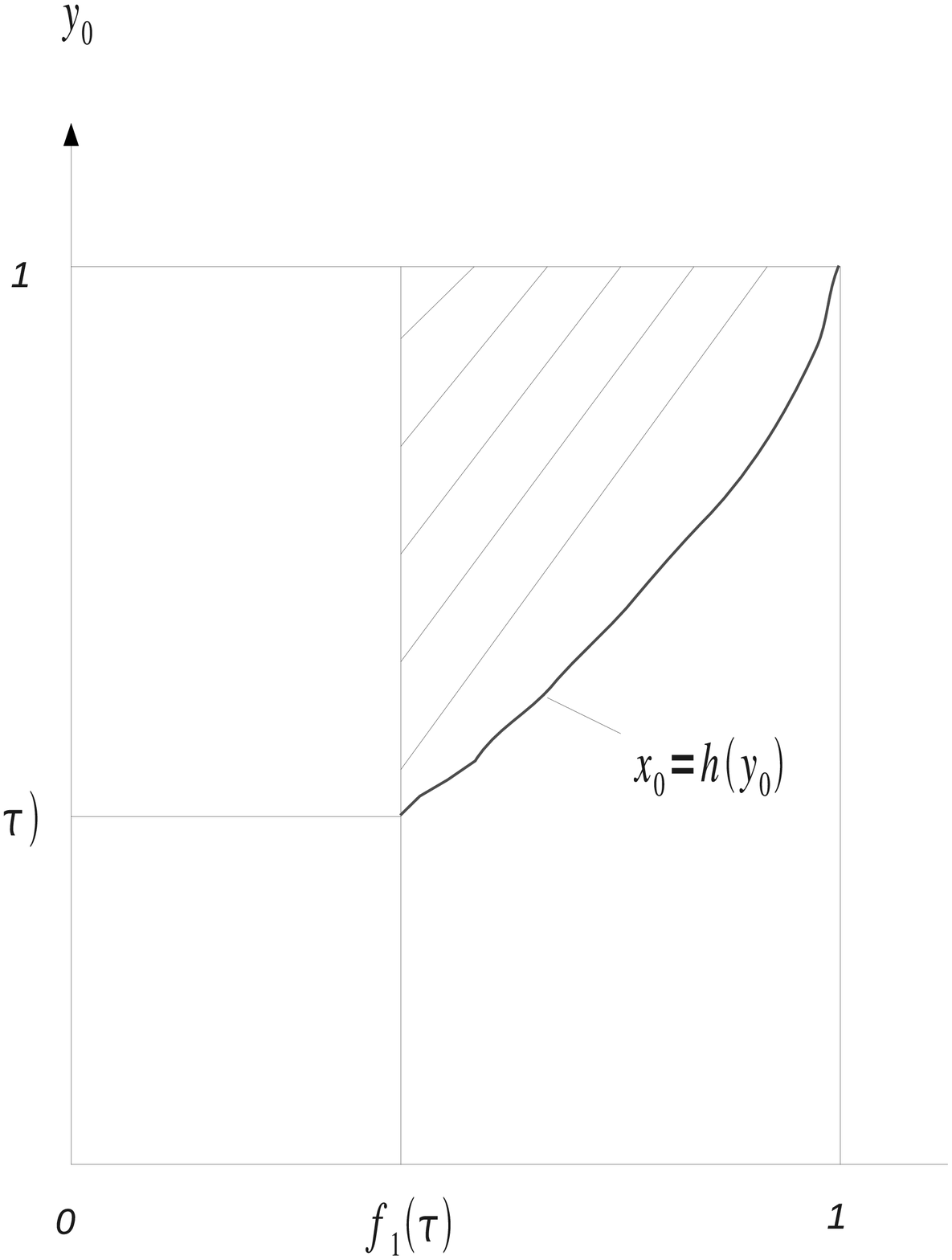}
\end{minipage}}%
\caption{(a) For any fixed $t\in[0,\tau]$, the characteristics $\xi_1, \xi_2$ and the definition of $\alpha, \beta$;
(b) The illustration of functions $\alpha=f_1(t)$ and $\beta=g_1(t)$, $t\in[0,\tau]$;
(c) For changing the order of integration.}
\label{Fig 4} 
\end{figure}

By (\ref{g})-(\ref{A}), we have proven that the vector function $\vec\phi_f$ defined by (\ref{for1})-(\ref{so3}) satisfies  \eqref{intequality}.
Next we prove that $\vec\phi_f\in C^0([0,\delta];L^1((0,1)^2))$. Moreover, we can prove that $\vec\phi_f$ even belongs to $C^0([0,\delta];L^p((0,1)^2))$, for all $p\in[1,\infty)$.

\begin{lem}\label{ex-sol}
The weak solution $\vec\phi_f$ to Cauchy problem \eqref{eq}-\eqref{bcy1}
belongs to $C^0([0,\delta];L^p((0,1)^2))$ for all $p\in[1,\infty)$.
\end{lem}

\begin{proof}
From the definition \eqref{for1}-\eqref{so3} of $\vec\phi_f$, we get easily that $\vec\phi_f\in L^{\infty}((0,\delta)\times(0,1)^2)$.
Next we prove that the vector function $\vec\phi_f$ belongs to $C^0([0,\delta];L^p((0,1)^2))$ for all 
$p\in[1,\infty)$, i.e., for every $\tilde t,t\in[0,\delta]$ with $\tilde t\geq t$ (the case that $\tilde t\leq t$ can be treated similarly),
we need to prove 
 \begin{equation*}  \|\vec\phi_f(\tilde t,\cdot)-\vec\phi_f(t,\cdot)\|_{L^p((0,1)^2)}\rightarrow 0,
    \quad \text{as} \quad |\tilde t-t|\rightarrow 0,\quad \forall p\in[1,\infty).
   \end{equation*} 
In order to do that, we estimate $\|\ov\phi^f_k(\tilde t,\cdot)-\ov\phi^f_k(t,\cdot)\|_{L^p((0,1)^2)}$,
$\|\widehat\phi^f_k(\tilde t,\cdot)-\widehat\phi^f_k(t,\cdot)\|_{L^p((0,1)^2)}$ and
$\|\widetilde\phi^f_k(\tilde t,\cdot)-\widetilde\phi^f_k(t,\cdot)\|_{L^p((0,1)^2)} (k=1,\cdots,N)$ separately.

Suppose that the characteristic passing through $(t,0,0)$ intersects time $\tilde t$ plane at $(\tilde t,q_1,q_2)$ (see Fig 5 (a)), we have
\begin{equation*}
q_1=\int_t^{\tilde t}\ov g_f(u_f(\sigma))d\sigma,\quad q_2=\int_t^{\tilde t}\ov h_f(y,u_f)(\sigma)d\sigma,
\end{equation*}
where $y(s), t\leq s\leq \tilde t$ satisfies 
\begin{equation*}
\displaystyle\f{dy}{ds}=\ov h_f(y,u_f)(s),\quad y(t)=0.
\end{equation*}
We get easily that 
\be \label{m}
q_1\leq C|\tilde t-t|,\quad q_2\leq C|\tilde t-t|.
\ee
Here and hereafter in this section, we denote by $C$ various 
constants which do not depend on $\tilde t$, $t$, $x$ or $y$.
Let $\ov S:=[q_1,1]\times[q_2,1]$ (see Fig 5 (a)), 
following our method to construct the solution $\ov\phi^f_k$, when $(x,y)\in \ov S$, we have
\begin{align*}
&|\ov\phi^f_k(\tilde t,x,y)-\ov\phi^f_k(t,x,y)|\\
&=\Big|\ov\phi^f_k(t,\tilde x,\tilde y)e^{-\int_t^{\tilde t}[\ov\lambda(\tilde y_2,U)+\f{\pa\ov h_f}{\pa y}(\tilde y_2,u_f)](\sigma)\,d\sigma}-\ov\phi^f_k(t,x,y)\Big|\\
&\leq  |\ov\phi^f_k(t,\tilde x,\tilde y)-\ov\phi^f_k(t,x,y)|
e^{-\int_t^{\tilde t}[\ov\lambda(\widetilde y_2,U)+\f{\pa\ov h_f}{\pa y}(\widetilde y_2,u_f)](\sigma)\,d\sigma}
+C|\ov\phi^f_k(t,x,y)||\tilde t-t|.
\end{align*}
Here $\widetilde\xi_2=(\widetilde x_2,\widetilde y_2)$ denotes the characteristic passing through $(\tilde t,x,y)$
that intersects time $t$ plane at
$(t,\tilde x,\tilde y)$. Hence $(\tilde x,\tilde y)$ is defined by
\begin{equation*}
\tilde x=x-\int_t^{\tilde t}\ov g_f(u_f(\sigma))d\sigma,\quad
\tilde y=y-\int_t^{\tilde t}\ov h_f(\widetilde y_2,u_f)(\sigma)d\sigma.
\end{equation*}
We have
\be\label{xy}
|\tilde x-x|\leq C|\tilde t-t|,\quad |\tilde y-y|\leq C|\tilde t-t|.
\ee
Since $\ov\phi^f_k\in{L^{\infty}((0,\delta)}\times(0,1)^2)$, for every $l>0$, there exists $C_l$ such that for every $t\in[0,\delta]$, there exists $\ov\phi^{fl}_k(t,\cdot)\in C^1([0,1]^2)$ satisfying 
\be\label{approximate1}
\|\ov\phi^{fl}_k(t,\cdot)-\ov\phi^f_k(t,\cdot)\|_{L^p((0,1)^2)}\leq \f{1}{l},\quad \|\ov\phi^{fl}_k(t,\cdot)\|_{C^1([0,1]^2)}\leq C_l.
\ee
Here and hereafter in this section, we denote by $C_l$ various constants that may depend on $l$ (the index of the corresponding approximating sequences $\ov\phi^{fl}_k(t,\cdot)$, $\widehat\phi^{fl}_k(t,\cdot)$ and $\widetilde\phi^{fl}_k(t,\cdot)$) but are independent of $0\leq t\leq \tilde t\leq \delta$.

Thus, we have
\begin{align}\label{oo}
&|\ov\phi^f_k(\tilde t,x,y)-\ov\phi^f_k(t,x,y)|\nonumber\\
\leq &C|\ov\phi^{fl}_k(t,\tilde x,\tilde y)-\ov\phi^{f}_k(t,\tilde x,\tilde y)|
+C|\ov\phi^{fl}_k(t,x,y)-\ov\phi^{f}_k(t,x,y)|\nonumber\\
&+C|\ov\phi^f_k(t,x,y)||\tilde t-t|+C_l|\tilde t-t|.
\end{align}
By \eqref{approximate1} and \eqref{oo}, we have
\be \label{estimate1}
\iint\limits_{\ov S}|\ov\phi^f_k(\tilde t,x,y)-\ov\phi^f_k(t,x,y)|^pdxdy\leq \f{C}{l^p}+C_l|\tilde t-t|^p.
\ee
Noting \eqref{m}, we have 
\be \label{estimate2}
\iint\limits_{[0,1]^2\backslash\ov S}|\ov\phi^f_k(\tilde t,x,y)-\ov\phi^f_k(t,x,y)|^pdxdy\leq C(q_1+q_2)\leq C|\tilde t-t|.
\ee
Combining \eqref{estimate1} with \eqref{estimate2}, we obtain
\be\label{ovestimate1}
\int_0^1\int_0^1|\ov\phi^f_k(\tilde t,x,y)-\ov\phi^f_k(t,x,y)|^pdxdy
\leq  \f{C}{l^p}
+C_l|\tilde t-t|^p+C|\tilde t-t|.
\ee
\begin{figure}[!htbp]
\subfigure[]{
\begin{minipage}[b]{0.32\textwidth}
\centering
\includegraphics[width=1.6in]{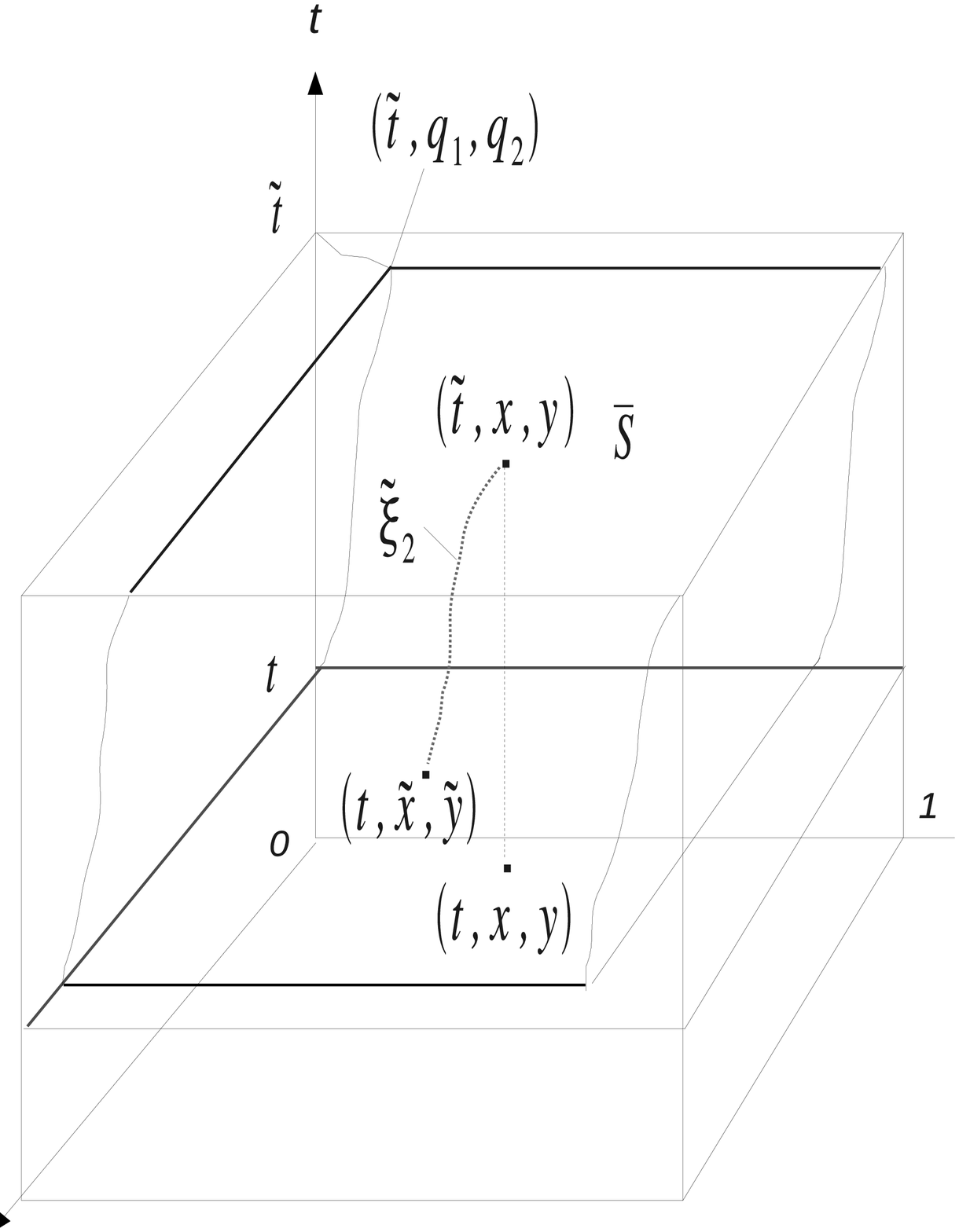}
\end{minipage}}%
\subfigure[]{
\begin{minipage}[t]{0.32\textwidth}
\centering
\includegraphics[width=1.6in]{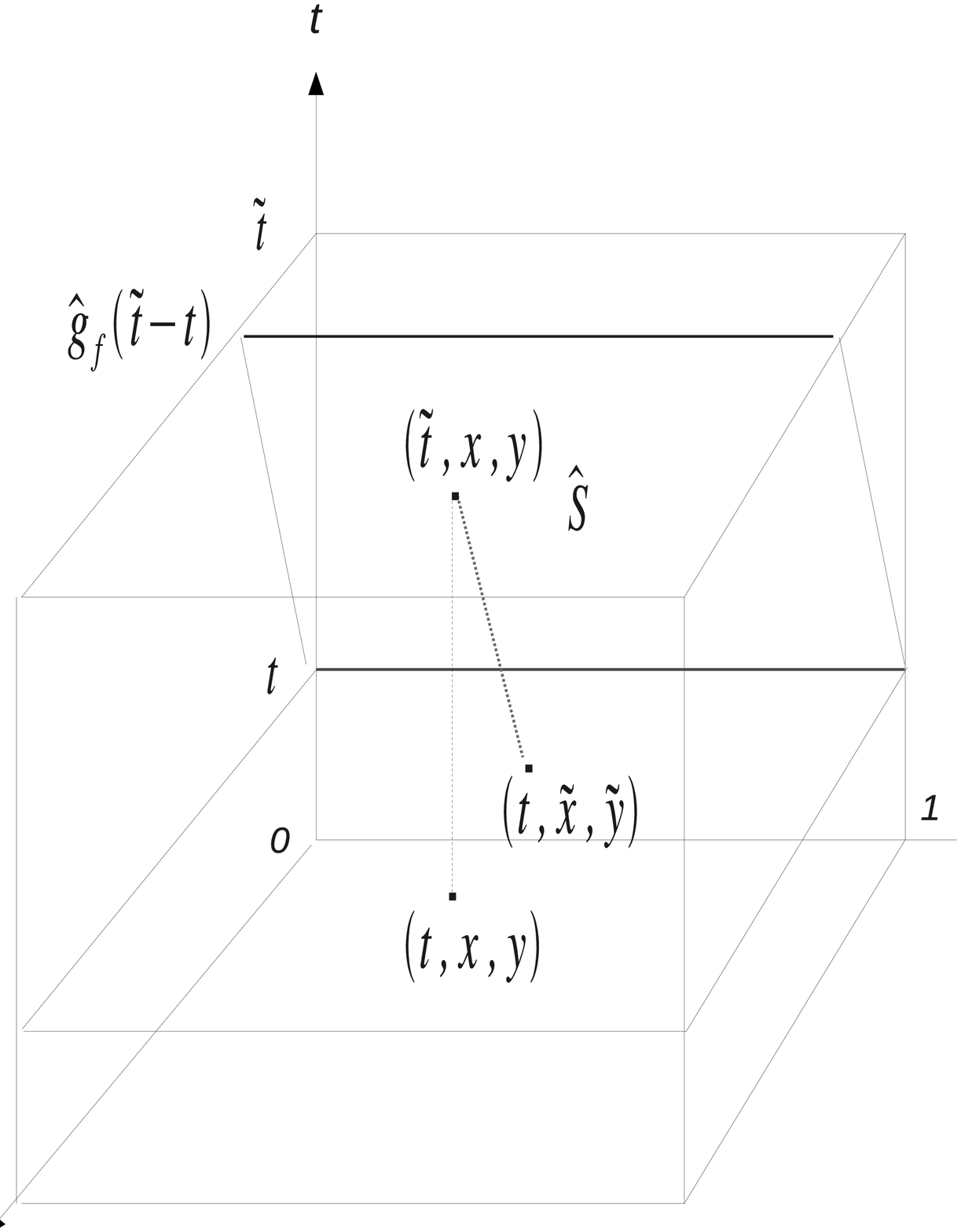}
\end{minipage}}%
\subfigure[]{
\begin{minipage}[t]{0.32\textwidth}
\centering
\includegraphics[width=1.6in]{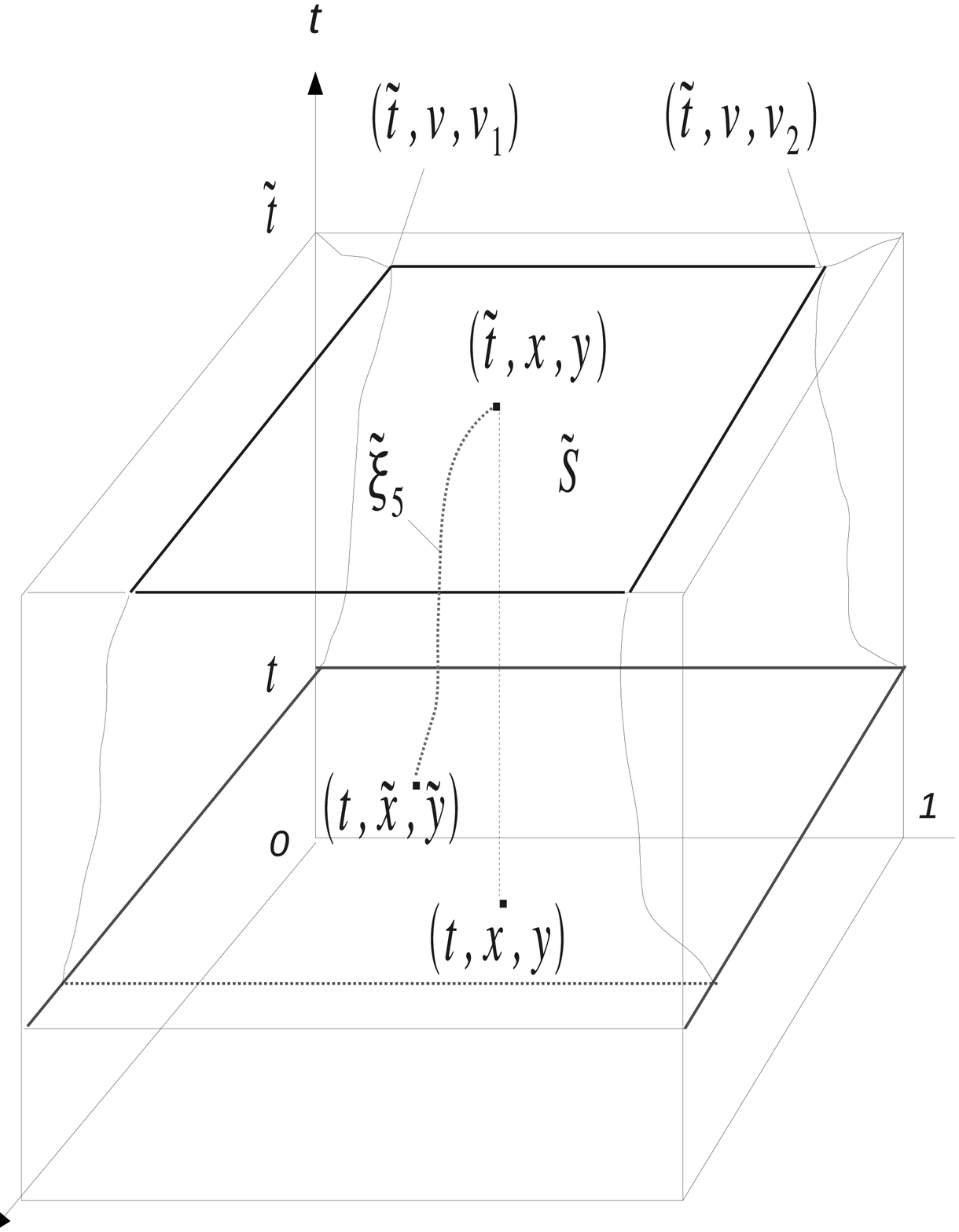}
\end{minipage}}
\caption{(a) For Phase 1, the definition of $\ov S$ and characteristic $\widetilde \xi_2$ which connects $(\tilde t,x,y)$ with $(t,\tilde x,\tilde y)$,
characteristic passing through $(t,0,0)$ intersects time $\tilde t$ plane at $(\tilde t,q_1,q_2)$;
(b) For Phase 2, the definition of $\widehat S$;
(c) For Phase 3, the definition of $\widetilde S$ and characteristic $\widetilde \xi_5$ which connects $(\tilde t,x,y)$ with $(t,\tilde x,\tilde y)$, characteristic 
passing through $(t,0,0)$ intersects time $\tilde t$ plane at $(\tilde t,v,v_1)$ and characteristic 
passing through $(t,0,1)$ intersects time $\tilde t$ plane at $(\tilde t,v,v_2)$.}
\label{Fig 6} 
\end{figure}

Next we estimate $\|\widehat\phi^f_k(\tilde t,\cdot)-\widehat\phi^f_k(t,\cdot)\|_{L^p((0,1)^2)}$. For Phase 2, the characteristic passing through
$(t,0,0)$ intersects time $\tilde t$ plane at $(\tilde t,\widehat g_f(\tilde t-t),0)$.
Let $\widehat S:=[\widehat g_f(\tilde t-t),1]\times[0,1]$ (see Fig 5 (b)), when $(x,y)\in \widehat S$,
we have
\begin{equation*}
|\widehat\phi^f_k(\tilde t,x,y)-\widehat\phi^f_k(t,x,y)|\!\!=\!\!|\widehat\phi^f_k(t,\tilde x,\tilde y)-\widehat\phi^f_k(t,x,y)|
\!\!=\!\!|\widehat\phi^f_k(t,x-\widehat g_f(\tilde t-t),y)-\widehat\phi^f_k(t,x,y)|.
\end{equation*}
Since $\widehat\phi^f_k\in{L^{\infty}((0,\delta)}\times(0,1)^2)$, for every $l>0$, there exists $C_l$ such that for every $t\in[0,\delta]$, there exists $\widehat\phi^{fl}_k(t,\cdot)\in C^1([0,1]^2)$ satisfying 
\be\label{approximate2}
\|\widehat\phi^{fl}_k(t,\cdot)-\widehat\phi^f_k(t,\cdot)\|_{L^p((0,1)^2)}\leq \f{1}{l},\quad \|\widehat\phi^{fl}_k(t,\cdot)\|_{C^1([0,1]^2)}\leq C_l.
\ee
Similar to \eqref{estimate1}, we can prove that
\be \label{estimate3}
\iint\limits_{\widehat S}|\widehat\phi^f_k(\tilde t,x,y)-\widehat\phi^f_k(t,x,y)|^pdxdy\leq \f{C}{l^p}+C_l|\tilde t-t|^p.
\ee
It is easy to check that
\be \label{estimate4}
\quad \iint\limits_{[0,1]^2\backslash\widehat S}|\widehat\phi^f_k(\tilde t,x,y)-\widehat\phi^f_k(t,x,y)|^pdxdy\leq C|\tilde t-t|.
\ee
Combining \eqref{estimate3} with \eqref{estimate4}, we get
\be\label{ovestimate2}
\int_0^1\int_0^1|\widehat\phi^f_k(\tilde t,x,y)-\widehat\phi^f_k(t,x,y)|^pdxdy
\leq \f{C}{l^p} +C_l|\tilde t-t|^p+C|\tilde t-t|.
\ee

Finally, we estimate $\|\widetilde\phi^f_k(\tilde t,\cdot)-\widetilde\phi^f_k(t,\cdot)\|_{L^p((0,1)^2)}$. 
Suppose that the characteristic passing through $(t,0,0)$ intersects time $\tilde t$ plane at $(\tilde t,v,v_1)$, and 
the characteristic passing through $(t,0,1)$ intersects time $\tilde t$ plane at $(\tilde t,v,v_2)$ (see Fig 5 (c)).
Similar to \eqref{m}, we can prove that
\be \label{m2}
v\leq C|\tilde t-t|,\quad v_1\leq C|\tilde t-t|,\quad 1-v_2\leq C|\tilde t-t|.
\ee
Let $\widetilde S:=[v,1]\times[v_1,v_2]$ (see Fig 5 (c)), when $(x,y)\in \widetilde S$,
we have 
\begin{equation*}
|\widetilde\phi^f_k(\tilde t,x,y)-\widetilde\phi^f_k(t,x,y)|=|\widetilde\phi^f_k(t,\tilde x,\tilde y)
e^{-\int_t^{\tilde t}[\widetilde\lambda(\widetilde y_5,U)+\f{\pa\widetilde h_f}{\pa y}(\widetilde y_5,u_f)](\sigma)d\sigma}-\widetilde\phi^f_k(t,x,y)|.
\end{equation*}
Here $\widetilde y_5=(\widetilde x_5,\widetilde y_5)$ denotes the characteristic passing through $(\tilde t,x,y)$ that intersects
time $t$ plane at $(t,\tilde x,\tilde y)$. Similar to \eqref{xy}, we can prove that
\begin{equation*}
|\tilde x-x|\leq C|\tilde t-t|,\quad |\tilde y-y|\leq C|\tilde t-t|.
\end{equation*}
Since $\widetilde\phi^f_k\in{L^{\infty}((0,\delta)}\times(0,1)^2)$, for every $l>0$, there exists $C_l$ such that for every $t\in[0,\delta]$, there exists $\widetilde\phi^{fl}_k(t,\cdot)\in C^1([0,1]^2)$ satisfying 
\be\label{approximate3}
\|\widetilde\phi^{fl}_k(t,\cdot)-\widetilde\phi^f_k(t,\cdot)\|_{L^p((0,1)^2)}\leq \f{1}{l},\quad \|\widetilde\phi^{fl}_k(t,\cdot)\|_{C^1([0,1]^2)}\leq C_l.
\ee
Similar to \eqref{estimate1}, we can prove that
\be \label{estimate5}
\iint\limits_{\widetilde S}|\widetilde\phi^f_k(\tilde t,x,y)-\widetilde\phi^f_k(t,x,y)|^pdxdy\leq \f{C}{l^p}+C_l|\tilde t-t|^p.
\ee
Noting \eqref{m2}, we have 
\be \label{estimate6}
\iint\limits_{[0,1]^2\backslash\widetilde S}|\widetilde\phi^f_k(\tilde t,x,y)-\widetilde\phi^f_k(t,x,y)|^pdxdy\leq C(v+v_1+1-v_2)\leq C|\tilde t-t|.
\ee
Combining \eqref{estimate5} with \eqref{estimate6}, we get
\be\label{ovestimate3}
\int_0^1\int_0^1|\widetilde\phi^f_k(\tilde t,x,y)-\widetilde\phi^f_k(t,x,y)|^pdxdy
\leq \f{C}{l^p} +C_l|\tilde t-t|^p+C|\tilde t-t|.
\ee
Therefore letting $l$ be large enough and then $|\tilde t-t|$ be small enough, the right hand side of \eqref{ovestimate1}, \eqref{ovestimate2}
and \eqref{ovestimate3} can be arbitrarily small.
Above all, we prove that the vector function $\vec\phi_f$ belongs to $C^0([0,\delta];L^p((0,1)^2))$ for all $p\in[1,\infty)$.
\end{proof}

Let us recall Definition \ref{weaksol} of a weak solution, we prove that the vector function $\vec\phi_f$ 
defined by \eqref{for1}-\eqref{so3} is indeed a weak solution to Cauchy problem \eqref{eq}-\eqref{bcy1}.

\section{Uniqueness of the solution}
In this section, we prove the uniqueness of the solution.
\begin{lem}\label{uni}
The weak solution to Cauchy problem \eqref{eq}-\eqref{bcy1} is unique.
\end{lem}

\begin{proof}
Assume that $\ov{\vec\phi}_f\!\!=\!(\ov{\ov\phi}^f_1,\cdots,\ov{\ov\phi}^f_N,\ov{\widehat\phi}^f_1,\cdots,\ov{\widehat\phi}^f_N,
\ov{\widetilde\phi}^f_1,\cdots,\ov{\widetilde\phi}^f_N)^T\!\!\in\!{C^0([0,\delta];L^1((0,1)^2))}$ is a weak solution to 
Cauchy problem \eqref{eq}-\eqref{bcy1}. Similar to Lemma 2.2 in \cite[Section 2.2]{Wang}, we can prove that for any fixed $t\in[0,\delta]$ and any $\vec{\Phi}(\tau,x,y):=(\ov\Phi_1,\cdots,\ov\Phi_N,\\\widehat\Phi_1,\cdots,\widehat\Phi_N,
\widetilde\Phi_1,\cdots,\widetilde\Phi_N)^T\in{C^1([0,t]\times[0,1]^2)}$ 
with 
\begin{align*}
 &\vec{\Phi}(\tau,1,y)= 0,\quad \forall (\tau,y)\in[0,t]\times[0,1],\nonumber\\
 &\ov\Phi_1(\tau,0,y)=\widetilde\Phi_1(\tau,0,y)=0,\quad \forall (\tau,y)\in[0,t]\times[0,1],\nonumber\\
 &\ov\Phi_k(\tau,x,0)=\widehat\Phi_k(\tau,x,0)=\ov\Phi_k(\tau,x,1)=\widehat\Phi_k(\tau,x,1)=0,\quad \forall (\tau,x)\in[0,t]\times[0,1],
 \end{align*}
one has
 \begin{align}\label{note}
 &\int_0^t\!\!\int_0^1\!\!\int_0^1 \ov{\vec\phi}_f(\tau,x,y)\cdot(\vec{\Phi}_{\tau}(\tau,x,y)
 \!\!+\!\!\ov A_f\vec{\Phi}_x(\tau,x,y)\!\!+\!\!\ov B_f\vec{\Phi}_y(\tau,x,y)\!\!+\!\ov C\vec{\Phi}(\tau,x,y)) dx dy d\tau
\\
  &\!\!+\!\!\int_0^1\!\!\int_0^1 \!\!\vec\phi_{f0}(x,y)\cdot\vec{\Phi}(0,x,y)dxdy\!\!+\!\!\sum_{k=1}^N\int_0^t\!\!\int_0^{\f{a_1}{a_2}}\widetilde h_f(0,\ov u_f(\tau))\ov{\ov\phi}^f_k(\tau,\f{a_2}{a_1}x,1)\widetilde\Phi_k(t,x,0)dxd\tau\nonumber\\
 &\!\!+\!\!\sum_{k=2}^N\int_0^t\!\!\int_0^1 \f{2\tau_{gf}}{a_1}\ \ov{\widehat\phi}^f_{k-1}(\tau,1,y)\ov\Phi_k(\tau,0,y)dyd\tau \!\!+\!\!\sum_{k=2}^N\int_0^t\!\!\int_0^1 \!\!\widetilde g_f\ov{\widetilde\phi}^f_{k-1}(\tau,1,y)\widetilde\Phi_k(\tau,0,y)dyd\tau\nonumber\\
    &\!\!+\sum_{k=1}^N \int_0^t\!\!\int_0^1 \f{a_1\widehat g_f\ov g_f(\ov u_f(\tau))}{\tau_{gf}}\ \ov{\ov\phi}^f_k(\tau,1,y)\widehat\Phi_k(\tau,0,y)dyd\tau =\int_0^1\int_0^1\ov{\vec\phi}_f(t,x,y)\cdot\vec\Phi(t,x,y)  \nonumber .
 \end{align}
In \eqref{note}, the velocity matrices $\ov A_f$, $\ov B_f$ and $\ov C$ are defined the same way as $A_f$, $B_f$ and $C$
but with $\ov M_f$ and $\ov M$ instead of $M_f$ and $M$,
where
\begin{align*}\ov M_f(t)=& \sum_{k=1}^N\int_0^1\int_0^1 a_1\gamma_s^2y
\ov{\ov\phi}^f_k(t,x,y)\,dx\,dy+\sum_{k=1}^N\int_0^1\int_0^1(a_2-a_1)\gamma^2_s y\ov{\widehat\phi}^f_k(t,x,y)\,dx\,dy\\
&+\sum_{k=1}^N\int_0^1\int_0^1a_2\gamma_0(\gamma_0 y +\gamma_s)\ov{\widetilde\phi}^f_k(t,x,y)\,dx\,dy ,
\end{align*}
and $\ov M(t)\!\!=\!\!\sum_{f=1}^{n}\ov M_f(t)$.
Let $\vec{\varphi}_0(x,y)\!\!=\!\!(\ov{\varphi}_{01},\cdots,\ov{\varphi}_{0N},\widehat\varphi_{01},\cdots,\widehat\varphi_{0N},
\widetilde{\varphi}_{01},\cdots,\widetilde{\varphi}_{0N})^T\!\in{C^1_0((0,1)^2)}$, and choose the test function as the solution
to the following backward linear Cauchy problem
\be \label{back}
\left\{
\begin{array}{l}
\vec{\Phi}_{\tau}+\ov A_f\vec{\Phi}_x+\ov B_f\vec{\Phi}_y=-\ov C\vec{\Phi},\quad 0\leq\tau\leq t,\quad (x,y)\in[0,1]^2,\\
\vec{\Phi}(t,x,y)=\vec{\varphi}_0(x,y),\quad (x,y)\in[0,1]^2,\\
\vec{\Phi}(\tau,1,y)=0,\quad \forall (\tau,y)\in[0,t]\times[0,1],\\
\ov\Phi_1(\tau,0,y)=\widetilde\Phi_1(\tau,0,y)=0,\quad \forall (\tau,y)\in[0,t]\times[0,1],\\
\ov\Phi_k(\tau,x,0)\!\!=\!\!\widehat\Phi_k(\tau,x,0)\!\!=\!\!\ov\Phi_k(\tau,x,1)\!\!=\!\!\widehat\Phi_k(\tau,x,1)\!\!=\!\!0,\forall (\tau,x)\in[0,t]\times[0,1].
\end{array}
\right.
\ee
For any fixed $t\in[0,\delta]$, we introduce three new subsets $\ov{\ov\omega}^{f,t}_1$, $\ov{\ov\omega}^{f,t}_2$ and $\ov{\ov\omega}^{f,t}_3$ of $[0,1]^2$
\begin{align*}
&\ov{\ov\omega}^{f,t}_1:=\Big\{(x,y)|\ \int_0^t\ov g_f(\ov u_f(\sigma))\, d\sigma\leq x\leq 1,\ \ov{\ov\eta}(t,\int_0^t\ov g_f(\ov u_f(\sigma))\, d\sigma)\leq y\leq 1\Big\},\\
&\ov{\ov\omega}^{f,t}_2:=\Big\{(x,y)|\ 0\leq x\leq \int_0^t \ov g_f(\ov u_f(\sigma))d\sigma,\ \ov{\ov\eta}(t,x)\leq y\leq 1\Big\},\\
&\ov{\ov\omega}^{f,t}_3:=[0,1]^2\backslash(\ov{\ov\omega}^{f,t}_1\cup\ov{\ov\omega}^{f,t}_2).
\end{align*}
Here $y=\ov{\ov\eta}(t,x)$ satisfies
\begin{equation*}
\displaystyle\f{d\ov{\ov\eta}}{ds}=\ov h_f(\ov{\ov\eta},\ov u_f)(s),\quad \ov{\ov\eta}(\ov\theta)=0,\quad \ov\theta\leq s\leq t,
\end{equation*}
with $\ov\theta$ defined by $x=\displaystyle\int_{\ov\theta}^t\ov g_f(\ov u_f(\sigma))d\sigma$.

For any fixed $t\in[0,\delta]$ and $(x,y)\in[0,1]^2$. If $(x,y)\in \ov{\ov\omega}^{f,t}_1$, we define $\ov\xi_2=(\ov x_2, \ov y_2)$ by
\begin{equation*} 
\displaystyle\f{d\ov x_2}{ds}=\ov g_f(\ov u_f(s)),\quad 
\displaystyle\f{d\ov y_2}{ds}=\ov h_f(\ov y_2,\ov u_f)(s),\quad \ov \xi_2(t)=(x,y).
\end{equation*}
Let us then define 
\be\label{barx0}
(\ov x_0,\ov y_0):=(\ov x_2(0),\ov y_2(0)).
\ee
If $(x,y)\in\ov{\ov\omega}^{f,t}_2$, we define $\ov\xi_3=(\ov x_3, \ov y_3)$ by
\begin{equation*} 
\displaystyle\f{d\ov x_3}{ds}=\ov g_f(\ov u_f(s)),\quad 
\displaystyle\f{d\ov y_3}{ds}=\ov h_f(\ov y_3,\ov u_f)(s),\quad \ov \xi_3(t)=(x,y).
\end{equation*}
There exists a unique $\ov \tau_0$ such that $\ov x_3(\ov \tau_0)=0$, so that we can define
\be\label{barbeta0}
 \ov\beta_0:=\ov y_3(\ov \tau_0).
\ee
By \eqref{note} and \eqref{back}, we obtain that
\begin{align}\label{ou}
&\int_0^1\int_0^1\ov{\vec\phi}_f(t,x,y)\cdot\vec\varphi^f_0(x,y)dxdy\\
=\!\!&\int_0^1\!\!\int_0^1 \!\!\vec\phi_{f0}(x,y)\cdot\vec{\Phi}(0,x,y)dxdy\!\!+\!\!\sum_{k=1}^N\!\!\int_0^t\!\!\int_0^{\f{a_1}{a_2}}\!\!\widetilde h_f(0,\ov u_f(\tau))\ov{\ov\phi}^f_k(\tau,\f{a_2}{a_1}x,1)\widetilde\Phi_k(t,x,0)dxd\tau\nonumber\\
 &+\!\!\sum_{k=2}^N\!\!\int_0^t\!\!\int_0^1 \f{2\tau_{gf}}{a_1}\ \ov{\widehat\phi}^f_{k-1}(\tau,1,y)\ov\Phi_k(\tau,0,y)dyd\tau\!\! +\!\!\sum_{k=2}^N\!\!\int_0^t\!\!\int_0^1\!\! \widetilde g_f\ov{\widetilde\phi}^f_{k-1}(\tau,1,y)\widetilde\Phi_k(\tau,0,y)dyd\tau\nonumber\\
    &+\sum_{k=1}^N \int_0^t\int_0^1 \f{a_1\widehat g_f\ov g_f(\ov u_f(\tau))}{\tau_{gf}}\ \ov{\ov\phi}^f_k(\tau,1,y)\widehat\Phi_k(\tau,0,y)dyd\tau.\nonumber
  \end{align}
Let us recall the definition of the characteristic $\ov\xi_2$ and also the definition \eqref{barx0} of $(\ov x_0,\ov y_0)$.
By solving the backward linear Cauchy problem \eqref{back} and noting \eqref{jy2} of Lemma \ref{jacobi} in Appendix 6.3, we have
\begin{align}\label{lastt}
&\int_0^1\int_0^1\vec\phi_{f0}(x,y)\cdot\vec\Phi(0,x,y)dxdy\nonumber\\
=&\sum_{k=1}^N\iint\limits_{\ov{\ov\omega}^{f,t}_1}
\ov\phi^f_{k0}(\ov x_0,\ov y_0)\ov\varphi_{0k}(x,y)e^{-\int_0^t[\ov\lambda(\ov y_2,\ov U)+\f{\pa\ov h_f}{\pa y}(\ov y_2,\ov u_f)](\sigma)\,d\sigma}dxdy\nonumber\\
&+\!\!\int_0^1\!\!\int_{\widehat g_ft}^1\!\!\widehat\phi^f_{k0}(x\!\!-\!\!\widehat g_ft,y)\widehat\varphi_{0k}(x,y)dxdy\!\!+\!\!
\sum_{k=1}^N\!\!\int_0^1\!\!\int_0^1\widetilde\phi^f_{k0}(\ov x_0,\ov y_0)\widetilde\Phi_k(0,\ov x_0,\ov y_0)d\ov x_0d\ov y_0.
\end{align}
As for the last term of \eqref{ou}, by solving the backward linear Cauchy problem \eqref{back}, we have
\begin{align*}
&\sum_{k=1}^N \int_0^t\int_0^1 \f{a_1\widehat g_f\ov g_f(\ov u_f(\tau))}{\tau_{gf}}\ \ov{\ov\phi}^f_k(\tau,1,y)\widehat\Phi_k(\tau,0,y)dyd\tau\\
=&\sum_{k=1}^N \int_0^1\int_0^{\widehat g_ft} \f{a_1\ov g_f(\ov u_f(t-\f{x}{\widehat g_f}))}{\tau_{gf}}\ \ov{\ov\phi}^f_k(t-\f{x}{\widehat g_f},1,y)\widehat\varphi_{0k}(x,y)dxdy.
\end{align*}
Since $\vec{\varphi}_0\in{C^1_0((0,1)^2)}$ and $t\in[0,\delta]$ are both arbitrary, we obtain in $C^0([0,\delta];L^1((0,1)^2))$ that
for $k=1,\cdots,N$
\be\label{hat1}
\ov{\widehat\phi}^f_k(t,x,y)=\left\{
\begin{array}{l}
\displaystyle\f{a_1\ov g_f(u_f(t-\f{x}{\widehat g_f}))}{\tau_{gf}}\ \ov{\ov\phi}^f_k(t-\f{x}{\widehat g_f},1,y),
\quad \text{if}\ (x,y)\in[0,\widehat g_ft]\times[0,1],\\
\ov{\widehat\phi}^f_k(t,x,y)=\widehat\phi^f_{k0}(x-\widehat g_ft,y)=\widehat\phi^f_k(t,x,y),\quad \text{if}\ 
(x,y)\in[\widehat g_ft,1]\times[0,1].
\end{array}
\right.
\ee
Let us recall the definition of the characteristic $\ov\xi_3$ and also the definition \eqref{barbeta0} of $(\ov\tau_0,\ov\beta_0)$.
By solving the backward linear Cauchy problem \eqref{back}, noting \eqref{hat1} and \eqref{jy3} of Lemma \ref{jacobi} in Appendix 6.3, we have
\begin{align*}
&\sum_{k=2}^N\int_0^t\int_0^1\f{2\tau_{gf}}{a_1}\ \ov{\widehat\phi}^f_{k-1}(\tau,1,y)\ov\Phi_k(\tau,0,y)dyd\tau\\
=&\sum_{k=2}^N\int_0^t\int_0^1\f{2\tau_{gf}}{a_1}\ \widehat\phi^f_{k-1}(\ov\tau_0,1,\ov\beta_0)\ov\Phi_k(\ov\tau_0,0,\ov\beta_0)d\ov\beta_0d\ov\tau_0\\
=&\sum_{k=2}^N\iint\limits_{\ov{\ov\omega}^{f,t}_2}\f{2\tau_{gf}\widehat\phi^f_{k-1}(\ov\tau_0,1,\ov\beta_0)}{a_1\ov g_f(\ov u_f(\ov\tau_0))}\ \ov\varphi_{0k}(x,y)
e^{-\int_{\ov\tau_0}^t[\ov\lambda(\ov y_3,\ov U)+\f{\pa \ov h_f}{\pa y}(\ov y_3,\ov u_f)]}dxdy.\\
\end{align*}
Since $\vec{\varphi}_0\in{C^1_0((0,1)^2)}$ and $t\in[0,\delta]$ are both arbitrary, we obtain in $C^0([0,\delta];L^1((0,1)^2))$ that for $k=1$ 
\be\label{solution21}
\ov{\ov\phi}^f_1(t,x,y)=\left\{
\begin{array}{l}
\ov\phi^f_{10}(\ov x_0,\ov y_0 )e^{-\int_0^t[\ov\lambda(\ov y_2,\ov U)+\f{\pa\ov h_f}{\pa y}(\ov y_2,\ov u_f)](\sigma)\,d\sigma},\quad \text{if}\ (x,y)\in\ov{\ov\omega}^{f,t}_1,\\
0,\quad \text{else}.
\end{array}
\right.
\ee
For $k=2,\cdots,N$, we have
\be\label{solution22}
\ov{\ov\phi}^f_k(t,x,y)=\left\{
\begin{array}{l}
\ov\phi^f_{k0}(\ov x_0,\ov y_0 )e^{-\int_0^t[\ov\lambda(\ov y_2,\ov U)+\f{\pa\ov h_f}{\pa y}(\ov y_2,\ov u_f)](\sigma)\,d\sigma},\quad \text{if}\ (x,y)\in \ov{\ov\omega}^{f,t}_1,\\
\displaystyle\f{2\tau_{gf}\widehat\phi_{k-1}(\ov\tau_0,1,\ov\beta_0)}{a_1\ov g_f(\ov u_f(\ov\tau_0))}\ 
e^{-\int_{\ov\tau_0}^t[\ov\lambda(\ov y_3,\ov U)+\f{\pa \ov h_f}{\pa y}(\ov y_3,\ov u_f)]},\quad \text{if}\ (x,y)\in \ov{\ov\omega}^{f,t}_2,\\
0,\quad \text{else}.
\end{array}
\right.
\ee
For any fixed $t\in[0,\delta]$, we introduce four new subsets $\ov{\widetilde\omega}^{f,t}_1$, $\ov{\widetilde\omega}^{f,t}_2$, $\ov{\widetilde\omega}^{f,t}_3$ and $\ov{\widetilde\omega}^{f,t}_4$ of $[0,1]^2$
\begin{align*}
&\ov{\widetilde\omega}^{f,t}_1\!\!:=\!\!\Big\{(x,y)|\ \widetilde g_ft\leq x\leq 1,\ \ov{\widetilde\eta}_1(t,\widetilde g_ft)\leq y\leq \ov{\widetilde\eta}_2(t,\widetilde g_ft)\Big\},\\
&\ov{\widetilde\omega}^{f,t}_2\!\!:=\!\!\Big\{(x,y)|\ 0\leq x\leq \widetilde g_ft,\! 0\leq y\leq \ov{\widetilde\eta}_1(t,x)\Big\}
\!\cup\!\Big\{(x,y)|\ \widetilde g_ft\leq x\leq 1,\! 0\leq y\leq \ov{\widetilde\eta}_1(t,\widetilde g_ft)\Big\},\\
&\ov{\widetilde\omega}^{f,t}_3\!\!:=\!\!\Big\{(x,y)|\ 0\leq x\leq \widetilde g_ft,\ \ov{\widetilde\eta}_1(t,x)\leq y\leq\ov{\widetilde\eta}_2(t,x)\Big\},\\
&\ov{\widetilde\omega}^{f,t}_4\!\!:=\!\![0,1]^2\backslash(\ov{\widetilde\omega}^{f,t}_1\cup\ov{\widetilde\omega}^{f,t}_2\cup\ov{\widetilde\omega}^{f,t}_3).
\end{align*} 
Here
$y=\ov{\widetilde\eta}_1(t,x)$ and $y=\ov{\widetilde\eta}_2(t,x)$ satisfy
\begin{align*}
&\displaystyle\f{d\ov{\widetilde\eta}_1}{ds}=\widetilde h_f(\ov{\widetilde\eta}_1,\ov u_f)(s),\quad \ov{\widetilde\eta}_1(t-\f{x}{\widetilde g_f})=0,
\quad t-\f{x}{\widetilde g_f}\leq s\leq t,\\
&\f{d\ov{\widetilde\eta}_2}{ds}=\widetilde h_f(\ov{\widetilde\eta}_2,\ov u_f)(s),\quad \ov{\widetilde\eta}_2(t-\f{x}{\widetilde g_f})=1,
\quad t-\f{x}{\widetilde g_f}\leq s\leq t.
\end{align*}
For any fixed $t\in[0,\delta]$ and $(x,y)\in[0,1]^2$. If $(x,y)\in \ov{\widetilde\omega}^{f,t}_1$, we define $\ov\xi_5=(\ov x_5, \ov y_5)$ by
\begin{equation*} 
\displaystyle\f{d\ov x_5}{ds}=\widetilde g_f,\quad 
\displaystyle\f{d\ov y_5}{ds}=\widetilde h_f(\ov y_5,\ov u_f)(s),\quad \ov \xi_5(t)=(x,y).
\end{equation*}
Let us then define $(\ov x_0,\ov y_0):=(\ov x_5(0),\ov y_5(0))$. By solving the backward linear Cauchy problem \eqref{back}
and noting \eqref{jy5} of Lemma \ref{jacobi} in Appendix 6.3, the last term of \eqref{lastt} can be rewritten as 
\begin{align*}
&\sum_{k=1}^N\int_0^1\int_0^1\widetilde\phi^f_{k0}(\ov x_0,\ov y_0)\widetilde\Phi_k(0,\ov x_0,\ov y_0)d\ov x_0d\ov y_0\\
=&\sum_{k=1}^N\iint\limits_{\ov{\widetilde\omega}^{f,t}_1}
\widetilde\phi^f_{k0}(\ov x_0,\ov y_0)\widetilde\varphi_{0k}(x,y)e^{-\int_0^t[\ov\lambda(\ov y_5,\ov U)+\f{\pa\ov h_f}{\pa y}(\ov y_5,\ov u_f)](\sigma)\,d\sigma}dxdy.
\end{align*}
If $(x,y)\in \ov{\widetilde\omega}^{f,t}_2$, we define $\ov\xi_6=(\ov x_6, \ov y_6)$ by
\begin{equation*} 
\displaystyle\f{d\ov x_6}{ds}=\widetilde g_f,\quad
\displaystyle\f{d\ov y_6}{ds}=\widetilde h_f(\ov y_6,\ov u_f)(s),\quad \ov \xi_6(t)=(x,y).
\end{equation*}
There exists a unique $\ov t_0$ such that $\ov y_6(\ov t_0)=0$, so that we can define $\ov\alpha_0:=\ov x_6(\ov t_0)$. 
By solving the backward linear Cauchy problem \eqref{back}, and noting \eqref{jy6} of Lemma \ref{jacobi} in Appendix 6.3, we have
\begin{align*}
&\sum_{k=1}^N\int_0^t\int_0^{\f{a_1}{a_2}}\widetilde h_f(0,\ov u_f(\tau))\ov{\ov\phi}^f_k(\tau,\f{a_2}{a_1}x,1)\widetilde\Phi_k(\tau,x,0)dxd\tau\\
=&\sum_{k=1}^N\int_0^t\int_0^{\f{a_1}{a_2}}\widetilde h_f(0,\ov u_f(\ov t_0))\ov{\ov\phi}^f_k(\ov t_0,\f{a_2}{a_1}\ov\alpha_0,1)\widetilde\Phi_k(\ov t_0,\ov\alpha_0,0)d\ov\alpha_0 d\ov t_0\\
=&\sum_{k=1}^N\iint\limits_{\ov{\widetilde\omega}^{f,t}_2}\ov{\ov\phi}^f_k(\ov t_0,\f{a_2}{a_1}\ov\alpha_0,1)\widetilde\varphi_{0k}(x,y)e^{-\int_{\ov t_0}^t[\widetilde \lambda(\ov y_6,\ov U)
+\f{\pa\widetilde h_f}{\pa y}(\ov y_6,\ov u_f)](\sigma)\, d\sigma}dxdy.
\end{align*}
If $(x,y)\in \ov{\widetilde\omega}^{f,t}_3$, we define $\ov\xi_9=(\ov x_9, \ov y_9)$ by 
\begin{equation*} 
\displaystyle\f{d\ov x_9}{ds}=\widetilde g_f,\quad 
\displaystyle\f{d\ov y_9}{ds}=\widetilde h_f(\ov y_9,\ov u_f)(s),\quad \ov \xi_9(t)=(x,y).
\end{equation*} 
There exists a unique $\ov\tau_0$ such that $\ov x_9(\ov\tau_0)=0$, so that we can define $\ov\beta_0:=\ov y_9(\ov\tau_0)$. 
By solving the backward linear Cauchy problem \eqref{back} and noting \eqref{jy9} of Lemma \ref{jacobi} in Appendix 6.3, we have
\begin{align*}
&\sum_{k=2}^N\int_0^t\int_0^1 \widetilde g_f\ov{\widetilde\phi}^f_{k-1}(\tau,1,y)\widetilde\Phi_k(\tau,0,y)dyd\tau\\
=&\sum_{k=2}^N\int_0^t\int_0^1 \widetilde g_f\ov{\widetilde\phi}^f_{k-1}(\ov\tau_0,1,\ov\beta_0)\widetilde\Phi_k(\ov\tau_0,0,\ov\beta_0)d\ov\beta_0d\ov\tau_0\\
=&\sum_{k=2}^N\iint\limits_{\ov{\widetilde\omega}^{f,t}_3}
\ov{\widetilde\phi}^f_{k-1}(\ov\tau_0,1,\ov\beta_0)\widetilde\varphi_{0k}(x,y)
e^{-\int_{\ov\tau_0}^t[\widetilde\lambda(\ov y_9,\ov U)+\f{\pa\widetilde h_f}{\pa y}(\ov y_9,\ov u_f)](\sigma)\,d\sigma}dxdy.
\end{align*}
Since $\vec{\varphi}_0\in{C^1_0((0,1)^2)}$ and $t\in[0,\delta]$ are both arbitrary, we obtain in $C^0([0,\delta];L^1((0,1)^2))$ 
that for $k=1$
\be \label{solution23}
\ov{\widetilde\phi}^f_1(t,x,y)=\left\{
\begin{array}{l}
\ov{\widetilde\phi}^f_{10}(\ov x_0,\ov y_0)e^{-\int_0^t[\widetilde\lambda(\ov y_5,\ov U)
+\f{\pa\widetilde h_f}{\pa y}(\ov y_5,\ov u_f)](\sigma)\, d\sigma},\quad \text{if}\ (x,y)\in \ov{\widetilde\omega}^{f,t}_1,\\
\ov{\ov\phi}^f_1(\ov t_0,\displaystyle\f{a_2}{a_1}\ov\alpha_0, 1)e^{-\int_{\ov t_0}^t[\widetilde \lambda(\ov y_6,\ov U)
+\f{\pa\widetilde h_f}{\pa y}(\ov y_6,\ov u_f)](\sigma)\, d\sigma},\quad \text{if}\ (x,y)\in \ov{\widetilde\omega}^{f,t}_2,\\
0,\quad \text{else}.
\end{array}
\right.
\ee
For $k=2,\cdots,N$, we get
\begin{align} \label{solution24}
\ov{\widetilde\phi}^f_k(t,x,y)=\left\{
\begin{array}{l}
\ov{\widetilde\phi}^f_{k0}(\ov x_0,\ov y_0)e^{-\int_0^t[\widetilde\lambda(\ov y_5,\ov U)
+\f{\pa\widetilde h_f}{\pa y}(\ov y_5,\ov u_f)](\sigma)\, d\sigma},\quad \text{if}\ (x,y)\in \ov{\widetilde\omega}^{f,t}_1,\\
\ov{\ov\phi}^f_1(\ov t_0,\displaystyle\f{a_2}{a_1}\ov\alpha_0, 1)e^{-\int_{\ov t_0}^t[\widetilde \lambda(\ov y_6,\ov U)
+\f{\pa\widetilde h_f}{\pa y}(\ov y_6,\ov u_f)](\sigma)\, d\sigma},\quad \text{if}\ (x,y)\in \ov{\widetilde\omega}^{f,t}_2,
\\
\ov{\widetilde\phi}^f_{k-1}(\ov\tau_0, 1,\ov\beta_0)e^{-\int_{\ov\tau_0}^t[\widetilde \lambda(\ov y_9,\ov U)+\f{\pa\widetilde h_f}{\pa y}(\ov y_9,\ov u_f)](\sigma)\,d\sigma},\quad \text{if}\ (x,y)\in \ov{\widetilde\omega}^{f,t}_3,\\
0,\quad \text{else}.
\end{array}
\right.
\end{align}

By \eqref{delta}, we claim that $\ov{\vec M}(t)\in\Omega_{\delta,K}$ and $\ov{\vec M}(t)$ satisfies the same contraction mapping function $\vec G$.
Since $\vec M$ is the unique fixed point of $\vec G$ in $\Omega_{\delta,K}$, therefore we have $\ov{\vec M} \equiv \vec M$.
Consequently, we have $\ov\xi_i\equiv\xi_i(i=1,\cdots,10)$, $\ov{\widetilde\eta}_i\equiv\widetilde\eta_i(i=1,2)$ and $\ov{\ov\eta}\equiv\ov\eta$, so that $(\ov x_0,\ov y_0)\equiv(x_0,y_0)$, $(\ov t_0,\ov\alpha_0)\equiv(t_0,\alpha_0)$, $(\ov\tau_0,\ov\beta_0)\equiv(\tau_0,\beta_0)$, $\ov{\ov\omega}^{f,t}_{i}\equiv\ov\omega^{f,t}_{i}(i=1,2,3)$ and $\ov{\widetilde\omega}^{f,t}_{i}\equiv\widetilde\omega^{f,t}_{i}(i=1,2,3,4)$.
Finally, by comparing the definition \eqref{hat1}-\eqref{solution24} of $\ov{\vec\phi}_f$ with the definition \eqref{for1}-\eqref{so3} of $\vec\phi_f$, we obtain 
$\ov{\vec\phi}_f\equiv\vec\phi_f$. This gives us the uniqueness of the weak solution for small time.
\end{proof}

\section{Proof of the existence of a global solution to the Cauchy problem}
Let us now prove the existence of global solution to Cauchy problem \eqref{eq}-\eqref{bcy1}.
Noting the definition \eqref{GGG} of $G_f$ and the definition \eqref{for1}-\eqref{so3} of the local solution $\vec\phi_f$,
it is easy to check that the following two estimates hold for all $t\in[0,\delta]$
\be \label{wt-bound-loc}
0\leq M(t)\leq K,
\ee
\be \label{phi-bound-loc}
\|\vec\phi_f(t,\cdot)\|_{L^{\infty}((0,1)^2)}\!\!\leq\!\! e^{2(N+1)TK_1}\!\!\max_k\!\!\Big\{\!(\f{2K_1}{K_2}+\f{a_1K_1}{\tau_{gf}})^N\|\ov\phi^f_{k0}\|,
(\f{2K_1}{K_2}+\f{2\tau_{gf}}{a_1K_2})^N\|\widehat\phi^f_{k0}\| ,\|\widetilde\phi^f_{k0}\|\!\Big\},
\ee
where $K$ is defined by \eqref{K}, $K_1$ is defined by \eqref{K1} and $K_2$ is defined by \eqref{K2}.
In order to obtain a global solution, we suppose that we have solved Cauchy problem \eqref{eq}-\eqref{bcy1} 
up to the moment $\tau \in[0,T]$ with
the weak solution $\vec{\phi}_f\in C^0([0,\tau];L^p((0,1)^2))$. It follows from our way to construct the weak solution, we
know that for any $0\leq t\leq\tau$, the weak solution is given by for $k=1$
\be \label{gso1}
\ov\phi^f_1(t,x,y):=\left\{
\begin{array}{l}
\ov\phi^f_{10}(x_0, y_0)e^{-\int_0^t[\ov\lambda(y_2,U)+\f{\pa\ov h_f}{\pa y}(y_2,u_f)](\sigma)\, d\sigma},\quad \text{if}\ (x,y)\in \ov\omega^{f,t}_1,\\
0,\quad \text{else}.
\end{array}
\right.
\ee
For $k=2,\cdots,N$
\be \label{gso2}
\ov\phi^f_k(t,x,y):=\!\!\left\{
\begin{array}{l}
\ov\phi^f_{k0}(x_0, y_0)e^{-\int_0^t[\ov\lambda(y_2,U)+\f{\pa\ov h_f}{\pa y}(y_2,u_f)](\sigma)\, d\sigma},\quad \text{if}\ (x,y)\in \ov\omega^{f,t}_1,\\
\displaystyle\f{2\tau_{gf}\widehat\phi^f_{k-1}(\tau_0,1,\beta_0)}{a_1\ov g_f(u_f(\tau_0))}\ e^{-\int_{\tau_0}^t[\ov\lambda(y_3,U)+\f{\pa\ov h_f}{\pa y}(y_3,u_f)](\sigma)\,d\sigma},\quad \text{if}\ (x,y)\in \ov\omega^{f,t}_2,\\
0,\quad \text{else}.
\end{array}
\right.
\ee
For $k=1,\cdots,N$
\be \label{gso3}
\widehat\phi^f_k(t,x,y):=\left\{
\begin{array}{l}
\widehat\phi^f_{k0}(x-\widehat g_ft,y),\quad \text{if}\ 0\leq t\leq\displaystyle\f{x}{\widehat g_f},\ y\in[0,1],\\
\displaystyle\f{a_1\ov g_f(u_f(t-\f{x}{\widehat g_f}))}{\tau_{gf}}\ \ov\phi^f_k(t-\f{x}{\widehat g_f},1,y),\quad \text{if}\ \f{x}{\widehat g_f}\leq t\leq T,\ y\in[0,1].
\end{array}
\right.
\ee
\be \label{gso4}
\widetilde\phi^f_1(t,x,y):=\left\{
\begin{array}{l}
\widetilde\phi^f_{10}(x_0,y_0)e^{-\int_0^t[\widetilde\lambda(y_5,U)
+\f{\pa\widetilde h_f}{\pa y}(y_5,u_f)](\sigma)\, d\sigma},\quad \text{if}\ (x,y)\in \widetilde\omega^{f,t}_1,\\
\ov\phi^f_1(t_0,\displaystyle\f{a_2}{a_1}\alpha_0, 1)e^{-\int_{t_0}^t[\widetilde\lambda(y_6,U)
+\f{\pa\widetilde h_f}{\pa y}(y_6,u_f)](\sigma)\, d\sigma},\quad \text{if}\ (x,y)\in \widetilde\omega^{f,t}_2,\\
0,\quad \text{else}.
\end{array}
\right.
\ee
For $k=2,\cdots,N$
\be \label{gso5}
\widetilde\phi^f_k(t,x,y):=\left\{
\begin{array}{l}
\widetilde\phi^f_{k0}(x_0,y_0)e^{-\int_0^t[\widetilde\lambda(y_5,U)
+\f{\pa\widetilde h_f}{\pa y}(y_5,u_f)](\sigma)\, d\sigma},\quad \text{if}\ (x,y)\in \widetilde\omega^{f,t}_1,\\
\ov\phi^f_k(t_0,\displaystyle\f{a_2}{a_1}\alpha_0, 1)e^{-\int_{t_0}^t[\widetilde\lambda(y_6,U)
+\f{\pa\widetilde h_f}{\pa y}(y_6,u_f)](\sigma)\, d\sigma},\quad \text{if}\ (x,y)\in \widetilde\omega^{f,t}_2,\\
\widetilde\phi^f_{k-1}(\tau_0, 1,\beta_0)e^{-\int_{\tau_0}^t[\widetilde \lambda(y_9,U)+\f{\pa\widetilde h_f}{\pa y}(y_9,u_f)](\sigma)\,d\sigma},\quad 
\text{if}\ (x,y)\in \widetilde\omega^{f,t}_3,\\
0,\quad \text{else}.
\end{array}
\right.
\ee

As time increases, for each cell cycle $k\ (k=1,\cdots,N)$ in Phase 1, the characteristic passing through the origin may have two possible cases.
It may either intersect with the front face (see Fig 6 (a)) or intersect with the right face (see Fig 6 (b)).
For each cell cycle $k\ (k=1,\cdots,N)$ in Phase 3, the characteristic passing through the origin will definitely intersect with the front face due to the fact that $\widetilde h_f(1,u_f)<0$ (see Fig 6 (c)).
We get in any case that
\begin{align*}
M_f(t):=&\sum_{k=1}^N\iint\limits_{\ov\omega^{f,t}_2}a_1\gamma^2_s y\ov\phi^f_k(t,x,y)dxdy
+\sum_{k=1}^N\int_0^1\int_0^1 (a_2-a_1)\gamma^2_s y\widehat\phi^f_k(t,x,y)dxdy\\
&+\!\!\sum_{k=1}^N\!\iint\limits_{\widetilde\omega^{f,t}_2}a_2\gamma_0(\gamma_0 y\!\!+\!\!\gamma_s)\widetilde\phi^f_k(t,x,y)dxdy
\!\!+\!\!\sum_{k=1}^N\!\iint\limits_{\widetilde\omega^{f,t}_3}a_2\gamma_0(\gamma_0 y\!\!+\!\!\gamma_s)\widetilde\phi^f_k(t,x,y)dxdy
\end{align*}

\begin{figure}[htbp!]
\subfigure[]{
\begin{minipage}[b]{0.32\textwidth}
\centering
\includegraphics[width=1.8in]{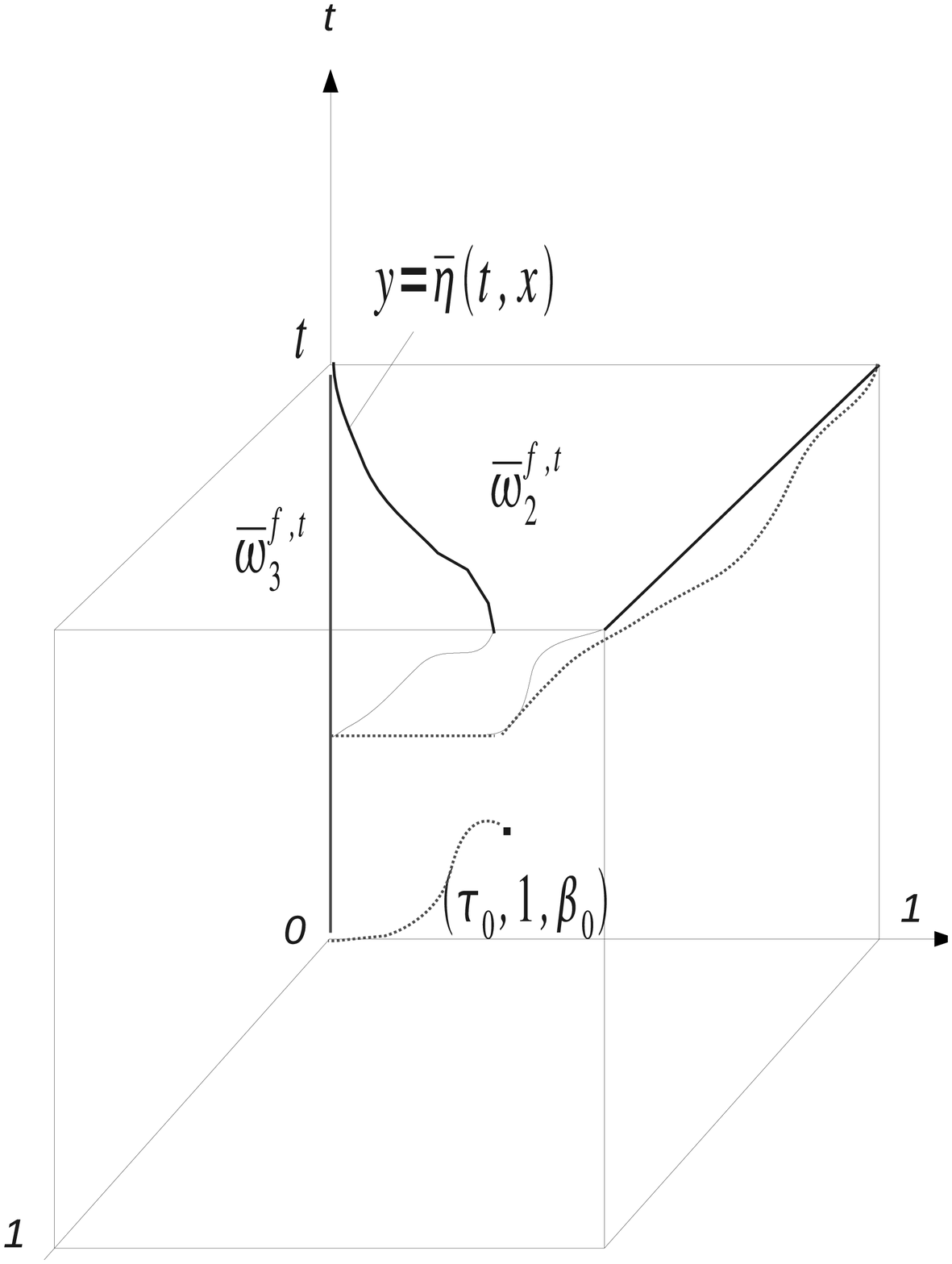}
\end{minipage}}%
\subfigure[]{
\begin{minipage}[b]{0.32\textwidth}
\centering
\includegraphics[width=1.8in]{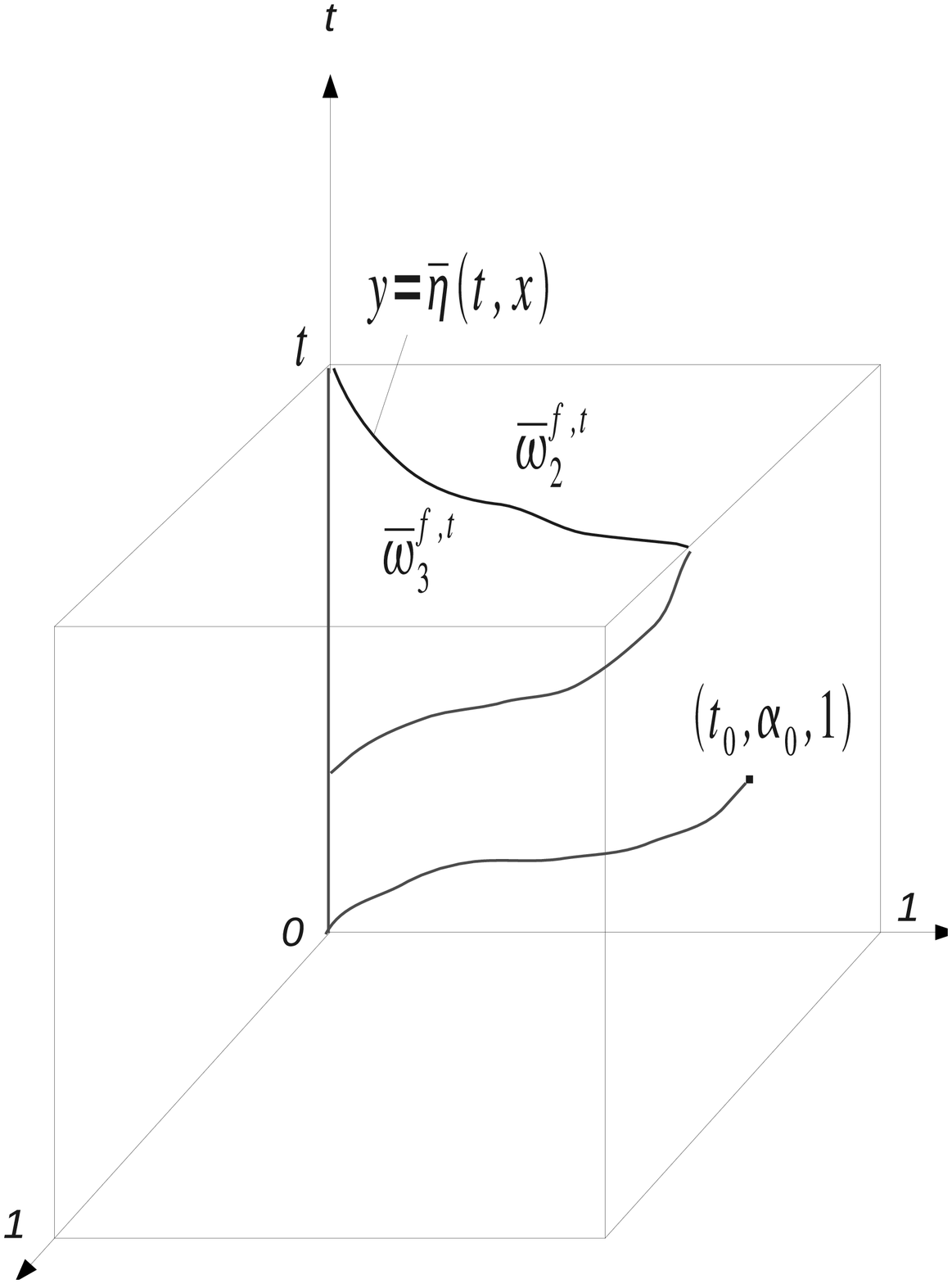}
\end{minipage}}%
\subfigure[]{
\begin{minipage}[b]{0.32\textwidth}
\centering
\includegraphics[width=1.8in]{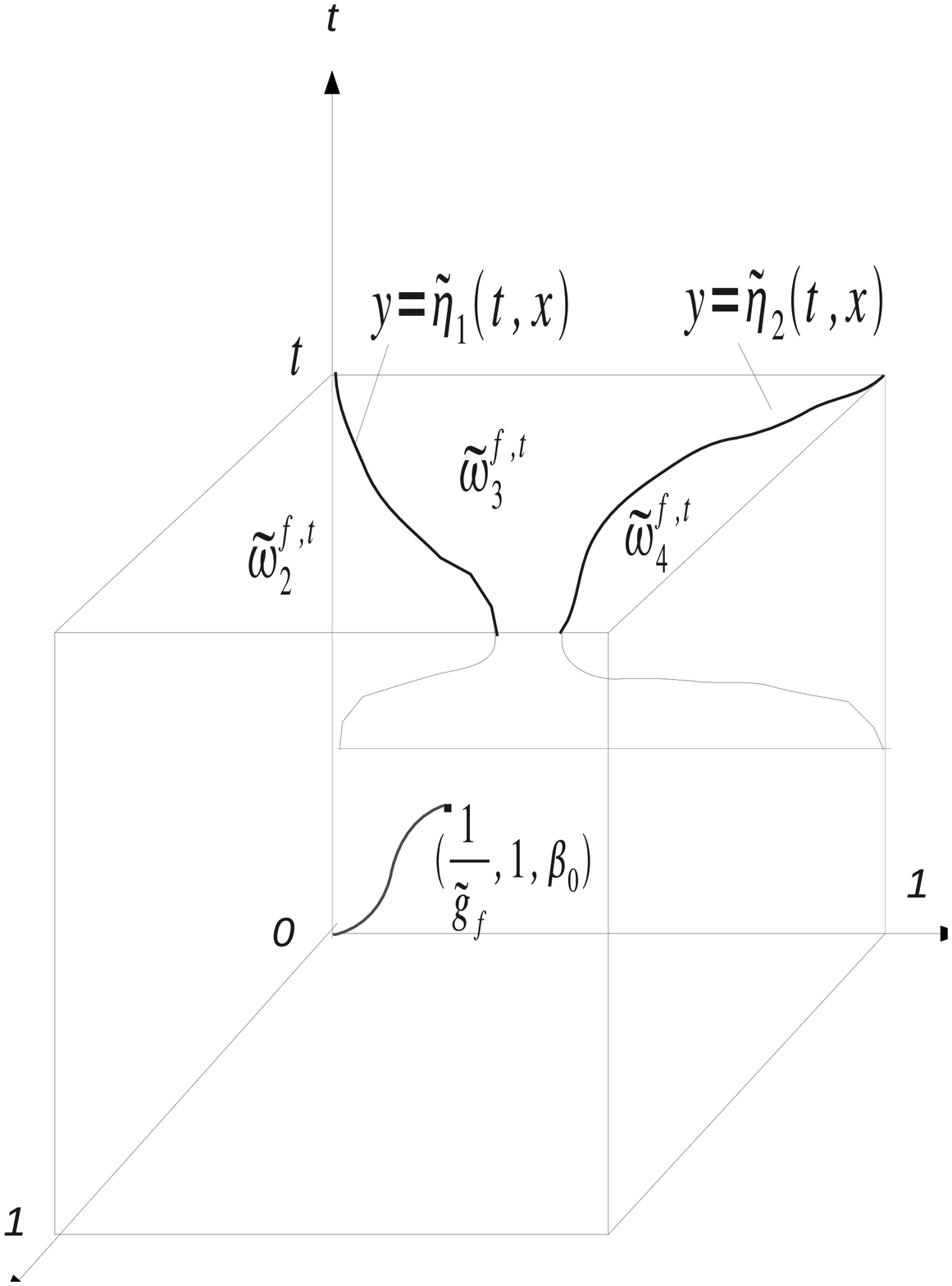}
\end{minipage}}%
\caption{For $t$ large enough. Case (a): the characteristic passing through the origin intersects the front face at $(\tau_0,1,\beta_0)$ in Phase 1;
Case (b): the characteristic passing through the origin intersects the right face at $(t_0,\alpha_0,1)$ in Phase 1;
Case (c): due to the fact that $\widetilde h_f(1,u_f)<0$ in Phase 3, the characteristic passing through the origin will definitely intersect the front face at $(\f{1}{\widetilde g_f},1,\beta_0)$. }
\label{Fig 5} 
\end{figure}

We can prove that estimate \eqref{wt-bound-loc} holds for every $t\in[0,T]$ by 
tracing back $\ov\phi^f_k(t,x,y)$, $\widehat\phi^f_k(t,x,y)$ and 
$\widetilde\phi^f_k(t,x,y)$ along the characteristics to the initial data at most $N$ times.
Moreover, noting the definition \eqref{gso1}-\eqref{gso5} of the global solution, it is easy to check 
that the uniform a priori estimate
\eqref{phi-bound-loc} holds for every $t\in [0,T]$. Hence we can
choose $\delta\in[0,T]$ independent of $\tau$. Applying Lemma \ref{ex-sol} and Lemma \ref{uni}
again, the weak solution $\vec{\phi}_f\in C^0([0,\tau];
L^p((0,1)^2))$ is extended to the time interval $[\tau,\tau+\delta] \cap
[\tau,T]$. Step by step, we finally obtain a unique global weak
solution $\vec{\phi}_f\in C^0([0,T];L^p((0,1)^2))$.
This finishes the proof of the existence of a global solution to 
Cauchy problem \eqref{eq}-\eqref{bcy1}.

\section{Appendix}
\subsection{Introduction of the model}
In this section, we first recall the systems of equations describing the dynamics of the cell density of a follicle in the model of F. Cl\'{e}ment \cite{F08, FC05, FC07}.
The cell population in a follicle $f$ is represented by cell density functions $\phi^f_{j,k}(t,a,\gamma)$
defined on each cellular phase $Q^f_{j,k}$ with age $a$ and maturity $\gamma$,
which satisfy the following conservation laws
\be \label{eq1}
\f{\pa\phi^f_{j,k}}{\pa t}+\f{\pa(g_f(u_f)\phi^f_{j,k})}{\pa a}+\f{\pa(h_f(\gamma, u_f)\phi^f_{j,k})}{\pa\gamma}=-\lambda(\gamma, U)\phi^f_{j,k},\quad 
\text{in}\ {Q^f_{j,k}}
\ee
\begin{align*}
&Q^f_{j,k}:=\Omega_{j,k}\times[0,T],\ \text{for}\ j=1,2,3,\ f=1,\cdots, n, \text{where}\\
&\Omega_{1,k}:=[(k-1)a_2,(k-1)a_2+a_1]\times[0,\gamma_s],\Omega_{2,k}:=[(k-1)a_2+a_1,ka_2]\times[0,\gamma_s],\\
&\Omega_{3,k}:=[(k-1)a_2,ka_2].
\end{align*}
Here $k=1,\cdots,N$, and $N$ is the number of consecutive cell cycles (see Fig 7).

Define the maturity operator $M$ as 
\be \label{M}
M(\varphi)(t):=\int_0^{\gamma_{max}}\int_0^{a_{max}}\gamma\varphi(t,a,\gamma)\,dad\gamma.
\ee
Then
\be\label{Mf}
M_f:=\sum_{j=1}^3\sum_{k=1}^N\int_0^{\gamma_{max}}\int_0^{a_{max}}\gamma\phi^f_{j,k}(t,a,\gamma)\,dad\gamma
\ee
is the global follicular maturity on the follicular scale, while
\be \label{MM}
M:=\sum_{f=1}^{n}\sum_{j=1}^3\sum_{k=1}^N\int_0^{\gamma_{max}}\int_0^{a_{max}}\gamma\phi^f_{j,k}(t,a,\gamma)\,dad\gamma
\ee
is the global maturity on the ovarian scale.

The velocity and source terms are all smooth functions of $u_f$ and $U$, where $u_f$ is the local control
and $U$ is the global control. In this paper, we consider close/open loop problem, that is to say $u_f$ and $U$ are functions
of $M_f, M$ and $t$. As an instance of close loop problem, the velocity and source terms are
given by the following expressions in \cite{F08} (all parameters are positive constants)
\begin{align*}
&g_f(u_f)=\tau_{gf}(1-g_1(1-u_f)), \quad \text{in}\quad {\Omega_{1,k}},\\
&g_f(u_f)=\tau_{gf},\quad \text{in} \quad {\Omega_{j,k}},\quad j=2,3,\\
&h_f(\gamma,u_f)=\tau_{hf}(-\gamma^2+(c_1\gamma+c_2)(1-e^{\f{-u_f}{\ov u}})),\quad \text{in} \quad {\Omega_{j,k}},\quad j=1,3,\\
&h_f(\gamma,u_f)=\lambda(\gamma,U)=0,\quad \text{in}\quad {\Omega_{2,k}},\\
&\lambda(\gamma,U)=Ke^{-(\f{\gamma-\gamma_s}{\ov\gamma})^2}(1-U),\quad \text{in}\quad {\Omega_{j,k}},\quad j=1,3.\\
&U=S(M)+U_0=U_0+U_s+ \frac{1-U_s}{1+e^{c(M-m)}},\\
&u_f=b(M_f)U=\min\{b_1+\f{e^{b_2M_f}}{b_3},1\}\cdot[U_0+U_s+ \frac{1-U_s}{1+e^{c(M-m)}}].
\end{align*}
The initial conditions are given as follows
\be \label{IC}
\phi^f_{j,k}(0,a,\gamma)=\phi^f_{k0}(a,\gamma)|_{\Omega_{j,k}},\qquad j=1,2,3.
\ee
The boundary conditions are given as follows
\begin{align} \label{b}
g_f(u_f)\phi^f_{1,k}(t,(k-1)a_2,\gamma)&=\begin{cases}
2\tau_{gf}\phi^f_{2,k-1}(t,(k-1)a_2,\gamma),\ \text{for}\ k\geq 2,\\
0,\quad \text{for}\ k=1,\end{cases}
\text{for}\ \gamma\!\in\!\ [0,\gamma_s].\nonumber\\
\phi^f_{1,k}(t,a,0)&=0,\quad \text{for}\ a\in\ [(k-1)a_2,(k-1)a_2+a_1],\nonumber\\
\tau_{gf}\phi^f_{2,k}(t,(k-1)a_2+a_1,\gamma)&=g_f(u_f)\phi^f_{1,k}(t,(k-1)a_2+a_1,\gamma),\quad \text{for}\ \gamma\!\in\!\ [0,\gamma_s],\nonumber\\
\phi^f_{2,k}(t,a,0)&=0,\quad \text{for}\ a\in\ [(k-1)a_2+a_1,ka_2],\nonumber\\
\phi^f_{3,k}(t,(k-1)a_2,\gamma)&=\begin{cases}
\phi^f_{3,k-1}(t,(k-1)a_2,\gamma),\quad \text{for}\ k \geq 2,\\
0,\quad \text{for}\ k=1,
\end{cases}
\text{for}\ \gamma\!\in\!\ [\gamma_s,\gamma_m],\nonumber\\
\phi^f_{3,k}(t,a,\gamma_s)&=\begin{cases}
\phi^f_{1,k}(t,a,\gamma_s),\quad \text{for}\ a\in\ [(k-1)a_2,(k-1)a_2+a_1],\nonumber\\
0,\quad \text{for}\ a\in\ [(k-1)a_2+a_1,ka_2].
\end{cases}
\end{align}

\vspace{20mm}
\setlength{\unitlength}{0.085in}
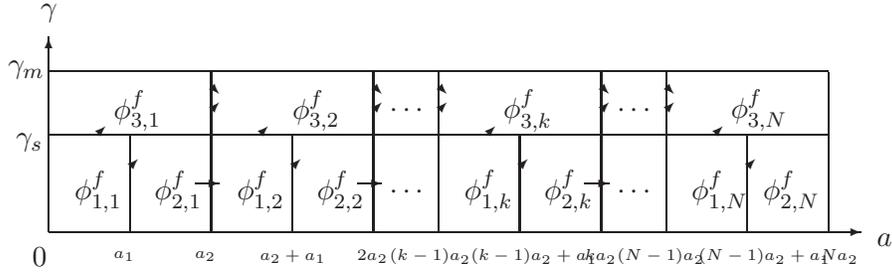
\begin{figure}[htbp!]
\begin{picture}(-10,10)

\put(7,1){\vector(1,0){50}} \put(7,1){\vector(0,1){12}}

\put(7,10.9){\line(1,0){48}} \put(7,7){\line(1,0){48}}
\put(12,1){\line(0,1){5.9}} \put(17,1){\line(0,1){9.8}}
\put(22,1){\line(0,1){5.9}} \put(27,1){\line(0,1){9.8}}
 \put(31,1){\line(0,1){9.8}}

\put(28,3){\makebox(2,1)[l]{$\cdots$}}
\put(28,8){\makebox(2,1)[l]{$\cdots$}}

 \put(36,1){\line(0,1){5.9}}
\put(41,1){\line(0,1){9.8}} \put(45,1){\line(0,1){9.8}}
\put(50,1){\line(0,1){5.9}} \put(55,1){\line(0,1){9.8}}

\put(5,6){\makebox(2,1)[l]{$\gamma_s$}}
\put(4.5,10.5){\makebox(2,1)[l]{$\gamma_m$}}
\put(6.5,14){\makebox(2,1)[l]{$\gamma$}}

\put(58,0){\makebox(2,1)[l]{$a$}}
\put(6,-1){\makebox(2,1)[l]{0}}
\put(11,-1){\makebox(2,1)[l]{\tiny$a_1$}}
\put(16,-1){\makebox(2,1)[l]{\tiny$a_2$}}
\put(20,-1){\makebox(2,1)[l]{\tiny$a_2+a_1$}}
\put(26,-1){\makebox(2,1)[l]{\tiny$2a_2$}}
\put(28,-1){\makebox(2,1)[l]{\tiny $(k-1)a_2$}}
\put(33,-1){\makebox(2,1)[l]{\tiny $(k-1)a_2+a_1$}}
\put(40,-1){\makebox(2,1)[l]{\tiny $ka_2$}}
\put(42,-1){\makebox(2,1)[l]{\tiny $(N-1)a_2$}}
\put(47,-1){\makebox(2,1)[l]{\tiny $(N-1)a_2+a_1$}}
\put(54.5,-1){\makebox(2,1)[l]{\tiny $Na_2$}}

\put(8,3){\makebox(2,1)[l]{ $\phi^f_{1,1}$}}
\put(13.5,3){\makebox(2,1)[l]{$\phi^f_{2,1}$}}
\put(11,8){\makebox(2,1)[l]{$\phi^f_{3,1}$}}

\put(18,3){\makebox(2,1)[l]{ $\phi^f_{1,2}$}}
\put(23.5,3){\makebox(2,1)[l]{$\phi^f_{2,2}$}}
\put(22,8){\makebox(2,1)[l]{$\phi^f_{3,2}$}}

\put(32,3){\makebox(2,1)[l]{ $\phi^f_{1,k}$}}
\put(37.5,3){\makebox(2,1)[l]{$\phi^f_{2,k}$}}
\put(35,8){\makebox(2,1)[l]{$\phi^f_{3,k}$}}

\put(46,3){\makebox(2,1)[l]{ $\phi^f_{1,N}$}}
\put(51,3){\makebox(2,1)[l]{$\phi^f_{2,N}$}}
\put(49,8){\makebox(2,1)[l]{$\phi^f_{3,N}$}}

\put(42,3){\makebox(2,1)[l]{$\cdots$}}
\put(42,8){\makebox(2,1)[l]{$\cdots$}}

\put(16,7.5){\vector(1,1){1.5}} \put(16,11){\vector(1,-1){1.5}} 
\put(26,7.5){\vector(1,1){1.5}} \put(26,11){\vector(1,-1){1.5}}  
\put(30,7.5){\vector(1,1){1.5}} \put(30,11){\vector(1,-1){1.5}} 
\put(40,7.5){\vector(1,1){1.5}} \put(40,11){\vector(1,-1){1.5}}
\put(44,7.5){\vector(1,1){1.5}} \put(44,11){\vector(1,-1){1.5}}

\put(9,6){\vector(1,1){1.5}} 
\put(11,4){\vector(1,1){1.5}}
\put(16,4){\vector(1,0){1.5}}
\put(19,6){\vector(1,1){1.5}} 
\put(21,4){\vector(1,1){1.5}}
\put(26,4){\vector(1,0){1.5}}
\put(33,6){\vector(1,1){1.5}} 
\put(35,4){\vector(1,1){1.5}}
\put(40,4){\vector(1,0){1.5}}
\put(47,6){\vector(1,1){1.5}} 
\put(49,4){\vector(1,1){1.5}}
\end{picture}
\caption{The illustration of $N$ cell cycles for follicle $f$; $a$ represents the age of the cell and $\gamma$ represents the maturity of the cell.
The top of the domain corresponds to the differentiation phase and the bottom to the the proliferation phase.}
\end{figure}
For sake of simplicity, we denote by $\gamma_0:=\gamma_m-\gamma_s>0$ in the whole paper the difference between the 
maximum $\gamma_s$ of the maturity in Phase 1 and the 
maximum $\gamma_m$ of the maturity in Phase 3.
\subsection{Mathematical reformulation}
In this section, we perform a mathematical reformulation of the original model introduced in Appendix 6.1.
Notice that in system \eqref{eq1}, all the unknowns are defined on
different domains, so that one has to solve the equations successively. Here
we transform system \eqref{eq1} into a regular one where the unknowns are
defined on the same domain $[0,T]\times [0,1]^2$. 
We denote by $\ov\phi^f_k$ the density functions and by $\ov g_f$ and $\ov h_f$ the age and 
maturity velocities for Phase 1; $\widehat\phi^f_k$ the density functions 
and $\widehat g_f$ the age velocities for Phase 2; $\widetilde\phi^f_k$ the density functions and $\widetilde g_f$, $\widetilde h_f$ the age and 
maturity velocities for Phase 3.

Let
\be 
\ov\phi^f_k(t,x,y):=\phi^f_{1,k}(t,a,\gamma),\quad
(a,\gamma)\in \Omega_{1,k},\ee
where %
\be
x:=\f {a-(k-1)a_2}{a_1},\ y:=\f{\gamma}{\gamma_s},
\ee
so that we get $(a,\gamma)\in \Omega_{1,k} \Longleftrightarrow (x,y)\in [0,1]^2$, and
\be 
(\ov\phi^f_k)_t+(\ov g_f\ov\phi^f_k)_x+(\ov h_f\ov\phi^f_k)_y= -\ov\lambda\ \ov\phi^f_k,\ee
with
\be 
\ov g_f(u_f):=\f{g_f(u_f)}{a_1},\quad \ov h_f(y,u_f):=\f{h_f(\gamma_s y,u_f)}{\gamma_s},\quad
\ov\lambda(y,U):=\lambda(\gamma_s y,U).\ee
Let
\be 
\widehat\phi^f_k(t,x,y):=\phi^f_{2,k}(t,a,\gamma),\quad (a,\gamma)\in \Omega_{2,k},
\ee
where
\be
x:=\f{a-(k-1)a_2-a_1}{a_2-a_1},\ y:=\f{\gamma}{\gamma_s}.
\ee
So that we get $(a,\gamma)\in \Omega_{2,k}\Longleftrightarrow (x,y)\in [0,1]^2$, and 
\be 
(\widehat\phi^f_k)_t+\widehat g_f(\widehat\phi^f_k)_x=0,\quad \widehat g_f:=\f{\tau_{gf}}{a_2-a_1}.
\ee
Let
\be \widetilde\phi^f_k(t,x,y):=\phi^f_{3,k}(t,a,\gamma),\quad
(a,\gamma)\in \Omega_{3,k},\ee
where%
\be
x:=\f {a-(k-1)a_2}{a_2},\ y:=\f{\gamma-\gamma_s}{\gamma_0},\ee
so that we get $(a,\gamma)\in \Omega_{3,k}\Longleftrightarrow (x,y)\in [0,1]^2$, and
\be (\widetilde\phi^f_k)_t+(\widetilde g_f\widetilde\phi^f_k)_x+(\widetilde h_f\widetilde\phi^f_k)_y=
-\widetilde\lambda\widetilde\phi^f_k,\ee
with
\be 
\widetilde g_f:=\f{\tau_{gf}}{a_2},\ \widetilde h_f(y,u_f):=\f{h_f(\gamma_0 y+\gamma_s
,u_f)}{\gamma_0},\ \widetilde\lambda(y,U):=\lambda(\gamma_0 y+\gamma_s,U).
\ee
Accordingly, let us denote the initial conditions \eqref{IC} with new notations by 
\be 
\vec\phi_{f0}=(\ov\phi^f_{10}(x,y),\cdots,\ov\phi^f_{N0}(x,y),\ \widehat\phi^f_{10}(x,y),\cdots\widehat\phi^f_{N0},\
 \widetilde\phi^f_{10}(x,y),\cdots,\widetilde\phi^f_{N0}(x,y)),
\ee
where
\be 
\ov\phi^f_{k0}(x,y):=\phi^f_{1,k}(0,a,\gamma),\quad \widehat\phi^f_{k0}(x,y):=\phi^f_{2,k}(0,a,\gamma),\quad \widetilde\phi^f_{k0}(x,y):=\phi^f_{3,k}(0,a,\gamma).
\ee
Let $\vec\phi_f=(\ov\phi^f_1,\cdots,\ov\phi^f_N,\ \widehat\phi^f_1,\cdots,\widehat\phi^f_N,\ 
\widetilde\phi^f_1,\cdots,\widetilde\phi^f_N)^T$, we have
\be \label{eq2}
{\vec{\phi}_f(t,x,y)}_t+(A_f\vec{\phi}_f(t,x,y))_x +(B_f\vec{\phi}_f(t,x,y))_y= C\vec{\phi}_f(t,x,y),
\ee
\begin{equation*}
t\in[0,T],\quad(x,y)\in {[0,1]^2},
\end{equation*}
where
\begin{align*}
A_f:&=\text{diag}\ \{\overbrace{\ov g_f,\cdots,\ov g_f}^{N},\ \overbrace{\widehat g_f,\cdots,\widehat g_f}^{N}\
\overbrace{\widetilde g_f,\cdots,\widetilde g_f}^{N}\},\\
B_f:&=\text{diag}\ \{\overbrace{\ov h_f,\cdots,\ov h_f}^{N},\ \overbrace{0,\cdots,0}^{N}\
\overbrace{\widetilde h_f,\cdots,\widetilde h_f}^{N}\},\\
C:&=-\text{diag}\ \{\overbrace{\ov\lambda,\cdots,\ov\lambda}^{N},\ \overbrace{0,\cdots,0}^{N}\
\overbrace{\widetilde\lambda,\cdots,\widetilde\lambda}^{N}\}.
\end{align*}
From the original expression of $h_f(\gamma,u_f)$ (see Appendix 6.1), let us define
\begin{equation*} 
\gamma_{\pm}(u_f):=\f{c_1(1-e^{\f{-u_f}{\ov u}})\pm\sqrt{c^2_1(1-e^{\f{-u_f}{\ov u}})^2+4c_2(1-e^{\f{-u_f}{\ov u}})}}{2}.
\end{equation*}
It is easy to see that $\gamma_+(u_f)$ is an increasing function of $u_f$, and
\begin{equation*}
h_f(\gamma,u_f)=\tau_{hf}(\gamma_+(u_f)-\gamma)(\gamma-\gamma_-(u_f)).
\end{equation*}
Hence, when $\gamma=\gamma_+(u_f)$, we have $h_f(\gamma,u_f)=0$, moreover
\begin{align*}
&h_f(\gamma,u_f)>0,\quad \text{if}\ 0\leq\gamma<\gamma_+(u_f),\\
&h_f(\gamma,u_f)<0,\quad \text{if}\ \gamma>\gamma_+(u_f).
\end{align*}
Furthermore, under the assumption that the local control $u_f$ satisfies $\gamma_+(u_f)>\gamma_s$,
and 
\begin{equation*}
0\leq\gamma_+(u_f)\leq\gamma_+(1)\simeq(1-e^{\f{-1}{\ov u}})\displaystyle\f{c_1+\sqrt{c^2_1+4c_2}}{2}<\gamma_m.
\end{equation*}
From the fact that $0<\gamma\leq\gamma_s$ in Phase 1, and $\gamma_s\leq\gamma\leq\gamma_m$ in Phase 3, we have
\begin{align*}
& h_f(\gamma,u_f)>0,\quad 0<\gamma\leq\gamma_s,\\
& h_f(\gamma_s,u_f)>0,\quad h_f(\gamma_m,u_f)<0.
\end{align*}
Corresponding to the new notations, we have assumptions
\begin{align}
&\ov h_f(y,u_f)>0,\quad \forall y\in[0,1],\nonumber\\
&\widetilde h_f(0,u_f)>0,\quad \widetilde h_f(1,u_f)<0.
\end{align}
Since
\begin{align*}
 M(\phi^f_{1,k})&=\iint\limits_{\Omega_{1,k}}
   \gamma\phi^f_{1,k}(t,a,\gamma) \,da\,d\gamma
   =\int_0^1\int_0^1 a_1\gamma_s^2 y\ov\phi^f_k(t,x,y)\,dx\,dy,\\
  M(\phi^f_{2,k})&=\iint\limits_{\Omega_{2,k}}
   \gamma\phi^f_{2,k}(t,a,\gamma) \,da\,d\gamma 
=\int_0^1\int_0^1(a_2-a_1)\gamma^2_s y\widehat\phi^f_k(t,x,y)\,dx\,dy,\\
 M(\phi^f_{3,k})&=\iint\limits_{\Omega_{3,k}}
   \gamma\phi^f_{3,k}(t,a,\gamma) \,da\,d\gamma
   =\int_0^1\int_0^1 a_2\gamma_0(\gamma_0 y +\gamma_s)
   \widetilde\phi^f_k(t,x,y)\,dx\,dy.
 \end{align*}
We have
\begin{align*} M_f(t)=&\sum_{k=1}^N M(\phi^f_{1,k})+\sum_{k=1}^N M(\phi^f_{2,k})+\sum_{k=1}^N M(\phi^f_{3,k})\nonumber\\
=&\sum_{k=1}^N \int_0^1\int_0^1 a_1\gamma_s^2 y
\ov\phi^f_k(t,x,y)dxdy+\sum_{k=1}^N \int_0^1\int_0^1(a_2-a_1)\gamma^2_s y\widehat\phi^f_k(t,x,y)dxdy\\
&+\sum_{k=1}^N \int_0^1\int_0^1 a_2\gamma_0(\gamma_0 y +\gamma_s) \widetilde\phi^f_k(t,x,y) \,dx\,dy.
\end{align*}
\subsection{Basic lemmas and some notations}
The following lemma is used to prove the existence of the
weak solution to Cauchy problem \eqref{eq}-\eqref{bcy1}, when we derive the contraction mapping
function $\vec G$, change 
variables in certain integrals (see Section 3) and prove the uniqueness of the solution (see Section 4). 
\begin{lem}\label{jacobi}
The characteristic $\xi_2=(x_2,y_2)$ passing through $(t,x,y)$ intersects the bottom face at $(0,x_0,y_0)$. We have
\be\label{jy2}
\f{\pa(x,y)}{\pa(x_0,y_0)}=e^{\int_0^t \f{\pa\ov h_f}{\pa y}(y_2,u_f)(\sigma)\, d\sigma}.
\ee
The characteristic $\xi_3=(x_3,y_3)$ passing through $(t,x,y)$ intersects the back face at $(\tau_0,0,\beta_0)$. We have
\be\label{jy3}
\f{\pa(x,y)}{\pa(\tau_0,\beta_0)}=-\ov g_f(u_f(\tau_0))\cdot e^{\int_{\tau_0}^t \f{\pa\ov h_f}{\pa y}(y_3,u_f)(\sigma)\, d\sigma}.
\ee
The characteristic $\xi_4=(x_4,y_4)$ passing through $(\tau_0,1,\beta_0)$ intersects the bottom face at $(0,x_0,y_0)$. We have
\be\label{jy4}
\f{\pa(\tau_0,\beta_0)}{\pa(x_0,y_0)}=\f{-1}{\ov g_f(u_f(\tau_0))}\cdot e^{\int_0^{\tau_0} \f{\pa\ov h_f}{\pa y}(y_4,u_f)(\sigma)\, d\sigma}.
\ee
The characteristic $\xi_5=(x_5,y_5)$ passing through $(t,x,y)$ intersects the bottom face at $(0,x_0,y_0)$. We have
\be\label{jy5}
\f{\pa(x,y)}{\pa(x_0,y_0)}=e^{\int_0^t \f{\pa\widetilde h_f}{\pa y}(y_5,u_f)(\sigma)\, d\sigma}.
\ee
The characteristic $\xi_6=(x_6,y_6)$ passing through $(t,x,y)$ intersects the left face at $(t_0,\alpha_0,0)$. We have
\be\label{jy6}
\f{\pa(x,y)}{\pa(t_0,\alpha_0)}=\widetilde h_f(u_f(t_0),0)\cdot e^{\int_{t_0}^t \f{\pa\widetilde h_f}{\pa y}(y_6,u_f)(\sigma)\, d\sigma}.
\ee
The characteristic $\xi_7=(x_7,y_7)$ passing through $(t_0,\alpha_0,1)$ intersects the bottom face at $(0,x_0,y_0)$. We have
\be\label{jy7}
\f{\pa(\alpha_0, t_0)}{\pa(x_0,y_0)}=\f{-1}{\ov h_f(u_f(t_0),1)}\cdot e^{\int_0^{t_0} \f{\pa\ov h_f}{\pa y}(y_7,u_f)(\sigma)\, d\sigma}.
\ee
The characteristic $\xi_8=(x_8,y_8)$ passing though $(t_0,\alpha_0,1)$ intersects the back face at $(\tau_0,0,\beta_0)$. We have
\be\label{jy8}
\f{\pa(\alpha_0, t_0)}{\pa(\tau_0,\beta_0)}=\f{-\ov g_f(u_f(\tau_0))}{\ov h_f(u_f(t_0),1)}\cdot e^{\int_{\tau_0}^{t_0} \f{\pa\ov h_f}{\pa y}(y_8,u_f)(\sigma)\, d\sigma}.
\ee
The characteristic $\xi_9=(x_9,y_9)$ passing through $(t,x,y)$ intersects the back face at $(\tau_0,0,\beta_0)$. We have
\be\label{jy9}
\f{\pa(x,y)}{\pa(\tau_0,\beta_0)}=-\widetilde g_f e^{\int_{\tau_0}^t\f{\pa\widetilde h_f}{\pa y}(y_9,u_f)(\sigma)\, d\sigma}.
\ee
The characteristic $\xi_{10}=(x_{10},y_{10})$ passing through $(\tau_0,1,\beta_0)$ intersects the bottom face at $(0,x_0,y_0)$. We have
\be\label{jy10}
\f{\pa(\tau_0,\beta_0)}{\pa(x_0,y_0)}=\f{-1}{\widetilde g_f}\cdot e^{\int_0^{\tau_0} \f{\pa\widetilde h_f}{\pa y}(y_{10},u_f)(\sigma)\, d\sigma}.
\ee
\end{lem}
The proof of this lemma is trivial, and can be found in \cite{LiYuBook} of 1-space dimension case, we omit here.

The expressions of contraction mapping coefficients $C_1^f$ and $C_2^f$ are as following
\begin{align}\label{Cf}
C_1^f:=&\f{a_1\gamma^2_m(2K^2_1-tK^2_1+3K_1+tK^2_1K_2+9K_1K_2)}{K_2(1-tK_1)}\sum_{k=1}^{N}\|\ov\phi^f_{k0}\|\\
&+\!\f{2(a_2-a_1)\gamma^2_m(K^2_1\!+\!2tK^2_1K_2\!+\!4K_1K_2)}{K_2(1-tK_1)}\!\sum_{k=2}^{N}\!\|\widehat\phi^f_{k0}\|
\!\!+\!\f{a_2\gamma^2_m(2K_1\!+\!2tK^2_1)}{1-tK_1}\!\sum_{k=1}^{N}\!\|\widetilde\phi^f_{k0}\|,\nonumber\\
C_2^f:=&\f{a_1\gamma^2_m(2K^2_1+2K_1+12K_1K_2-2tK^2_1K_2)}{K_2(1-tK_1)}\sum_{k=1}^{N}\|\ov\phi^f_{k0}\|\nonumber\\
&+\f{2(a_2-a_1)\gamma^2_m(K^2_1+6K_1K_2)}{K_2(1-tK_1)}\sum_{k=2}^{N}\|\widehat\phi^f_{k0}\|
+\f{4a_2\gamma^2_mK_1}{1-tK_1}\|\widetilde\phi^f_{k0}\|\nonumber.
\end{align}

\section*{Acknowledgements}
The author would like to thank the professors Fr\'{e}d\'{e}rique
Cl\'{e}ment, Jean-Michel Coron and Zhiqiang Wang for their interesting comments and
many valuable suggestions on this work.

\end{document}